\documentclass[11pt]{amsart}
\usepackage{amsmath, amscd, amssymb, graphicx}

\numberwithin{equation}{section}
\newtheorem{theorem}{Theorem}[section]

\newtheorem{proposition}[theorem]{Proposition}
\newtheorem{lemma}[theorem]{Lemma}

\newtheorem{conjecture}[theorem]{Conjecture}
\theoremstyle{remark}
\newtheorem{remark}{Remark}[section]
\theoremstyle{definition}

\newtheorem{definition}{Definition}[section]

\begin{document}
\title[Orthogonal Quantum Group Invariants of Links]{Orthogonal Quantum
Group Invariants of Links}
\author{Lin Chen}
\address{L. Chen, Simons Center for Geometry and Physics, State University
of New York, Stony Brook, NY 11794, USA (chenlin@math.sunysb.edu)}
\author{Qingtao Chen}
\address{Q. Chen, Department of Mathematics, University of Southern
California, Los Angeles, CA 90089, USA (qingtaoc@usc.edu)}

\begin{abstract}
We study the Chern-Simons partition function of orthogonal quantum group
invariants, and propose a new orthogonal Labastida-Mari\~{n}o-Ooguri-Vafa
conjecture as well as degree conjecture for free energy associated to the
orthogonal Chern-Simons partition function. We prove the degree conjecture
and some interesting cases of orthogonal LMOV conjecture. In particular, We
provide a formula of colored Kauffman polynomials for torus knots and links,
and applied this formula to verify certain case of the conjecture at roots
of unity except $1$. We also derive formulas of Lickorish-Millett type for
Kauffman polynomials and relate all these to the orthogonal LMOV conjecture.
\end{abstract}

\maketitle

\setcounter{tocdepth}{5} \setcounter{page}{1}

\section{Introduction}

\label{sec1}

\subsection{Overview}

Jones' seminal papers \cite{J1, J2} initiated a new era in knot theory. The
HOMFLY polynomial \cite{HOMFLY} and Kauffman \cite{Kau} polynomial for links
were subsequently discovered. In the 1990's, Witten-Reshetikhin-Turaev
constructed the colored version of these invariants, either by path
integrals in physics \cite{Wit}, or by the representation theory of quantum
groups \cite{RT1, RT2}. These works lead to a unified understanding of
quantum group invariants of links.

The colored HOMFLY polynomials, which are associated to the special linear
quantum groups, have been studied more carefully after physicists proposed a
conjectural relationship between Chern-Simons theory and Gromov-Witten
invariants. The Mari\~{n}o-Vafa formula and the topological vertex \cite%
{AKMV, LLLZ, LLZ1, LLZ2} are examples illustrating this so-called string
duality. The Labastida-Mari\~{n}o-Ooguri-Vafa conjecture \cite{LM, LMV, OV}
gave highly non-trivial relations between colored HOMFLY polynomials. The
first such relation is the classical Lichorish-Millett theorem \cite{LiM}.
The integers coefficients that appear in the LMOV conjecture are called the
BPS numbers in string theory, and also related to the integrality in the
Gopakumar-Vafa conjecture \cite{GV} for Gromov-Witten invariants \cite{Pen}.
By using the cabling technique, Xiao-Song Lin and Hao Zheng \cite{LZ}
obtained a formula for colored HOMFLY polynomials of torus links in terms of
Littlewood-Richardson coefficients, and they were able to check certain
cases of the LMOV conjecture for a few (small) torus knots and links. The
LMOV conjecture was recently proved by Kefeng Liu and Pan Peng \cite{LP},
based on the cabling technique and a careful degree analysis of the cut-join
equations.

Actually the LMOV conjecture is part of a bigger picture, the large $N$
duality, proposed by 't Hooft \cite{tHo} in the 1970's. Large $N$ duality
states that the duality between Chern-Simons gauge theory of $S^{3}$ and
topological string theory on the resolved conifold.

In mathematics, the LMOV conjecture predicts that the reformulated
invariants (some combination) of colored HOMFLY/Kauffman polynomials are in
the ring $\mathbb{Z}[t,t^{-1}][q-q^{-1}]$, where $q$ is the quantum
deformation number. Through this way, these reformulated invariants has the
similar expression as the original HOMFLY/Kauffman polynomials which has
variables $q-q^{-1}$, $t$ and $t^{-1}$.

\subsection{Orthogonal Labastida-Mari\~{n}o-Ooguri-Vafa Conjecture}

The study of colored Kauffman polynomials is more difficult. For instance,
the definition of the Chern-Simons partition function for the orthogonal
quantum groups involves representations of the Brauer centralizer algebras,
which admit more complicated orthogonal relations (see \cite{Ram1, Ram2,
Ram3}). The orthogonal analog of cut-join equation \cite{LLZ1, LP} can be
found in \cite{Che1}.

In this paper, we propose a new conjecture on the reformulated invariants,
developed in collaboration with Nicolai Reshetikhin, which is the orthogonal
quantum group analog of the original LMOV conjecture. Let $\mathcal{L}$ be a
link with $L$ components and let $Z_{CS}^{SO}(\mathcal{L},q,t)$ be the
orthogonal Chern-Simons partition function defined in Section \ref{sec4}.
Expand the free energy
\begin{equation*}
F^{SO}(\mathcal{L},q,t)=\log Z_{CS}^{SO}(\mathcal{L},q,t)=\sum_{%
\overrightarrow{\mu }\neq \overrightarrow{0}}F_{\overrightarrow{\mu }%
}^{SO}pb_{\overrightarrow{\mu }}(\overrightarrow{z})\text{.}
\end{equation*}%
Then the reformulated invariants are defined by
\begin{equation*}
g_{\overrightarrow{\mu }}(q,t)=\sum_{k|\overrightarrow{\mu }}\frac{\mu (k)}{k%
}F_{\overrightarrow{\mu }/k}^{SO}(q^{k},t^{k}).
\end{equation*}%
We conjecture that

\begin{conjecture}[Orthogonal LMOV]
\begin{equation*}
\frac{\mathrm{z}_{\overrightarrow{\mu }}(q-q^{-1})^{2}\cdot \lbrack g_{%
\overrightarrow{\mu }}(q,t)-g_{\overrightarrow{\mu }}(q,-t)]}{2\overset{L}{%
\underset{\alpha =1}{\prod }}\overset{\ell (\mu ^{\alpha })}{\underset{i=1}{%
\prod }}(q^{\mu _{i}^{\alpha }}-q^{-\mu _{i}^{\alpha }})}\in \mathbb{Z}%
[q-q^{-1}][t,t^{-1}].
\end{equation*}
\end{conjecture}

and

\begin{conjecture}[Degree]
\begin{equation*}
val_{u}(F_{\overrightarrow{\mu }}^{SO})\geq \ell (\overrightarrow{\mu })-2,
\end{equation*}%
where $q=e^{u}$, $val_{u}(F_{\overrightarrow{\mu }}^{SO})$ is the valuation
of the variable $u$ and $\ell (\overrightarrow{\mu })$ is the sum of the
lengths of the partition corresponding to each component of the link $%
\mathcal{L}$.
\end{conjecture}

This conjecture is a mathematical formulation of the conjecture made by
Bouchard-Florea-Mari\~{n}o \cite{BFM}, and the integer coefficients on the
right hand side of the above conjecture is closely related to BPS numbers in
string theory \cite{BFM}. More recent progress can be found in \cite{Mar},
which is a refined version of \cite{BFM}. The framing version can be found
in \cite{BR, PBR}. Our formulation is still quite different from that in
\cite{BFM, Mar}. The reason for this is that \cite{BFM, Mar} uses
representations of Hecke algebra, whereas our approach is based on
representations of the Birman-Murakami-Wenzl algebra, and uses type-B Schur
function instead of type-A Schur function as the basis in the orthogonal
Chern-Simons partition function.

Theorems that partly answer the orthogonal LMOV conjecture proposed in this
paper are listed below. For more precise statements of these theorems, see
Sections \ref{sec5}, \ref{sec7}, \ref{sec8} and \ref{sec9}.

\begin{theorem}
\label{Thm1.3}The conjecture is true for all partitions when the link is
trivial (namely is a disjoint union of unlinked unknots).
\end{theorem}

\begin{theorem}
\label{Thm1.4}The conjecture is true for partitions of the shape $%
\overrightarrow{\mu }=((1^{d_{1}}),(1^{d_{2}}),\cdots ,(1^{d_{L}}))$, where $%
(1^{d_{\alpha }})=(1,1,...,1)$ $\vdash d_{\alpha }$ for $1\leq \alpha \leq L$%
.
\end{theorem}

\begin{theorem}
\label{Thm1.5}The conjecture is true if and only if it is true for
partitions of the shape $\overrightarrow{\mu }=((d_{1}),(d_{2}),\cdots
,(d_{L}))$.
\end{theorem}

\begin{theorem}
\label{Thm1.6}The conjecture asymptotically holds (for all partitions $%
\overrightarrow{\mu }$ and all knots/links) as $q$ tends to $1$.
\end{theorem}

\begin{theorem}
\label{Thm1.7}The conjecture is true when $\mathcal{L}$ is: the torus
knots/links $T(2,k)$, where $k$ is odd/even, and each component of the
partition $\overrightarrow{\mu }$ is of the form $(1)$, $(1,1)$ or $(2)$;
the two components torus link $T(2,2k)$ for partition $(3),(1)$; the three
components torus link $T(3,3k)$ for the partition $(2),(1),(1)$. These
examples give evidence for the conjecture at non-trivial roots of unity.
\end{theorem}

We also prove the degree conjecture.

\begin{theorem}
\label{Thm1.8}The following degree estimation holds
\end{theorem}

\begin{equation*}
val_{u}(F_{\overrightarrow{\mu }}^{SO})\geq \ell (\overrightarrow{\mu })-2%
\text{.}
\end{equation*}

In addition, we use the cabling technique developed in \cite{LZ} to
calculate colored Kauffman polynomials for torus knots and links, which are
employed to test the orthogonal LMOV conjecture (Theorem \ref{Thm1.7}).

This paper is organized as follows: In Section 2, we review some basic
knowledge of partitions, the Birman-Murakami-Wenzl algebra and irreducible
representation of the Brauer algebra. In Section 3, we review the definition
of the quantum group invariants of links and use the cabling formula to
simplify the computation of these invariants. As an application of the
cabling formula, we obtain colored Kauffman polynomials of all torus knots
and links for all partitions (irreducible representations). In Section 4, we
define the Chern-Simons partition function for orthogonal quantum groups and
the corresponding reformulated invariants. Also, we compute the orthogonal
Chern-Simons partition function for disjoint union of unknots (Theorem \ref%
{Thm1.3}). In Section 5, we propose a new orthogonal LMOV conjecture and
degree conjecture. Then we test torus knots and links as supporting examples
(Theorem \ref{Thm1.7}), which can not be treated as special cases of the
proof in the following sections. In Section 6, we obtain formulas of
Lickorish-Millett type by using skein relations at the intersections of two
different link components. This trick is also widely used in Section 7.
Anyway, this section is quite independent and such Lickorish-Millett type
formulas can also be treated as an application of the orthogonal LMOV
conjecture, which is the starting point of this paper. In Section 7, we
prove the equivalence between the vanishing of the first three coefficients
of $F_{\overrightarrow{\mu }}$ for trivial partitions $\overrightarrow{\mu }$
(each component of partitions have only one box), predicted by the degree
conjecture, and the Lichorish-Millett type formulas obtained in Section 6.
We also prove the orthogonal LMOV conjecture for column-like Young diagram
(Theorem \ref{Thm1.4}) as a generalization of such Lichorish-Millett type
formulas. In Section 8 and 9, we prove that if the orthogonal LMOV
conjecture is valid for the case of rows, then the orthogonal LMOV is valid
for all partitions (Theorem \ref{Thm1.5}). In addition, the proof of the
degree conjecture is also presented there (Theorem \ref{Thm1.7}), which
implies that the orthogonal LMOV Conjecture asymptotically holds (for all
partitions $\overrightarrow{\mu }$ and all knots/links) as $q$ tends to $1$
(Theorem \ref{Thm1.6}). In Section 10 (Appendix), we first compute explicit
expressions of the Chern-Simon partition function for the unknot. We then
review an alternative definition of the colored Kauffman polynomial via the
Markov trace (skein approach) and test the Hopf link for the orthogonal LMOV
conjecture by using this new definition. We also give an explicit
computation of the quantum trace for orthogonal quantum groups directly from
the universal $R$-matrix. Finally, we list the character table of Brauer
algebra and type-B Schur functions, whose specialization gives colored
Kauffman polynomials of the unknot (quantum dimensions) for small
partitions. These tables are mainly used to compute colored Kauffman
polynomial for torus knots and links. The tables of the integers
coefficients predicted by the orthogonal LMOV conjecture are also presented.

\subsection{Acknowledgments}

The authors greatly benefited from Nicolai Reshetikhin's early participation
in this project, and owe him a lot for his help, advices and support. We
also thank Kefeng Liu, Pan Peng and Hao Zheng for explaining their works to
us and many very useful discussions. We thank Marcos Mari\~{n}o for
communicating with us on this subject and his interest and numerous useful
comments. We thank Francis Bonahon for his enthusiasm, advices, and support.
Part of this work was done while the authors visited the Center of
Mathematical Science at Zhejiang University. The second author also thanks
the Hausdorff Institute of Mathematics at Bonn and IH\'{E}S for their
hospitality. The first draft of this paper was ready in the Fall 2008. The
first author passed away in a tragic accident in 2009. The current version
is presented here by the second author in memory of his good friend and
collaborator Lin Chen.

\section{Young Diagram and Birman-Murakami-Wenzl Algebra}

\label{sec2}

\subsection{Partition and Young Diagram}

A \emph{composition} $\mu $ of $n$, denoted by $\mu \models n$, is a finite
sequence of positive integers $(\mu _{1},\mu _{2},...,\mu _{\ell })$ such
that

\begin{equation*}
\mu _{1}+\mu _{2}+...+\mu _{\ell }=n\text{.}
\end{equation*}

The number of parts in $\mu $ is called the length of $\mu $ and denote by $%
\ell =\ell (\mu )$. The size of composition $\mu $ is defined by%
\begin{equation*}
|\mu |=\overset{\ell (\lambda )}{\underset{i=1}{\sum }}\mu _{i}.
\end{equation*}

A \emph{partition} $\lambda $ is a composition such that

\begin{equation*}
\lambda _{1}\geq \lambda _{2}\geq ...\geq \lambda _{\ell }>0\text{.}
\end{equation*}

Denote by $\mathcal{P}$ the set of all partitions. We identify a partition
with its Young diagram.

If $|\lambda |=d$, we say $\lambda $ is a partition of $d$, denote by $%
\lambda \vdash d$.

We use $m_{i}(\lambda )$ to denote the number of times that $i$ appears in $%
\lambda $. Denote the automorphism group of the partition $\lambda $ by $%
Aut(\lambda )$.

The order of $Aut(\lambda )$ is given by%
\begin{equation*}
|Aut(\lambda )|=\underset{i}{\prod }m_{i}(\lambda )!
\end{equation*}

There is another way to rewrite a partition $\lambda $ in the following
format%
\begin{equation*}
(1^{m_{1}(\lambda )}2^{m_{2}(\lambda )}\cdot \cdot \cdot )
\end{equation*}

For Instance, we have $(5,3,3,2,2,2,1)=(1^{1}2^{3}3^{2}5^{1})$

Define the following numbers associated to a partition $\lambda $.

\begin{equation*}
\mathrm{z}_{\lambda }=\underset{i}{\prod }i^{m_{i}(\lambda )}m_{i}(\lambda
)!,
\end{equation*}

\begin{equation*}
\kappa _{\lambda }=\underset{j=1}{\overset{\ell (\lambda )}{\prod }}\lambda
_{j}(\lambda _{j}-2j+1)\text{.}
\end{equation*}

\subsection{Partitionable set and infinite series}

Following the notations of \cite{LP}, we present some basic knowledge of
partitionable set here.

The concept of partition can be generalized to the following partitionable
set.

A countable set $(S,+)$ is called a \emph{partitionable} set if the
following holds

\begin{enumerate}
\item $S$ is totally ordered.

\item $S$ is an Abelian semi-group with summation $"+"$.

\item The minimum element $\mathbf{0}$ in $S$ is the zero-element of the
semi-group, i.e., for any $a\in S$,
\end{enumerate}

\begin{equation*}
\mathbf{0}+a=a=a+\mathbf{0}.
\end{equation*}

For simplicity, we may briefly write $S$ instead of $(S,+)$.

The following sets are examples of partitionable set:

\begin{enumerate}
\item The set of all nonnegative integers $\mathbb{Z}_{\geq 0}$;

\item The set of all partitions $\mathcal{P}$. The order of $\mathcal{P}$
can be defined as follows:

$\forall \lambda $, $\mu \in \mathcal{P}$, $\lambda \geq \mu $ iff $|\lambda
|>|\mu |$, or $|\lambda |=|\mu |$ and there exists a $j$ such that $\lambda
_{i}=\mu _{j}$ for $i\leq j-1$ and $\lambda _{j}>\mu _{i}$. The summation is
taken to be $"\cup "$ and the zero-element is $(0)$.

\item $\mathcal{P}^{n}$. The order of $\mathcal{P}^{n}$ is defined similarly
as $(2)$:

$\forall \overrightarrow{A}$, $\overrightarrow{B}\in \mathcal{P}^{n}$, $%
\overrightarrow{A}\geq \overrightarrow{B}$ iff $\overset{n}{\underset{i=1}{%
\sum }}|A^{i}|>\overset{n}{\underset{i=1}{\sum }}|B^{i}|$, or $\overset{n}{%
\underset{i=1}{\sum }}|A^{i}|=\overset{n}{\underset{i=1}{\sum }}|B^{i}|$ and
there is a $j$ such that $A^{i}=B^{i}$ for $i\leq j-1$ and $A^{j}>B^{j}$.
\end{enumerate}

Define

\begin{equation*}
\overrightarrow{A}\cup \overrightarrow{B}=(A^{1}\cup B^{1},A^{2}\cup
B^{2},...,A^{n}\cup B^{n})
\end{equation*}

$((0),(0),...,(0))$ is the zero-element. Then $\mathcal{P}^{n}$ is a
partitionable set.

Let $S$ be a partitionable set. One can define partition with respect to $S$
in the similar manner as that of $\mathbb{Z}_{\geq 0}$: a finite sequence of
non-increasing non-minimum elements in $S$. We will call it an \emph{%
S-partition}, $(\mathbf{0})$ the zero $S$-partition. Denote by $\mathcal{P}%
(S)$ the set of all $S$-partitions.

For an $S$-partition $\Lambda $, we can define the automorphism group of $%
\Lambda $ in a similar way as that in the definition of traditional
partition. Given $\beta \in S$, denote by $m_{\beta }(\Lambda )$ the number
of times that $\beta $ occurs in the parts of $\Lambda $, we then have
\begin{equation*}
Aut\Lambda =\prod_{\beta \in S}m_{\beta }(\Lambda )!\,.
\end{equation*}%
Introduce the following quantities associated with $\Lambda $,
\begin{equation*}
\mathfrak{u}_{\Lambda }=\frac{\ell (\Lambda )!}{|Aut\Lambda |}\,,\qquad {%
\Theta }_{\Lambda }=\frac{(-1)^{\ell (\Lambda )-1}}{\ell (\Lambda )}%
\mathfrak{u}_{\Lambda }\,.
\end{equation*}

The following Lemma will be used in Section \ref{sec4} to deduce the
reformulated invariants.

\begin{lemma}[\protect\cite{LP}, Lemma 2.3]
\label{Lemma2.1} Let $S$ be a partionable set. If $f(t)=\underset{n\geq 0}{%
\sum }a_{n}t^{n}$, then%
\begin{equation*}
f\left( \underset{\beta \in S}{\underset{\beta \neq \mathbf{0}}{\sum }}%
A_{\beta }p_{\beta }(x)\right) =\underset{\Lambda \in \mathcal{P}(S)}{\sum }%
a_{\ell (\Lambda )}A_{\Lambda }p_{\Lambda }(x)\mathfrak{u}_{\Lambda }\text{,}
\end{equation*}

where
\begin{equation*}
p_{\Lambda }(x)=\underset{j=1}{\overset{\ell (\Lambda )}{\prod }}p_{\Lambda
_{j}}\text{, }A_{\Lambda }=\underset{j=1}{\overset{\ell (\Lambda )}{\prod }}%
A_{\Lambda _{j}}\text{.}
\end{equation*}
\end{lemma}

\begin{proof}
Note that%
\begin{equation*}
\left( \underset{\beta \in S}{\underset{\beta \neq \mathbf{0}}{\sum }}\eta
_{\beta }\right) ^{n}=\underset{\ell (\Lambda )=n}{\underset{\Lambda \in
\mathcal{P}(S)}{\sum }}\eta _{\Lambda }\mathfrak{u}_{\Lambda }\text{.}
\end{equation*}
\end{proof}

\subsection{Birman-Murakami-Wenzl Algebra}

The centralizer algebra
\begin{equation}
\mathrm{End}_{U_{q}(\mathfrak{so}(2N+1))}(V^{\otimes n})=\{f\in \mathrm{End}%
(V^{\otimes n})|fx=xf,\forall x\in U_{q}(\mathfrak{so}(2N+1))\}
\end{equation}%
for the standard representation of $U_{q}(\mathfrak{so}(2N+1))$ on $V=%
\mathbb{C}^{2N+1}$ is isomorphic, when $N>n$, to the Birman-Murakami-Wenzl
algebra $C_{n}$.

Let $\mathbb{C}(t,q)$ be the field of rational functions with two variables.
For each positive integer $n$, the Birman-Murakami-Wenzl algebra is defined
to be an algebra $C_n$ over $\mathbb{C}(t,q)$ as follows. The algebra $C_1$
is of one dimensional, and thus is identified to $\mathbb{C}(t,q)$. For $n>1$%
, $C_n$ is generated over $\mathbb{C}(t,q)$ by the generators $%
g_1,g_2,\cdots,g_{n-1},e_1,e_2,\cdots,e_{n-1}$ and the relations

(A1)\hskip0.5cm $g_{i}g_{i+1}g_{i}=g_{i+1}g_{i}g_{i+1}$ for $1\leq i\leq n-2$

(A2)\hskip 0.5cm $g_ig_j=g_jg_i$ if $|i-j|\geqslant 2$

(A3)\hskip 0.5cm $e_ig_i=t^{-1}e_i$

(A4)\hskip 0.5cm $e_ig_{i-1}^{\pm 1}e_i=t^{\pm 1}e_i$

(A5)\hskip 0.5cm $(q-q^{-1})(1-e_i)=g_i-g_i^{-1}$.

The first two properties are the braiding relations. The following two
properties are immediate from the above definition

(P1)\hskip 0.5cm $e_i^2=xe_i$ for $x=1+\frac{t-t^{-1}}{q-q^{-1}}$

(P2)\hskip 0.5cm $(g_i-t^{-1})(g_i+q^{-1})(g_i-q)=0$.

When the variable $q,t$ approaches to $1$, while $x=1+\frac{t-t^{-1}}{%
q-q^{-1}}$ is fixing, the above BMW algebra specializes to the Brauer
algebra $Br_{n}$, which is semisimple and isomorphic to the centralizer
algebra $\mathrm{End}_{\mathfrak{so}(2N+1)}(V^{\otimes n})$ if $N>n$, cf.
\cite{Bra} and also \cite{Wey}. The Birman-Murakami-Wenzl algebras are
semisimple except possibly when $q$ takes the value of roots of unity or $%
t=q^{m}$ for some integer $m$. Obviousely, the BMW algebra is the
deformation of the Brauer algebra.

\subsection{Irreducible Representations of Brauer Algebras}

For our purpose, we focus the generic case when the BMW Algebras $C_n$ are
semisimple. In this situation, its irreducible representation can be
described similar to the Brauer Algebras $Br_n$.

As the centralizer algebra $\mathrm{End}_{\mathfrak{so}(2N+1)}V^{\otimes n}$%
, $Br_{n}$ contains the group algebra $\mathbb{C}[S_{n}]$ as a direct
summand, thus all the irreducible representations of $S_{n}$ are also
irreducible representations of $Br_{n}$, labeled by partitions of the
integer $n$. Indeed, the set of irreducible representations of $Br_{n}$ are
bijective to the set of partitions of the integers $n-2k$, where $%
k=0,1,\cdots ,[\frac{n}{2}]$ \cite{Ram2, Wen1}. Thus the semi-simple algebra
$Br_{n}$ can be decomposed into the direct sum of simple algebras
\begin{equation*}
Br_{n}\cong \overset{\lbrack \frac{n}{2}]}{\bigoplus_{k=0}}%
\bigoplus_{\lambda \vdash n-2k}M_{d_{\lambda }\times d_{\lambda }}(\mathbb{C}%
).
\end{equation*}

The work of Beliakova and Blanchet \cite{BB} constructed an explicit basis
of the above decomposition. An up and down tableau $\Lambda =(\lambda
_{1},\lambda _{2},\cdots ,\lambda _{n})$ is a tube of $n$ Young diagrams
such that $\lambda _{1}=(1)$ and each $\lambda _{i}$ is obtained by adding
or removing one box from $\lambda _{i-1}$. Let $\lambda $ be a partition of $%
n-2k$. Denote by $|\Lambda |=\lambda $ if $\lambda _{n}=\lambda $, and we
say an up and down tableau $\Lambda $ is of shape $\lambda $. There is a
minimal path idempotent $p_{\Lambda }\in Br_{n}$ associated to each $\Lambda
$. Then the minimal central idempotent $\pi _{\lambda }$ of $Br_{n}$
correspond to the irreducible representation labeled by $\lambda $ is given
by
\begin{equation*}
\pi _{\lambda }=\sum_{|\Lambda |=\lambda }p_{\Lambda }.
\end{equation*}%
In particular, the dimension of the irreducible representations $d_{\lambda
} $ is the number of up and down tableau of shape $\lambda $. More detail
can be found in \cite{BB, Wen1}.

The characters table and the orthogonal relations can be found in \cite%
{Ram1, Ram2, Ram3}. The values of a character of $Br_{n}$ is completely
determined by its values on the set of elements $e^{k}\otimes \gamma
_{\lambda }$, where $e$ is the conjugacy class of $e_{1},\cdots ,e_{n-1}$
and $\gamma _{\lambda }$ is the conjugacy class in $S_{n-2k}$ labeled by the
partition $\lambda $ of $n-2k$. The notion $e^{k}\otimes \gamma _{\lambda }$
stands for the tangle in the following diagram.
\begin{align*}
& \ e_{0}\ \ \ \ \ e_{2}\ \ \ \ \cdots \ \ \ \ \ e_{2k}\ \ \ \ \ \ \gamma
_{\lambda }\ \ \  \\
& {%
\includegraphics[height=1.0153in, width=2.1318in]
{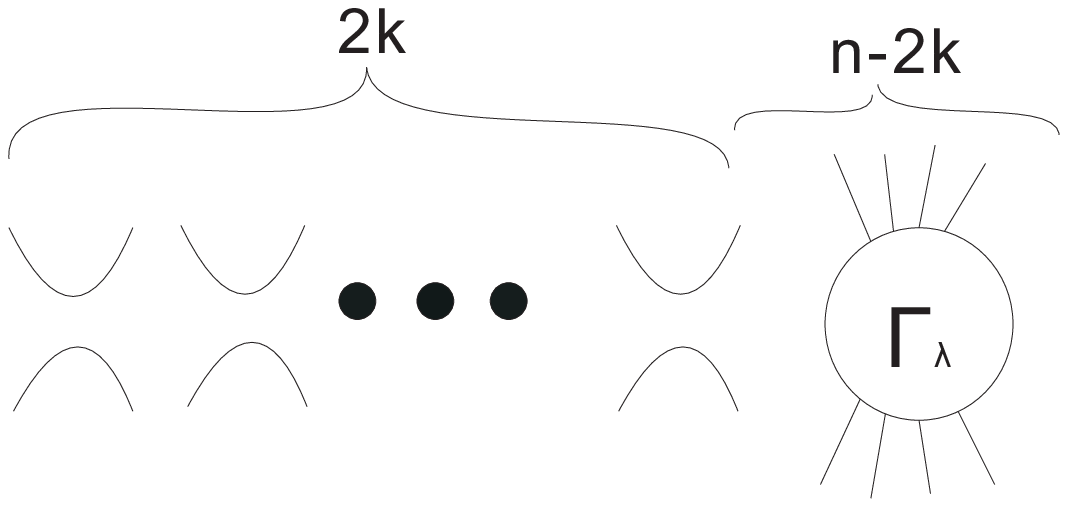}}
\end{align*}%
where $\Gamma _{\lambda }$ is a diagram in the conjugacy class of $S_{n-2k}$
labeled by a partition $\lambda $ of $n-2k$.

Denote $\chi _{A}$ the character of the irreducible representation of $%
Br_{n} $ labeled by a partition $A\vdash n-2k$ for some $k$, and denote by $%
\chi _{B}^{S_{n}}$ the character of the irreducible representation of $S_{n}$
labeled by a partition $B\vdash n$. It is known that when $A$ is a partition
of $n$, then $\chi _{A}(e^{m}\otimes \gamma _{\lambda })=0$ for all $m>0$
and partition $\lambda \vdash n-2m$, and $\chi _{A}(\gamma _{\mu })=\chi
_{A}^{S_{n}}(\gamma _{\mu })$ for partition $\mu \vdash n$ coincide with the
characters of the permutation group $S_{n}$ \cite{Ram2}.

\subsection{Schur-Weyl Duality}

Both $\mathfrak{so}(2N+1)$ and $Br_{n}$ acts on the tensor product $%
V^{\otimes n}$ and their actions commute each other. As a bi-module, $%
V^{\otimes n}$ has the following decomposition
\begin{equation*}
V^{\otimes n}=\bigoplus_{\lambda }V_{\lambda }\otimes U_{\lambda },
\end{equation*}%
where $\lambda $ runs through all the partitions of $n,n-2,n-4,\cdots ,0$, $%
V_{\lambda }$ (resp. $U_{\lambda }$) is the irreducible representation of $%
\mathfrak{so}(2N+1)$ (resp. $Br_{n}$) labeled by $\lambda $. A similar
decomposition holds for the pair $U_{q}(\mathfrak{so}(2N+1))$ and $C_{n}$.

A power symmetric function of a sequence of variables $z=(z_{i})_{i\in
\mathbb{Z}}$ is defined by

\begin{equation*}
pb_{n}(z)=(z_{0})^{n}+\underset{i=1}{\overset{+\infty }{\sum }}%
[(z_{i})^{n}+(z_{-i})^{n}].
\end{equation*}

For a partition $\lambda $,%
\begin{equation*}
pb_{\lambda }(z)=\overset{\ell (\lambda )}{\underset{j=1}{\prod }}%
pb_{\lambda _{j}}(z).
\end{equation*}

Denote $\widehat{Br}_{n}$ the set of all the characters of $Br_{n}$. For
each partition $A$, we use $sb_{A}$ to denote the type-B Schur function
associated to $A$ with infinitely many variables $z_{0},z_{\pm 1},z_{\pm
2},\cdots $, which are completely determined by the system of equations
inductively
\begin{equation}
x^{k}pb_{\lambda }=\sum_{A\in \widehat{Br}_{n}}\chi _{A}(e^{\otimes
k}\otimes \gamma _{\lambda })sb_{A}.
\end{equation}%
The parameter $x$ is the structure constant in the definition of the Brauer
algebra $Br_{n}$. The type-B Schur functions is independent of this
parameter $x$, as one can see from the character formula of Brauer algebra,
given by A. Ram in \cite{Ram2} Theorem 5.1. If $A$ is a partition of $n$,
then $sb_{A}$ is a symmetric polynomial of degree $n$ (not necessarily
homogeneous).

Throughout this paper, we fix the following notations for partition set $%
\mathcal{P}^{L}$, where $L$ is the number of components of link $\mathcal{L}$%
.

For $\overrightarrow{\mu }=(\mu ^{1},\mu ^{2},\cdots ,\mu ^{L})\in \mathcal{P%
}^{L}$, denote%
\begin{equation}
|\overrightarrow{\mu }|=(|\mu ^{1}|,|\mu ^{2}|,\cdots ,|\mu ^{L}|)\in
\mathbb{Z}^{L}
\end{equation}

and define%
\begin{equation}
||\overrightarrow{\mu }||=\underset{\alpha =1}{\overset{L}{\sum }}|\mu
^{\alpha }|.
\end{equation}%
Write%
\begin{equation}
\ell (\overrightarrow{\mu })=\underset{\alpha =1}{\overset{L}{\sum }}\ell
(\mu ^{\alpha })
\end{equation}
for the sum of the length of each partition.

We denote $pb_{\overrightarrow{\mu }}(\overrightarrow{z})=\overset{L}{%
\underset{\alpha =1}{\prod }}pb_{\mu ^{\alpha }}(z_{\alpha })$, where $%
z_{\alpha }=(z_{\alpha ,i})_{i\in \mathbb{Z}}$.

Let $\widehat{Br}_{|\overrightarrow{\mu }|}$ denotes the set $\widehat{Br}%
_{|\mu ^{1}|}\times \cdots \times \widehat{Br}_{|\mu ^{L}|}$, then $\chi _{%
\overrightarrow{A}}(\gamma _{\overrightarrow{\mu }})=\underset{\alpha =1}{%
\overset{L}{\prod }}\chi _{A^{\alpha }}(\gamma _{\mu ^{\alpha }})$ for the
character $\chi _{A^{\alpha }}$ of $Br_{|\mu ^{\alpha }|}$ labeled by $%
A^{\alpha }$, a partition of $|\mu ^{\alpha }|-2k^{\alpha }$, and the
conjugacy class $\gamma _{\mu ^{\alpha }}$ of $Br_{d_{\alpha }}$ labeled by $%
\mu ^{\alpha }$.

\section{Colored Kauffman Polynomials and Cabling Formula}

\label{sec3}

\subsection{Colored Kauffman Polynomials (Orthogonal Quantum Groups
Invariants)}

Let $B_{m}$ be the braid group of $m$ strands which is generated by $\sigma
_{1},\cdots ,\sigma _{m-1}$ with following defining relations:

\begin{equation}
\left\{
\begin{array}{c}
\sigma _{i}\sigma _{j}=\sigma _{j}\sigma _{i} \\
\sigma _{i}\sigma _{j}\sigma _{i}=\sigma _{j}\sigma _{i}\sigma _{j}%
\end{array}%
\right.
\begin{array}{c}
if\text{ }|i-j|\geq 2 \\
if\text{ }|i-j|=1%
\end{array}%
\end{equation}

Every link can be represented by the closure of some element in braid group $%
B_{m}$. This kind of braid representation is not unique. We fix such a braid
representation, then we define the quantum group invariants of link via this
braid. Finally we will see such kind of definition is independent of the
choice of the braid representation.

Let $\mathfrak{g}$ be a finite dimensional complex simple Lie algebra and $%
U_{q}(\mathfrak{g})$ be the corresponding quantized enveloping algebra.

The ribbon category structure associated with $U_{q}(\mathfrak{g})$ is given
by the following data:

\begin{enumerate}
\item Associated to each pair of $U_{q}(\mathfrak{g})$-modules $V$ and $W$,
there is an isomorphism
\begin{equation*}
\check{\mathcal{R}}_{V,\,W}:\,V\otimes W\rightarrow W\otimes V
\end{equation*}%
such that
\begin{align*}
\check{\mathcal{R}}_{U\otimes V,\,W}& =(\check{\mathcal{R}}_{U,\,W}\otimes
\mathrm{id}_{V})(\mathrm{id}_{U}\otimes \check{\mathcal{R}}_{V,\,W}) \\
\check{\mathcal{R}}_{U,\,V\otimes W}& =(\mathrm{id}_{V}\otimes \check{%
\mathcal{R}}_{U,\,W})(\check{\mathcal{R}}_{U,\,V}\otimes \mathrm{id}_{W})
\end{align*}%
for $U_{q}(\mathfrak{g})$-modules $U$, $V$, $W$.

Given $f\in \mathrm{Hom}_{U_{q}(\mathfrak{g})}(U,\widetilde{U})$, $g\in
\mathrm{Hom}_{U_{q}(\mathfrak{g})}(V,\widetilde{V})$, one has the following
naturality condition:
\begin{equation*}
(g\otimes f)\circ \check{\mathcal{R}}_{U,\,V}=\check{\mathcal{R}}_{%
\widetilde{U},\,\widetilde{V}}\circ (f\otimes g)\,.
\end{equation*}

\item There exists an element $K_{2\rho }\in U_{q}(\mathfrak{g})$, called
the enhancement of $\check{\mathcal{R}}$, such that
\begin{equation*}
K_{2\rho }(v\otimes w)=K_{2\rho }(v)\otimes K_{2\rho }(w)
\end{equation*}%
for any $v\in V$, $w\in W$. Here $\rho $ is the half-sum of all positive
roots of $\mathfrak{g}$.

Moreover, for every $z\in \mathrm{End}_{U_{q}(\mathfrak{g})}(V,W)$ with $z=%
\underset{i}{\sum }x_{i}\otimes y_{i}$, $x_{i}\in End(V)$, $y_{i}\in End(W)$
one has the quantum trace%
\begin{equation*}
\mathrm{tr}_{W}(z)=\underset{i}{\sum }\mathrm{tr}(y_{i}K_{2\rho })\cdot
x_{i}\in \mathrm{End}_{U_{q}(\mathfrak{g})}(V)
\end{equation*}

\item For any $U_{q}(\mathfrak{g})$-module $V$, the ribbon structure ${%
\theta }_{V}:V\rightarrow V$ associated to $V$ satisfies
\begin{equation*}
\theta _{V}^{\pm 1}=\mathrm{tr}_{V}\check{\mathcal{R}}_{V,V}^{\pm 1}.
\end{equation*}%
The ribbon structure also satisfies the following naturality condition
\begin{equation*}
x\cdot \theta _{V}=\theta _{\widetilde{V}}\cdot x.
\end{equation*}%
for any $x\in \mathrm{Hom}_{U_{q}(\mathfrak{g})}(V,\widetilde{V})$.
\end{enumerate}

Let $\mathcal{L}$ be a link with $L$ components $\mathcal{K}_{\alpha }$, $%
\alpha =1,\ldots ,L$, represented by the closure of $\beta \in B_{m}$. We
associate each $\mathcal{K}_{\alpha }$ an irreducible representation $%
V_{A^{\alpha }}$ of quantized universal enveloping algebra $U_{q}(\mathfrak{g%
})$ labeled by highest weight $A^{\alpha }$. In the sense of \cite{Ram2},
these irreducible representations can be labeled by partitions. By abuse of
notations, we use $A^{\alpha }$'s to denote those partitions. Let $%
i_{1},\cdots ,i_{m}$ be integers such that $i_{k}=\alpha $ if the $k$-th
strand of $\beta $ belongs to the $\alpha $-th component of $\mathcal{L}$.

Let $U$, $V$ be two $U_{q}(\mathfrak{g})$-modules labeling two outgoing
strands of the crossing, the braiding $\check{\mathcal{R}}_{U,V}$ (resp. $%
\check{\mathcal{R}}_{V,U}^{-1}$) is assigned as in following figure.

\begin{align*}
& {\includegraphics[width=1in]
{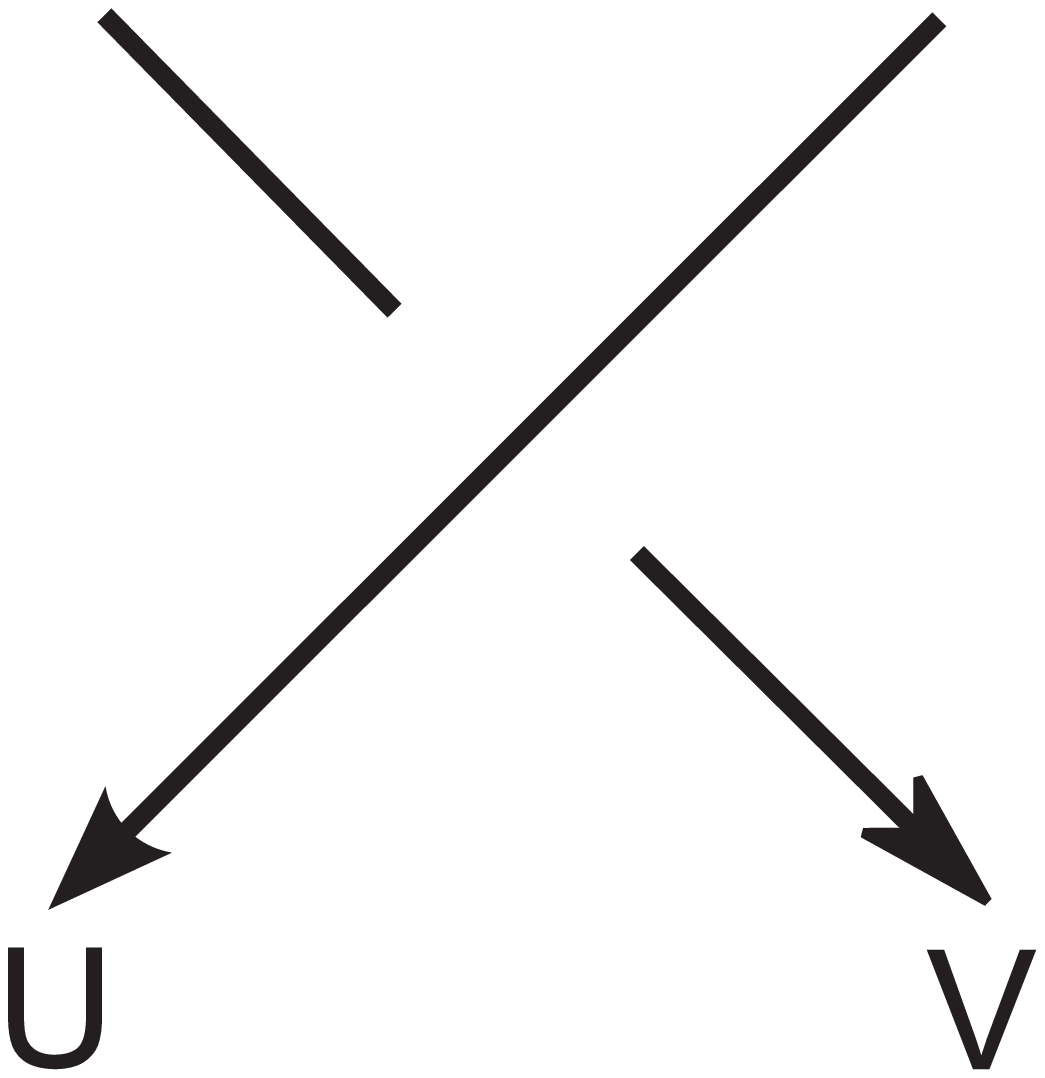}\hskip0.75in\includegraphics[width=1in]
{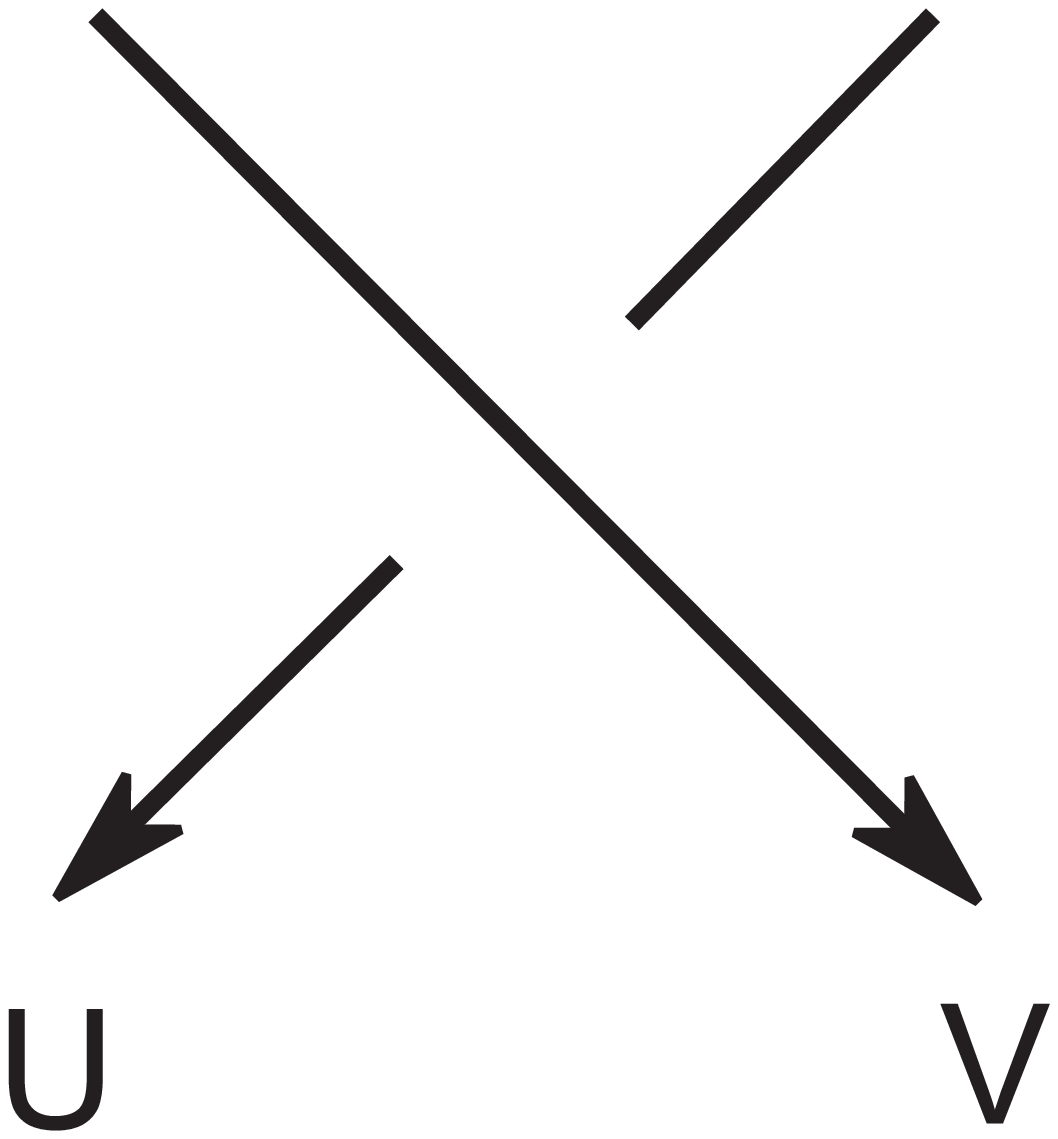}} \\
& \text{ \ \ \ \ \ \ }\check{\mathcal{R}}_{U,V}\text{ \ \ \ \ \ \ \ \ \ \ \
\ \ \ \ \ \ \ \ \ \ \ \ \ \ \ \ }\check{\mathcal{R}}_{V,U}^{-1} \\
& \text{ \ \ \ \ \ \ \ \ \ \ \ \ \ \ Assign crossing by }\check{\mathcal{R}}.
\end{align*}

The above assignment will give a representation of $B_{m}$ on $U_{q}(%
\mathfrak{g})$-module $V_{A^{i_{1}}}\otimes \cdots \otimes V_{A^{i_{m}}}$.
Namely, for any generator $\sigma _{j}\in B_{m}$,

define
\begin{equation*}
h(\sigma _{j})=\mathrm{id}_{V_{A^{i_{1}}}}\otimes \cdots \otimes \check{%
\mathcal{R}}_{V_{A^{i_{j+1}}},V_{A^{i_{j}}}}\otimes \cdots \otimes \mathrm{id%
}_{V_{A^{i_{m}}}}\,,
\end{equation*}

and

\begin{equation*}
h(\sigma _{j}^{-1})=\mathrm{id}_{V_{A^{i_{1}}}}\otimes \cdots \otimes \check{%
\mathcal{R}}_{V_{A^{i_{j}}},V_{A^{i_{j+1}}}}^{-1}\otimes \cdots \otimes
\mathrm{id}_{V_{A^{i_{m}}}}\,,
\end{equation*}

Therefore, any link $\mathcal{L}$ will provide an isomorphism
\begin{equation*}
h(\beta )\in \mathrm{End}_{U_{q}(\mathfrak{g})}(V_{A^{i_{1}}}\otimes \cdots
\otimes V_{A^{i_{m}}})\,.
\end{equation*}

The representation of braid group $B_{n}$ on $V^{\otimes n}$ factors through
the BMW algebra $C_{n\text{ }}$by sending $\sigma _{j}$ to $g_{j}\in C_{n}$.
By abuse of notations, we still denote this via $g_{j}=h(\sigma _{j})$.

The quantum trace%
\begin{equation*}
\mathrm{tr}_{V_{A^{i_{1}}}\otimes \cdots \otimes V_{A^{i_{m}}}}h(\beta )
\end{equation*}

defines framing dependent link invariant of link $\mathcal{L}$.

In order to eliminate the framing dependency, we make the following
refinement \cite{LZ}

\begin{equation*}
W_{V_{A^{1}},\cdots ,V_{A^{L}}}^{\mathfrak{so}(2N+1)}(\mathcal{L};q)=\theta
_{V_{A^{1}}}^{-w(\mathcal{K}_{1})}\cdot \cdot \cdot \theta _{V_{A^{L}}}^{-w(%
\mathcal{K}_{L})}\mathrm{tr}_{V_{A^{i_{1}}}\otimes \cdots \otimes
V_{A^{i_{m}}}}(h(\beta ))\text{,}
\end{equation*}%
where $w(\mathcal{K}_{\alpha })$ is the writhe number of $\mathcal{K}%
_{\alpha }$ in $\beta $, i.e., the number of positive crossing minus the
number of negative crossings.

The above quantity is invariant under the Markov moves, hence is an
invariant of the underlying link $\mathcal{L}$.

Quantum group invariants of links can be defined over any complex simple Lie
algebra $\mathfrak{g}$. However, in this paper, we mainly consider the
quantum group invariants of links defined over $\mathfrak{so}(2N+1)$. More
generally, one can also include the case for $\mathfrak{so}(2N)$ and $%
\mathfrak{sp}(2N)$; however, we will not do so, since the quantum group
invariants associated to these Lie algebras all give the colored Kauffman
polynomials. To distinguish $U_{q}(\mathfrak{so}(2N+1))$ from the quantum
group corresponding to spin group, we only consider those representations
parameterized by the highest weights in the root lattice of the Lie group $%
SO(2N+1)$, instead of the spin group. These highest weights are, similar to
the case of $\mathfrak{sl}_{N}$, partitions of length at most $N$, i.e $%
\{\mu |\mu _{1}\geqslant \mu _{2}\geqslant \cdots \geqslant \mu
_{N}\geqslant 0\}$.

Let's consider $U_{q}(\mathfrak{so(}2N+1))$, the quantized universal
enveloping algebra of orthogonal lie algebra $\mathfrak{so(}2N+1)$, which is
generated by $\{H_{i},X_{i}^{+},X_{i}^{-}\}$ together with the following
defining relations:%
\begin{equation*}
\lbrack H_{i},H_{j}]=0\text{, }[H_{i},X_{j}^{\pm }]=\pm (C)_{ij}X_{j}^{\pm }%
\text{ and }[X_{i}^{+},X_{j}^{-}]=\delta _{ij}\frac{q^{H_{i}}-q^{-H_{i}}}{%
q-q^{-1}}\text{,}
\end{equation*}%
where $C$ is the Cartan matrix of $\mathfrak{g=so(}2N+1)$

and the Serre type relations

\begin{equation*}
\underset{k=0}{\overset{1-(C)_{ij}}{\sum }}(-1)^{k}\frac{[1-(C)_{ij}]_{q}!}{%
[1-(C)_{ij}-k]_{q}![k]_{q}!}(X_{i}^{\pm })^{1-(C)_{ij}-k}X_{j}^{\pm
}(X_{i}^{\pm })^{k}=0\text{, for all }i\neq j\text{,}
\end{equation*}%
where $[k]_{q}!=\overset{k}{\underset{i=1}{\prod }}[i]_{q}$ and the $q$%
-number is defined as%
\begin{equation*}
\lbrack n]_{q}=\frac{q^{n}-q^{-n}}{q-q^{-1}}\text{.}
\end{equation*}

When $q\rightarrow 1$, the universal enveloping algebra $U_{q}(\mathfrak{so(}%
2N+1))$ reduces to the Lie algebra $\mathfrak{so(}2N+1)$.

Drinfeld \cite{Dri} defined the universal $\mathcal{R}$-matrix of $U_{q}(%
\mathfrak{g})$ as
\begin{equation}
\mathcal{R}=q^{\underset{i,j}{\sum }(C^{-1})_{ij}H_{i}\otimes
H_{j}}\prod_{\beta \in \vartriangle ^{+}}\exp _{q}[(1-q^{-2})X_{\beta
}^{+}\otimes X_{\beta }^{-}]\text{,}
\end{equation}%
where $\vartriangle ^{+}$denotes the set of positive roots and the $q$%
-exponential is of the form
\begin{equation*}
\exp _{q}(x)=\sum_{k=0}^{\infty }q^{\frac{1}{2}k(k+1)\frac{x^{k}}{[k]_{q}!}}%
\text{.}
\end{equation*}

The ribbon category structure is defined by letting $\check{R}=P_{12}%
\mathcal{R}$ for the above universal $\mathcal{R}$-matrix, and taking $%
K_{2\rho }$ to be $q^{-\rho ^{\ast }}$. The operator $P_{12}:V\otimes
W\rightarrow W\otimes V$ switches the two components, and $\rho ^{\ast }$
denotes the element in the Cartan subalgebra $\mathfrak{h}\subset \mathfrak{g%
}$ corresponding to $\rho $.

The positive roots of $\mathfrak{so}(2N+1)$ are given by $\vartheta _{i}\pm
\vartheta _{j}$ for $1\leqslant i<j\leqslant N$ and $\vartheta
_{1},\vartheta _{2},\cdots ,\vartheta _{N}$, where $\vartheta _{i}$ has
eigenvalue $x_{i}$ when acting on the matrix element%
\begin{equation*}
\mathrm{diag}\{-x_{N},-x_{N-1},\cdots ,-x_{1},0,x_{1},\cdots ,x_{N-1},x_{N}\}
\end{equation*}
in the Cartan subalgebra. The sum of the positive roots is given by%
\begin{equation*}
2\rho =\underset{i=1}{\overset{N}{\sum }}\vartheta _{i}+\underset{1\leqslant
i<j\leqslant N}{\sum }[(\vartheta _{i}-\vartheta _{j})+(\vartheta
_{i}+\vartheta _{j})]=\underset{i=1}{\overset{N}{\sum }}(2N+1-2i)\vartheta
_{i}\text{,}
\end{equation*}%
and%
\begin{equation*}
K_{2\rho }=\mathrm{diag}\{q^{1-2N},q^{3-2N},\cdots
,q^{-3},q^{-1},1,q,q^{3}\cdots ,q^{2N-3},q^{2N-1}\}.
\end{equation*}

Alternatively, we can write%
\begin{equation*}
K_{2\rho }(v_{i})=\left\{
\begin{array}{c}
q^{2i-1-2N}v_{i} \\
v_{i} \\
q^{2i-3-2N}v_{i}%
\end{array}%
\right.
\begin{array}{c}
1\leq i\leq N \\
i=N+1 \\
N+2\leq i\leq 2N+1%
\end{array}%
\text{.}
\end{equation*}

The universal matrix $\check{\mathcal{R}}$ acting on $V\otimes V$ for the
natural representation of $U_{q}(\mathfrak{so}(2N+1))$ on $V$ is given by
Turaev \cite{Tur}:

\begin{align*}
\check{\mathcal{R}}=& q\sum_{i\neq N+1}E_{i,i}\otimes
E_{i,i}+E_{N+1,N+1}\otimes E_{N+1,N+1}+\sum_{j}\underset{i\neq 2N+2-j}{%
\sum_{i\neq j}}E_{j,i}\otimes E_{i,j} \\
& +q^{-1}\sum_{i\neq N+1}E_{2N+2-i,i}\otimes
E_{i,2N+2-i}+(q-q^{-1})\sum_{i<j}E_{i,i}\otimes E_{j,j} \\
& -(q-q^{-1})\sum_{i<j}q^{\overline{i}-\overline{j}}E_{2N+2-j,i}\otimes
E_{j,2N+2-i}\text{,}
\end{align*}%
where $E_{i,j}$ is the $(2N+1)\times (2N+1)$ matrix with%
\begin{equation*}
(E_{i,j})_{kl}=\left\{
\begin{array}{c}
1 \\
0%
\end{array}%
\right.
\begin{array}{c}
(k,l)=(i,j) \\
elsewhere%
\end{array}%
\end{equation*}

and

\begin{equation*}
\overline{i}=\left\{
\begin{array}{c}
i+\frac{1}{2} \\
i \\
i-\frac{1}{2}%
\end{array}%
\right.
\begin{array}{c}
1\leq i\leq N \\
i=N+1 \\
N+2\leq i\leq 2N+1%
\end{array}%
\text{.}
\end{equation*}

The ribbon structure $\theta _{V_{A^{\alpha }}}$ is equal to $q^{<A^{\alpha
},A^{\alpha }+2\rho >}$ for $1\leq \alpha \leq L$.

Define the orthogonal quantum group invariants $W_{A^{1},\cdots ,A^{L}}^{SO}(%
\mathcal{L},q,t)\in \mathbb{C}(q,t)$ such that
\begin{equation}
W_{A^{1},\cdots ,A^{L}}^{SO}(\mathcal{L};q,q^{2N})=W_{V_{A^{1}},\cdots
,V_{A^{L}}}^{\mathfrak{so}(2N+1)}(\mathcal{L};q)=q^{\underset{\alpha =1}{%
\overset{L}{\sum }}-<A^{\alpha },A^{\alpha }+2\rho >w(\mathcal{K}_{\alpha })}%
\mathrm{tr}_{V_{A^{i_{1}}}\otimes \cdots \otimes V_{A^{i_{m}}}}(h(\beta ))%
\text{.}
\end{equation}

Then we want to compute the identity $q^{<A^{\alpha },A^{\alpha }+2\rho >}$.
We will first introduce the representation theory of the BMW algebra.

From now on, we only restrict ourselves in the case when the
Birman-Murakami-Wenzl algebra $C_{n}$ is semisimple and $N$ is large. The
representations of $C_{n}$ can be described in the same way as the Brauer
algebra $Br_{n}$. The semi-simplicity implies that the representation $%
V^{\otimes n}$ of $C_{n}$ admits a direct sum decomposition
\begin{equation*}
V^{\otimes n}=\bigoplus_{\lambda \in \widehat{Br}_{n}}d_{\lambda }\cdot
V_{\lambda }.
\end{equation*}%
The multiplicities $d_{\lambda }$ are all positive integers. In particular,
any irreducible representation $V_{A}$ of $U_{q}(so(2N+1))$ appear as a
direct summand of $V^{\otimes r}$ for integer $r=|A|,|A|+2,|A|+4,\cdots $.
By Schur lemma,
\begin{equation*}
C_{n}\cong \mathrm{End}_{U_{q}(\mathfrak{so}(2N+1))}V^{\otimes n}\cong
\bigoplus_{\lambda \in \widehat{Br}_{n}}C_{\lambda },
\end{equation*}%
where $C_{\lambda }=\mathrm{End}_{U_{q}(\mathfrak{so}(2N+1))}(d_{\lambda
}V_{\lambda })$ is isomorphic to the $d_{\lambda }\times d_{\lambda }$
matrix algebra, labelled by the characters $\widehat{Br}_{n}$ of $Br_{n}$ as
the decomposition of $V^{\otimes n}$.

A \emph{minimal idempotent} $p\in C_{n}$ satisfies $p^{2}=p$ and the action
of $U_{q}(\mathfrak{so}(2N+1))$ on the subspace $p\cdot V^{\otimes n}$ is an
irreducible representation. Another description of $p$ is that, there exist
exactly one $\lambda \in \widehat{Br}_{n}$ such that the restriction of $p$
to $C_{\lambda }$ is non-zero, and it's a diagonalizable matrix with exactly
one eigenvalue 1 and all others 0.

Let $y$ be an element in $C_{n}$, and the normal (or say, non-quantum) trace
of its $\lambda $ component via the above isomorphism is denoted by $\zeta
_{n}^{\lambda }(y)$. Since $y$ and all the idempotents are elements in $%
C_{n} $, they are finite linear combinations of products of the generators $%
g_{i}$'s and $e_{i}$'s, which imply $\zeta _{n}^{\lambda }(y)$ is, in
general, a rational function of $q$ and $t$.

It is not hard to get the following identity from the Turaev's \cite{Tur}
construction of universal matrix $\check{\mathcal{R}}$ (See Section \ref%
{sec10} for detail):%
\begin{equation*}
\theta _{V}=q^{2N}\cdot \mathrm{id}_{V}\text{,}
\end{equation*}%
where $V$ is the standard representation of $U_{q}(\mathfrak{so}(2N+1))$ on $%
\mathbb{C}^{2N+1}$.

More generally, we have the following lemma obtained by Reshetikhin \cite%
{Res}

\begin{lemma}
\label{AM type lemma}For each partition $\lambda \vdash n-2f$ with $\ell
(\lambda )\leq N$, one has%
\begin{equation*}
\theta _{V_{\lambda }}=q^{\kappa _{\lambda }+2N(n-2f)}\cdot \mathrm{id}%
_{V\lambda }\text{,}
\end{equation*}%
where $\kappa _{\lambda }=\underset{j=1}{\overset{\ell (\lambda )}{\prod }}%
\lambda _{j}(\lambda _{j}-2j+1)$.
\end{lemma}

This result can be understand in the following way. First we have%
\begin{equation*}
\theta _{V}=q^{2N}\cdot \mathrm{id}_{V}\text{.}
\end{equation*}

A result of Aiston-Morton (Theorem 5.5 of \cite{AM}, cf., Theorem 4.1 of
\cite{LZ}) states that%
\begin{equation*}
\theta _{V_{\lambda }}=q^{\kappa _{\lambda }+nN-\frac{n^{2}}{N}}\cdot
\mathrm{id}_{V\lambda }\text{.}
\end{equation*}

In \cite{LZ}, they use a different normalization for universal $\check{%
\mathcal{R}}$-matrices, thus they have%
\begin{equation*}
q^{\frac{1}{N}}\theta _{V}=q^{N}\cdot \mathrm{id}_{V}
\end{equation*}%
and also a different corresponding normalization for $h:\mathbb{C}%
B_{n}\rightarrow C_{n}(V)$ factoring through the Hecke algebra $\mathcal{H}%
_{n}(q)$ via%
\begin{equation*}
q^{\frac{1}{N}}\sigma _{i}\mapsto g_{i}\mapsto q^{\frac{1}{N}}h_{V}(\sigma
_{i})
\end{equation*}

Then we translate their normalization to ours, i.e.,%
\begin{equation*}
\sigma _{i}\mapsto g_{i}\mapsto h(\sigma _{i})\text{,}
\end{equation*}%
\begin{equation*}
\theta _{V}=q^{N}\cdot \mathrm{id}_{V}
\end{equation*}

and%
\begin{equation*}
\theta _{V_{\lambda }}=q^{\kappa _{\lambda }+nN}\cdot \mathrm{id}_{V\lambda }%
\text{.}
\end{equation*}

Then it is quite easy to get%
\begin{equation*}
\theta _{V_{\lambda }}=q^{\kappa _{\lambda }+2N(n-2f)}\cdot \mathrm{id}%
_{V\lambda }\text{,}
\end{equation*}

Now we can write down the explicit formula for orthogonal quantum group
invariants as follows
\begin{equation}
W_{A^{1},\cdots ,A^{L}}^{SO}(\mathcal{L};q,t)=q^{-\overset{L}{\underset{%
\alpha =1}{\sum }}\kappa _{A^{\alpha }}w(\mathcal{K}_{\alpha })}t^{-\overset{%
L}{\underset{\alpha =1}{\sum }}|A^{\alpha }|w(\mathcal{K}_{\alpha })}\cdot
\mathrm{tr}_{V_{A^{i_{1}}}\otimes \cdots \otimes V_{A^{i_{m}}}}(h(\beta ))
\end{equation}%
for all sufficiently integers $N$. In particular, when the link is trivial
with $L$ components, the quantum group invariant is the product of quantum
dimensions
\begin{equation}
W_{A^{1},\cdots ,A^{L}}^{SO}(\bigcirc ^{L};q,q^{2N})=\prod_{\alpha
=1}^{L}\dim _{q}(V_{A^{\alpha }})\text{.}  \label{E:qdim of unknots}
\end{equation}

The quantum dimension is computed by Wenzl \cite{Wen2}, which we quote it
here. Let $\lambda $ be a partition. We also identify it to the
corresponding Young diagram. For each pair of positive integers $(i,j)$,
define%
\begin{equation*}
h(i,j)=\lambda _{i}+\lambda _{j}^{\prime }-i-j+1
\end{equation*}%
to be the \emph{hook length}, where $\lambda ^{\prime }$ is the transposed
Young diagram of $\lambda $. Also define%
\begin{equation*}
d(i,j)=\left\{
\begin{array}{c}
\lambda _{i}+\lambda _{j}-i-j+1 \\
-\lambda _{i}^{\prime }-\lambda _{j}^{\prime }+i+j-1%
\end{array}%
\right.
\begin{array}{c}
i\leq j \\
i>j%
\end{array}%
\end{equation*}

\begin{theorem}[Wenzl \protect\cite{Wen2}]
\label{Wenzl} Let $\lambda $ be a Young diagram with $m$ rows and let $%
\mathcal{Q}_{\lambda }(t,q)$ be the rational function given by
\begin{equation}
\mathcal{Q}_{\lambda }(t,q)=\prod_{(j,j)\in \lambda }(1+\frac{tq^{\lambda
_{j}-\lambda _{j}^{\prime }}-t^{-1}q^{\lambda _{j}^{\prime }-\lambda _{j}}}{%
[h(j,j)]_{q}})\underset{i\neq j}{\prod_{(i,j)\in \lambda }}\frac{%
tq^{d(i,j)}-t^{-1}q^{-d(i,j)}}{[h(i,j)]_{q}}
\end{equation}%
Then the quantum trace $\dim _{q}V_{\lambda }$ of the representation of $%
U_{q}(\mathfrak{so}(2N+1))$ corresponding to $\lambda $ equal to $\mathcal{Q}%
_{\lambda }(q^{2N},q)$ for all $N>|\lambda |$.
\end{theorem}

In the above expression, if we fix $t$ and let $q$ tends to $1$, the pole
order of $\mathcal{Q}_{\lambda }(t,q)$ is $|\lambda |$, the number of boxes
in the Young diagram. The poles order at $q=1$ of the quantum group
invariant of unknots in (\ref{E:qdim of unknots}) is $||\overrightarrow{A}||=%
\underset{\alpha =1}{\overset{L}{\sum }}|A^{\alpha }|$.

The special value $sb_{A}(q^{1-2N},q^{3-2N},\cdots
,q^{-1},1,q,...,q^{2N-3},q^{2N-1})=\mathcal{Q}_{\lambda }(q,q^{2N})$ is the
quantum dimension $\dim _{q}(V_{A})$, denoted also by $sb_{A}(q,t)$. Here we
only evaluate the function in the variables%
\begin{equation*}
z_{-N},z_{1-N},\cdots ,z_{-1},z_{0},z_{1},\cdots ,z_{N-1},z_{N},
\end{equation*}%
and setting all the rest variables equal to zero.

The quantum dimension of small partitions can be found in Section \ref{sec10}%
, where we use the symbol of type-B Schur function $sb_{A}(q,t)$.

Similar to type-$A$ Schur function, type-$B$ Schur function has the
following expansion

\begin{equation*}
sb_{\lambda }(z_{-N},z_{1-N},\cdots ,z_{-1},z_{0},z_{1},\cdots
,z_{N-1},z_{N})=\underset{\ell (\mu )\leq 2N+1}{\sum_{\mu \models n}}\dim
(p_{\lambda }V^{\otimes n}\cap M^{\mu })\cdot \prod_{i=-N}^{N}z_{i}^{\mu
_{(i+N+1)}}\text{,}
\end{equation*}%
where $M^{\mu }$ is called the permutation module defined by

\begin{equation*}
M^{\mu }=\left\{ v\in V^{\otimes n}|H_{i}(v)=q^{\mu _{i}-\mu _{i+1}}v\text{
for }i=1,2,...,N-1\text{ and }H_{N}(v)=q^{\mu _{N}}v\right\} \text{,}
\end{equation*}%
and $\dim (p_{\lambda }V^{\otimes n}\cap M^{\mu })$ is called the \emph{%
Kostka number}.

When all the representations $A^{1},\cdots ,A^{L}$ are the natural
representation of $\mathfrak{so}(2N+1)$ on $\mathbb{C}^{2N+1}$, i.e., the
partitions $A^{\alpha }$ all equal to $(1)$, the invariant%
\begin{equation*}
W_{A^{1},\cdots ,A^{L}}^{SO}(\mathcal{L},q,t)=t^{2lk(\mathcal{L})}(1+\frac{%
t-t^{-1}}{q-q^{-1}})K_{\mathcal{L}}(q,t)
\end{equation*}%
for the Kauffman polynomial $K_{\mathcal{L}}(q,t)$, where we normalized the
Kauffman polynomials such that $K_{\bigcirc }(q,t)=1$. The orthogonal group
invariants $W_{A^{1},\cdots ,A^{L}}^{SO}(\mathcal{L};q,t)$ for general $%
A^{\alpha }$ are also called colored Kauffman polynomials.

It is normally very hard to calculate these quantum group invariants.
Anyway, we can simplify the computation a lot with the help of cabling
technique discussed in the next subsection.

\subsection{Cabling technique}

The following Lemma proved by Xiao-Song Lin and Hao Zheng \cite{LZ} reduce
the study of quantum group invariants of arbitrary representations to the
study of the links and minimal idempotents.

\begin{lemma}[\protect\cite{LZ}, Lemma 3.3]
\label{Lemma3.1}Let $\beta \in B_{m}$ be a braid and $p_{\alpha }\in
C_{d_{\alpha }},\alpha =1,\cdots ,L$ be $L$ minimal idempotents
corresponding to the irreducible representations $V_{A^{1}},\cdots
,V_{A^{L}} $, where $A^{\alpha }$ denote the partition of $|A^{\alpha
}|=d_{\alpha }$ labelling $V_{A^{\alpha }}$. Denote $\overrightarrow{d}%
=(d_{1},\cdots ,d_{L}) $ and let $i_{1},\cdots ,i_{m}$ be integers such that
$i_{k}=\alpha $ if the $k$-th strand of $\beta $ belongs to the $\alpha $-th
component of $\mathcal{L}$. Let $\beta _{\overrightarrow{d}}$ be the cabling
braid of $\beta $, replacing the $k$-th strand of $\beta $ by $d_{i_{k}}$
parallel ones. Then
\begin{equation}
\mathrm{tr}_{V_{A^{i_{1}}}\otimes \cdots \otimes V_{A^{i_{m}}}}h(\beta )=%
\mathrm{tr}_{V^{\otimes n}}[h(\beta \mathcal{_{\overrightarrow{d}}})\cdot
(p_{i_{1}}\otimes \cdots \otimes p_{i_{m}})],
\end{equation}%
where $n=d_{i_{1}}+d_{i_{2}}+\cdot \cdot \cdot +d_{i_{m}}$.
\end{lemma}

One immediately gets the following lemmas proved in \cite{LZ} and
reformulated into the setting of orthogonal group case.

\begin{lemma}
\label{Lemma3.2}For any element $y\in C_{n}$,
\begin{equation}
\mathrm{tr}_{V^{\otimes n}}y=\sum_{k=0}^{[\frac{n}{2}]}\sum_{\lambda \vdash
n-2k}\zeta _{n}^{\lambda }(y)\cdot sb_{\lambda }(q^{1-2N},q^{3-2N},\cdots
,q^{-1},1,q,\cdots ,q^{2N-3},q^{2N-1}).
\end{equation}%
For any braid $\beta \in B_{m}$, taking $y=h(\beta _{\overrightarrow{d}%
})\cdot (p_{i_{1}}\otimes p_{i_{2}}\otimes \cdot \cdot \cdot \otimes
p_{i_{m}})$, where the closure of $\beta $ is the link $\mathcal{L}$. The
setting is same as that in \ref{Lemma3.1} after replacing $q^{2N}$ by $t$,
we have
\begin{equation}
W_{\overrightarrow{A}}^{SO}(\mathcal{L},q,t)=q^{\overset{L}{\underset{\alpha
=1}{\sum }}-\kappa _{A^{\alpha }}w({\mathcal{K}_{\alpha }})}t^{-\overset{L}{%
\underset{\alpha =1}{\sum }}|A^{\alpha }|w(\mathcal{K}_{\alpha })}\cdot
\sum_{k=0}^{[\frac{n}{2}]}\sum_{\lambda \vdash n-2k}\zeta _{n}^{\lambda
}(h(\beta _{\overrightarrow{d}})\cdot (p_{i_{1}}\otimes p_{i_{2}}\otimes
\cdot \cdot \cdot \otimes p_{i_{m}}))\cdot \mathcal{Q}_{\lambda }(q,t)\text{,%
}  \label{E:Colored Kauffman}
\end{equation}%
where $n=|A^{i_{1}}|+\cdot \cdot \cdot +|A^{i_{m}}|$.
\end{lemma}

\subsection{An Explicit Formula of Colored Kauffman Polynomials for Torus
Links}

The coefficients $\zeta _{n}^{\lambda }(h(\beta _{\overrightarrow{d}})\cdot
(p_{i_{1}}\otimes p_{i_{2}}\otimes \cdot \cdot \cdot \otimes p_{i_{m}}))$ in
(\ref{E:Colored Kauffman}) are usually hard to compute. However, they are
computable for torus links. The torus link $T(r,k)$ is the closure of $%
(\delta _{r})^{k}=(\sigma _{1}\cdots \sigma _{r-1})^{k}$. It is a knot if
and only if $(r,k)=1$. For example, $T(2,3)$ is the trefoil knot, and $%
T(2,2) $ is the Hopf link. We developed the following method in this
subsection based on the work of Lin-Zheng \cite{LZ}.

\begin{lemma}
\label{Lemma3.5} For each partition $\lambda \vdash (n-2f)$ where $%
f=0,1,\cdots ,[\frac{n}{2}]$, we have
\begin{equation}
h((\delta _{n})^{n})\cdot p_{\lambda }=q^{\kappa _{\lambda }-4fN}\cdot
p_{\lambda }=q^{\kappa _{\lambda }}t^{-2f}\cdot p_{\lambda }.
\end{equation}
\end{lemma}

\begin{proof}
Again write $V$ for the standard representation of $U_{q}(\mathfrak{so}%
(2N+1))$ on the vector space $\mathbb{C}^{2N+1}$.

From Lemma \ref{AM type lemma}, for each partition $\lambda \vdash n-2f$
with $\ell (\lambda )\leq N$, one has%
\begin{equation*}
\theta _{V_{\lambda }}=q^{\kappa _{\lambda }+2N(n-2f)}\cdot \mathrm{id}%
_{V_{\lambda }}\text{,}
\end{equation*}

Substitute the above formula to the following formula proved in Lemma 3.2 of
\cite{LZ}
\begin{equation*}
(\theta _{V}^{\otimes n}\cdot h((\delta _{n})^{n})\cdot p_{\lambda }=\theta
_{V_{\lambda }}\cdot p_{\lambda }
\end{equation*}%
and the result follows.
\end{proof}

In the following, we assume $z_{0}=1$ and $z_{-n}z_{n}=1$ for all positive
integer $n=1,2,\cdots ,N$, i.e, the matrix $\text{diag}(z_{-N},z_{1-N},%
\cdots z_{-1},z_{0},z_{1},\cdots ,z_{N-1},z_{N})$ is a generic element in
the maximal torus of $SO(2N+1,\mathbb{C})$. Let the constants $\widetilde{c}%
_{\overrightarrow{A}}^{\lambda }$ be the rational number determined by
equations
\begin{equation}
\prod_{\alpha =1}^{L}sb_{A^{\alpha }}(z_{-N}^{r},\cdots
,z_{-1}^{r},z_{0}^{r},z_{1}^{r},\cdots
,z_{N}^{r})=\sum_{f=0}^{[rn/2]}\sum_{\lambda \vdash rn-2f}\widetilde{c}_{%
\overrightarrow{A}}^{\lambda }\cdot sb_{\lambda }(z_{-N},\cdots
,z_{-1},z_{0},z_{1},\cdots ,z_{N}).  \label{E:3.11}
\end{equation}

\begin{theorem}
\label{Thm3.6} Let $\mathcal{L}$ be the torus link $T(rL,kL)$ with $r$,$k$
relatively prime. $A^{\alpha }$ is a partition of $d_{\alpha }$ for each $%
\alpha =1,2,\cdots ,L$ and $n=d_{1}+d_{2}+\cdots +d_{L}$. Then
\begin{equation}
W_{\overrightarrow{A}}^{SO}(\mathcal{L},q,t)=q^{-kr\overset{L}{\underset{%
\alpha =1}{\sum }}\kappa _{A^{\alpha }}}\cdot t^{-k(r-1)n}\cdot \sum_{f=0}^{[%
\frac{nr}{2}]}\sum_{\lambda \vdash (rn-2f)}\widetilde{c}_{\overrightarrow{A}%
}^{\lambda }\cdot q^{\frac{k\kappa _{\lambda }}{r}}t^{-\frac{2fk}{r}}\cdot
sb_{\lambda }(q,t)
\end{equation}
\end{theorem}

Theorem \ref{Thm3.6} gives an explicit formula of the orthogonal quantum
group invariants (colored Kauffman polynomials) of torus links in terms of
constants $\widetilde{c}_{\overrightarrow{A}}^{\lambda }$. In \cite{Ste},
Sebastien Stevan generalize this result to all classic gauge group and cable
knots. In Section \ref{sec5}, we use this formula to verify certain cases of
Conjecture \ref{Main Conj}. The proof of Theorem \ref{Thm3.6} follows from
the Cabling formula (\ref{E:Colored Kauffman}), Lemma \ref{Lemma3.5} and the
following Lemma \ref{Lemma3.7}.

\begin{lemma}
\label{Lemma3.7} Let $n=\Vert \overrightarrow{A}\Vert $, where $A^{\alpha
}\vdash d_{\alpha }$, and let $r$ and $k$ be two relatively prime positive
integers. Take $\beta \in B_{rn}$ to be the braid obtained by cabling the $%
(iL+j)$-th strand of $(\delta _{rL})^{kL}$ to $|A^{j}|$ parallel ones. For
each partition $\lambda \vdash (rn-2f)$, where $f=0,1,2,\cdots ,[\frac{rn}{2}%
]$, we have
\begin{equation}
\zeta _{rn}^{\lambda }(h(\beta )\cdot (p_{A^{1}}\otimes \cdot \cdot \cdot
\otimes p_{A^{L}})^{\otimes r})=\widetilde{c}_{\overrightarrow{A}}^{\lambda
}\cdot q^{-k\underset{\alpha =1}{\overset{L}{\sum }}\kappa _{A^{\alpha }}+%
\frac{k\kappa _{\lambda }}{r}}t^{-\frac{2kf}{r}}.
\end{equation}
\end{lemma}

\begin{proof}
Write $p_{\overrightarrow{A}}=p_{A^{1}}\otimes \cdots \otimes p_{A^{L}}$ and
let $\pi _{\lambda }$ be the unit of $C_{\lambda }$. Obviously $\pi
_{\lambda }$ is in the center of $C_{rn}$. A slightly non-obvious fact is
that $h(\beta )$ commutes with $p_{\overrightarrow{A}}^{\otimes n}$
following from the naturality of $\check{\mathcal{R}}$. Let
\begin{equation}
x_{\lambda }=\pi _{\lambda }\cdot h(\beta )\cdot p_{\overrightarrow{A}%
}^{\otimes r}
\end{equation}%
be a matrix in $C_{\lambda }$, whose trace is
\begin{equation}
\mathrm{tr}(x_{\lambda })=\zeta ^{\lambda }(h(\beta )\cdot p_{%
\overrightarrow{A}}^{\otimes r}).
\end{equation}

The Cabling of torus link has the nice property
\begin{equation}
h(\beta ^{r})=h((\delta _{rn})^{krn})\cdot (h((\delta
_{d_{1}})^{-kd_{1}})\otimes \cdots \otimes h((\delta
_{d_{L}})^{-kd_{L}}))^{\otimes r}.
\end{equation}%
The previous Lemma \ref{Lemma3.5} then imply
\begin{equation}
x_{\lambda }^{r}=\pi _{\lambda }\cdot h(\beta ^{r})\cdot p_{\overrightarrow{A%
}}^{\otimes r}=q^{-kr\overset{L}{\underset{\alpha =1}{\sum }}\kappa
_{A^{\alpha }}+k\kappa _{\lambda }}t^{-2kf}\cdot \pi _{\lambda }\cdot p_{%
\overrightarrow{A}}^{\otimes r}.
\end{equation}%
Thus the eigenvalues of $x_{\lambda }$ are either $0$ or $q^{-kr\overset{L}{%
\underset{\alpha =1}{\sum }}\kappa _{A^{\alpha }}+\frac{k\kappa _{\lambda }}{%
r}}t^{\frac{-2kf}{r}}$ times a $r$-th root of unity. Together with the fact
that $\mathrm{tr}(x_{\lambda })\in \mathbb{Q}(q,t)$, we see that $\mathrm{tr}%
(x_{\lambda })=a^{\lambda }\cdot q^{-k\overset{L}{\underset{\alpha =1}{\sum }%
}\kappa _{A^{\alpha }}+\frac{k\kappa _{\lambda }}{r}}t^{-\frac{2kf}{r}}$ for
some $a^{\lambda }\in \mathbb{Q}$ independent of $q$ and $t$.

The rest is to compute this rational number $a^{\lambda }$. Passing to the
limit $q\rightarrow 1$ and $t\rightarrow 1$. The element $h(\beta )$ reduce
to a permutation $\tau \in S_{rn}$ acting cyclicly on the $V^{\otimes n}$%
-factors of%
\begin{equation*}
V^{\otimes rn}=V^{\otimes n}\otimes \cdots \otimes V^{\otimes n}\text{.}
\end{equation*}%
\begin{align*}
& \sum_{f=0}^{[\frac{rn}{2}]}\sum_{\lambda \vdash (rn-2f)}a^{\lambda }\cdot
sb_{\lambda }(z_{-N},z_{1-N},\cdots ,z_{-1},z_{0},z_{1},\cdots
,z_{N-1},z_{N}) \\
=& \sum_{f=0}^{[\frac{rn}{2}]}\sum_{\lambda \vdash (rn-2f)}a^{\lambda }%
\underset{\ell (\mu )\leq 2N+1}{\sum_{\mu \models rn}}\dim (p_{\lambda
}V^{\otimes rn}\cap M^{\mu })\cdot \prod_{i=-N}^{N}z_{i}^{\mu _{(i+N+1)}} \\
=& \underset{\ell (\mu )\leq 2N+1}{\sum_{\mu \models rn}}\mathrm{tr}(\tau
|_{p_{\overrightarrow{A}}^{\otimes r}V^{\otimes rn}\cap M^{\mu }})\cdot
\prod_{i=-N}^{N}z_{i}^{\mu _{(i+N+1)}} \\
=& \underset{\ell (\mu )\leq 2N+1}{\sum_{\mu \models n}}\dim (p_{%
\overrightarrow{A}}V^{\otimes n}\cap M^{\mu })\cdot
\prod_{i=-N}^{N}z_{i}^{r\mu _{(i+N+1)}} \\
=& \prod_{\alpha =1}^{L}[\underset{\ell (\mu )\leq 2N+1}{\sum_{\mu \models
n_{\alpha }}}\dim (p_{A^{\alpha }}V^{\otimes n_{\alpha }}\cap M^{\mu })\cdot
\prod_{i=-N}^{N}z_{i}^{r\mu _{(i+N+1)}}] \\
=& \prod_{\alpha =1}^{L}sb_{A^{\alpha }}(z_{-N}^{r},z_{1-N}^{r},\cdots
,z_{-1}^{r},z_{0}^{r},z_{1}^{r},\cdots ,z_{N-1}^{r},z_{N}^{r}).
\end{align*}%
Compare to (\ref{E:3.11}), $a^{\lambda }=\widetilde{c}_{\overrightarrow{A}%
}^{\lambda }$.
\end{proof}

\begin{remark}
Similar computations by starting with $U_{q}(\mathfrak{sp}(2N))$ and $U_{q}(%
\mathfrak{so}(2N))$ lead to the same theorem for Kauffman polynomials. Thus
together with the type-A analog proved in \cite{LZ} Theorem 5.1, these
computations provide formulas of quantum group invariants of Torus links
associated to simple Lie-algebras of type A, B, C and D.
\end{remark}

\section{Orthogonal Chern-Simons Partition Function}

\label{sec4}

\subsection{Partition Function}

The \emph{orthogonal Chern-Simons partition function} of $\mathcal{L}$ is
defined by
\begin{equation}
Z_{CS}^{SO}(\mathcal{L};q,t)=\sum_{\overrightarrow{\mu }\in \mathcal{P}^{L}}%
\frac{pb_{\overrightarrow{\mu }}(\overrightarrow{z})}{\mathrm{z}_{%
\overrightarrow{\mu }}}\cdot \sum_{\overrightarrow{A}\in \widehat{Br}_{_{|%
\overrightarrow{\mu }|}}}\chi _{\overrightarrow{A}}(\gamma _{\overrightarrow{%
\mu }})W_{\overrightarrow{A}}^{SO}(\mathcal{L};q,t)\text{.}
\end{equation}

The above definition is motivated from physicists' path integral approach
\cite{BR}, and it is different from the definition given by Equation (4.10)
in \cite{BR}. Unlike the $SU(N)$ Chern-Simons partition function, the above $%
Z_{CS}^{SO}(\mathcal{L};q,t)$ can not be simplified to%
\begin{equation}
Z_{CS}^{SO}(\mathcal{L};q,t)=\sum_{\overrightarrow{A}\in \mathcal{P}^{L}}W_{%
\overrightarrow{A}}^{SO}(\mathcal{L};q,t)sb_{\overrightarrow{A}}(%
\overrightarrow{z})\text{.}
\end{equation}

Because orthogonality of type-A Schur function fails in type-B case \cite%
{Ram1, Ram2, Ram3}.

Define the \emph{free energy} as follows

\begin{equation}
F^{SO}(\mathcal{L};q,t)=\log Z_{CS}^{SO}(\mathcal{L};q,t)\text{.}
\end{equation}

The partition function of unknots with $L$ components can be computed
explicitly (See Proposition \ref{Prop10.2} for detail). In fact we have the
following expression for the free energy
\begin{equation}
F^{SO}(\bigcirc ^{L};q,t)=\sum_{n=1}^{+\infty }\frac{1}{n}\cdot (1+\frac{%
t^{n}-t^{-n}}{q^{n}-q^{-n}})\cdot \sum_{\alpha =1}^{L}pb_{n}(z_{\alpha })%
\text{,}
\end{equation}

\subsection{Reformulated Invariants}

The reformulated link invariants are rational functions $g_{\overrightarrow{A%
}}(q,t)\in \mathbb{C}(q,t)$ determined by the expansion
\begin{equation}
F^{SO}(\mathcal{L};q,t)=\sum_{d=1}^{\infty }\sum_{\overrightarrow{\mu }\neq
\overrightarrow{0}}\frac{1}{d}g_{\overrightarrow{\mu }}(q^{d},t^{d})\prod_{%
\alpha =1}^{L}pb_{\mu ^{\alpha }}((z_{\alpha })^{d})\text{.}
\end{equation}

As in \cite{LM}, define the operator $\psi _{d}$ by
\begin{equation}
\psi _{d}\circ F(q,t;pb(\overrightarrow{z}))=F(q^{d},t^{d};pb(%
\overrightarrow{z}^{d}))\text{.}
\end{equation}%
Then define the \emph{plethystic exponential} \cite{GK}
\begin{equation}
\mathrm{Exp}(F)=\exp (\sum_{k=1}^{+\infty }\frac{\psi _{k}}{k}\circ F)
\end{equation}%
and its inverse
\begin{equation}
\mathrm{Log}(F)=\sum_{k=1}^{+\infty }\frac{\mu (k)}{k}\log (\psi _{k}\circ F)%
\text{,}
\end{equation}%
where $\mu (k)$ is the M\"{o}bius function. In terms of these operators, we
could write
\begin{equation}
Z_{CS}^{SO}(\mathcal{L};q,t)=\mathrm{Exp}(\sum_{\overrightarrow{\mu }\neq
\overrightarrow{0}}g_{\overrightarrow{\mu }}(q,t)\prod_{\alpha
=1}^{L}pb_{\mu ^{\alpha }}(z_{\alpha }))\text{.}
\end{equation}

If we expand the partition function
\begin{equation}
Z_{CS}^{SO}(\mathcal{L};q,t)=1+\sum_{\overrightarrow{\mu }\neq
\overrightarrow{0}}Z_{\overrightarrow{\mu }}^{SO}pb_{\overrightarrow{\mu }}(%
\overrightarrow{z})\text{,}
\end{equation}

where $Z_{\overrightarrow{\mu }}^{SO}(\mathcal{L};q,t)=\underset{%
\overrightarrow{A}\in \widehat{Br}_{_{|\overrightarrow{\mu }|}}}{\sum }\frac{%
\chi _{\overrightarrow{A}}(\gamma _{\overrightarrow{\mu }})}{\mathrm{z}_{%
\overrightarrow{\mu }}}W_{\overrightarrow{A}}^{SO}(\mathcal{L};q,t)$.

If we expand the free energy
\begin{equation}
F^{SO}(\mathcal{L};q,t)=\sum_{\overrightarrow{\mu }\neq \overrightarrow{0}%
}F_{\overrightarrow{\mu }}^{SO}pb_{\overrightarrow{\mu }}(\overrightarrow{z})%
\text{.}
\end{equation}%
From Lemma \ref{Lemma2.1} (Lemma 2.3 in \cite{LP}), we have
\begin{equation}
F_{\overrightarrow{\mu }}^{SO}=\underset{|\Lambda |=\overrightarrow{\mu }}{%
\sum_{\Lambda \in \mathcal{P}(\mathcal{P}^{L})}}\frac{(-1)^{\ell (\Lambda
)-1}\ell (\Lambda )!}{\ell (\Lambda )|Aut\Lambda |}Z_{\Lambda }^{SO}\text{.}
\label{E:4.11}
\end{equation}%
Clearly $F_{\overrightarrow{\mu }}^{SO}$ is a rational function of $q$ and $%
t $. The reformulated invariants then can be defined by
\begin{equation*}
g_{\overrightarrow{\mu }}(q,t)=\sum_{k|\overrightarrow{\mu }}\frac{\mu (k)}{k%
}F_{\overrightarrow{\mu }/k}^{SO}(q^{k},t^{k})\text{,}
\end{equation*}%
where $\mu (k)$ is the M\"{o}bius function.

\section{Orthogonal Labastida-Mari\~{n}o-Ooguri-Vafa Conjecture}

\label{sec5}

\subsection{Orthogonal LMOV Conjecture}

Now we can state the main conjecture of this paper, which is the analog of
LMOV conjecture for orthogonal Chern-Simons theory.

\begin{conjecture}[Orthogonal LMOV]
\label{Main Conj} The rational function $g_{\overrightarrow{\mu }}(q,t)\in
\mathbb{Q}(q,t)$ has the property that
\begin{equation*}
\frac{\mathrm{z}_{\overrightarrow{\mu }}(q-q^{-1})^{2}\cdot \lbrack g_{%
\overrightarrow{\mu }}(q,t)-g_{\overrightarrow{\mu }}(q,-t)]}{2\overset{L}{%
\underset{\alpha =1}{\prod }}\overset{\ell (\mu ^{\alpha })}{\underset{i=1}{%
\prod }}(q^{\mu _{i}^{\alpha }}-q^{-\mu _{i}^{\alpha }})}\in \mathbb{Z}%
[q-q^{-1}][t,t^{-1}]\text{.}
\end{equation*}
\end{conjecture}

We may write the above (conjectured) polynomial as
\begin{equation*}
\frac{\mathrm{z}_{\overrightarrow{\mu }}(q-q^{-1})^{2}\cdot \lbrack g_{%
\overrightarrow{\mu }}(q,t)-g_{\overrightarrow{\mu }}(q,-t)]}{2\overset{L}{%
\underset{\alpha =1}{\prod }}\overset{\ell (\mu ^{\alpha })}{\underset{i=1}{%
\prod }}(q^{\mu _{i}^{\alpha }}-q^{-\mu _{i}^{\alpha }})}=\sum_{g\in \mathbb{%
Z}_{+}/2}\sum_{\beta \in \mathbb{Z}}N_{\overrightarrow{\mu },g,\beta
}(q-q^{-1})^{2g}t^{\beta }\text{.}
\end{equation*}%
The integers $N_{\overrightarrow{\mu },g,\beta }$ (or their linear
combinations, depends on a choice of basis) are explained as BPS numbers in
string theory \cite{BFM, Mar}, and these numbers should coincide with the
BPS numbers calculated by the Gromov-Witten theory (see for example \cite%
{Pan, Pen}). Physicists predict that the Gromov-Witten theory of orientifold
are dual to the type-B Chern-Simons theory \cite{BFM}, i.e, the partition
functions of these two theories coincide up to some normalization. Thus the
integers $N_{\overrightarrow{\mu },g,\beta }$ are conjecturally equal to
some linear combinations of intersection numbers on the moduli space of
stable maps from curves into un-oriented manifolds. However, a mathematical
construction of such moduli space is still lacking.

\begin{remark}
Actually the anti-symmetrization $\frac{g_{\overrightarrow{\mu }}(q,t)-g_{%
\overrightarrow{\mu }}(q,-t)}{2}$ is not necessary for some knots/links.
Thus if we expand $\frac{\mathrm{z}_{\overrightarrow{\mu }}(q-q^{-1})^{2}g_{%
\overrightarrow{\mu }}(q,t)}{\overset{L}{\underset{\alpha =1}{\prod }}%
\overset{\ell (\mu ^{\alpha })}{\underset{i=1}{\prod }}(q^{\mu _{i}^{\alpha
}}-q^{-\mu _{i}^{\alpha }})}$, then we may get more integer coefficients.
Readers may find the proof of most theorems except for some cases of Theorem %
\ref{Thm1.7} are still valid for $\frac{\mathrm{z}_{\overrightarrow{\mu }%
}(q-q^{-1})^{2}g_{\overrightarrow{\mu }}(q,t)}{\overset{L}{\underset{\alpha
=1}{\prod }}\overset{\ell (\mu ^{\alpha })}{\underset{i=1}{\prod }}(q^{\mu
_{i}^{\alpha }}-q^{-\mu _{i}^{\alpha }})}$.
\end{remark}

\begin{remark}
Physicists Bouchard-Florea-Mari\~{n}o \cite{BFM} and more recently Mari\~{n}%
o \cite{Mar} have similar conjectures for a different partition functions.
It seems none of these definitions are equivalent to the definition given
here. However, it is pointed out by Mari\~{n}o that the reformulated
invariants may coincide for some examples of torus knots. But we obtain
different integer invariants for torus links. Thus relations between these
conjectures are still unclear at the present time. It shows that
anti-symmetrization process is not necessary for some links and knots.
Anyway, in next subsection, we will leave the integer coefficient invariants
of torus knots and links before anti-symmetrization for interested readers
to investigate the relationship between the conjecture proposed in \cite%
{BFM, Mar} and ours.
\end{remark}

To describe the behavior of the reformulated invariants near $q=1$, let $%
q=e^u$ and embed $\mathbb{Q}(q,t)$ into $\mathbb{Q}(t)((u))$. Denote $%
val_u(F_{\overrightarrow{\mu}}^{SO})$ the valuation of $F_{\overrightarrow{%
\mu}}^{SO}$ in the valuation field $\mathbb{Q}(t)((u))$. This valuation is
the same as the zero order of the rational function $F_{\overrightarrow{\mu}%
}^{SO}$ at $q=1$.

\begin{conjecture}[Degree]
\label{Degree Conj} The valuation $val_{u}(F_{\overrightarrow{\mu }}^{SO})$
is greater than or equal to $\ell (\overrightarrow{\mu })-2$.
\end{conjecture}

The Conjecture \ref{Degree Conj} is indeed claiming that all the
coefficients of lower degree vanish. It is not a consequence of Conjecture %
\ref{Main Conj}. We will see later that this vanishing is closely related to
formulas of Lickorish-Millett type. This kind of degree conjecture is also
an important part of the Liu-Peng's paper \cite{LP}. We will prove
Conjecture \ref{Degree Conj} in Section \ref{sec8} and \ref{sec9}.

\subsection{Torus Link as Supporting Examples of Main Conjecture}

In this subsection, we verify the orthogonal LMOV conjecture by testing
torus links and knots for small partitions.

In addition, several examples of torus links and knots of type $T(2,k)$
suggest that the anti-symmetrization of the reformulated invariants $g_{%
\overrightarrow{\mu }}(q,t)$ in the Conjecture \ref{Main Conj} is necessary.
In the following, we will denote $q-q^{-1}$ by $z$ for simplicity. We
compute the colored Kauffman polynomials for these examples in Section \ref%
{sec10} (Appendix). For tables of integer coefficients $N_{\overrightarrow{%
\mu },g,\beta }$ of these torus links and knots, please refer to Section \ref%
{sec10} (Appendix).

Example 1: Take $r=1$, the torus link $T(2,2k)$ has $2$ components.

Case A: Consider the partition $(1),(1)$ for link $T(2,2k)$

Denote $W_{(n)}(unknot)$ by $W_{(n)}$ in the following computations, where $%
n\in \mathbb{Z}_{\geq 0}$.

It is easy to verify that%
\begin{align*}
\mathrm{z}_{(1),(1)}g_{(1),(1)}& =W_{(1),(1)}-W_{(1)}^{2} \\
& =q^{2k}sb_{(2)}+q^{-2k}sb_{(1,1)}+t^{-2k}-sb_{(1)}^{2} \\
& =\left( \frac{q^{2k+1}+q^{-2k-1}}{q^{1}+q^{-1}}-1\right) t^{2}+\left(
\frac{q^{2k-1}+q^{-2k+1}}{q+q^{-1}}-1\right) t^{-2}-\left( \frac{q^{k}-q^{-k}%
}{q-q^{-1}}\right) ^{2}+t^{-2k}\text{.}
\end{align*}%
Thus all the integer invariant numbers $N_{\overrightarrow{\mu },g,\beta }=0$%
.

Case B: Consider the partition $(1,1),(1)$ for link $T(2,2k)$:

\begin{align*}
\mathrm{z}_{(1,1),(1)}g_{(1,1),(1)}&
=2(Z_{(1,1),(1)}-Z_{(1),(1)}Z_{(1)}-Z_{(1,1)}Z_{(1)}+Z_{(1)}^{3}) \\
&
=W_{(2),(1)}+W_{(1,1),(1)}+W_{(1)}-2W_{(1),(1)}W_{(1)}-(W_{(2)}+W_{(1,1)}+1)W_{(1)}+2W_{(1)}^{3}
\\
&
=q^{4k}sb_{(3)}+q^{-2k}sb_{(2,1)}+q^{-2k}t^{-2k}sb_{(1)}+q^{2k}sb_{(2,1)}+q^{-4k}sb_{(1,1,1)}+q^{2k}t^{-2k}sb_{(1)}
\\
&
-2(q^{2k}sb_{(2)}+q^{-2k}sb_{(1,1)}+t^{-2k})sb_{(1)}+2sb_{(1)}^{3}-(sb_{(2)}+sb_{(1,1)})sb_{(1)}%
\text{.}
\end{align*}%
It is interesting that the rational function $\dfrac{(q-q^{-1})^{2}}{%
(q-q^{-1})^{3}}\mathrm{z}_{(1,1),(1)}g_{(1,1),(1)}(q,t)$ is already in the
ring $\mathbb{Z}[t,t^{-1}][q-q^{-1}]$ without anti-symmetrization. We list
the first three in this family:
\begin{align*}
\frac{(q-q^{-1})^{2}\mathrm{z}_{(1,1),(1)}g_{(1,1),(1)}(T(2,2))}{%
(q-q^{-1})^{3}}& =(-t^{-3}+3t^{-1}-3t+t^{3})+(t^{-2}-2+t^{2})z \\
\frac{(q-q^{-1})^{2}\mathrm{z}_{(1,1),(1)}g_{(1,1),(1)}(T(2,4))}{%
(q-q^{-1})^{3}}& =(-4t^{-5}+4t^{-3}+12t^{-1}-20t+8t^{3}) \\
& +(4t^{-4}+4t^{-2}-20+12t^{2})z \\
& +(-t^{-5}+t^{-3}+3t^{-1}-9t+6t^{3})z^{2} \\
& +(t^{-4}+t^{-2}-9+7t^{2})z^{3}+(-t+t^{3})z^{4}+(-1+t^{2})z^{5} \\
\frac{(q-q^{-1})^{2}\mathrm{z}_{(1,1),(1)}g_{(1,1),(1)}(T(2,6))}{%
(q-q^{-1})^{3}}& =(-9t^{-7}+9t^{-5}-3t^{-3}+45t^{-1}-72t+30t^{3}) \\
& +(9t^{-6}+27t^{-2}-90+54t^{2})z \\
& +(-6t^{-7}+6t^{-5}-t^{-3}+39t^{-1}-93t+55t^{3})z^{2} \\
& +(6t^{-6}+27t^{-2}-114+81t^{2})z^{3} \\
& +(-t^{-7}+t^{-5}+11t^{-1}-47t+36t^{3})z^{4} \\
& +(t^{-6}+9t^{-2}-55+45t^{2})z^{5} \\
& +(t^{-1}-11t+10t^{3})z^{6}+(-12+t^{-2}+11t^{2})z^{7} \\
& +(-t+t^{3})z^{8}+(-1+t^{2})z^{9}\text{.}
\end{align*}

Please see Section \ref{sec10} for the table of integers $N_{\overrightarrow{%
\mu },g,\beta }$ after anti-symmetrization.

The conjectural prediction on $g_{(1,1),(1)}$ is also proved in Section \ref%
{sec7}. Next we compute $g_{(2),(1)}(T(2,2k))$, which will not be covered by
any proof in following sections.

Case C: Consider the partition $(2),(1)$ for link $T(2,2k)$%
\begin{align*}
\mathrm{z}_{(2),(1)}g_{(2),(1)}& =2(Z_{(2),(1)}-Z_{(2)}Z_{(1)}) \\
& =(W_{(2),(1)}-W_{(1,1),(1)}+W_{(1)})-W_{(1)}(W_{(2)}-W_{(1,1)}+1) \\
&
=(q^{4k}sb_{(3)}+q^{-2k}sb_{(2,1)}+q^{-2k}t^{-2k}sb_{(1)}-q^{2k}sb_{(2,1)}-q^{-4k}sb_{(1,1,1)}-q^{2k}t^{-2k}sb_{(1)}+sb_{(1)})
\\
& -sb_{(1)}(sb_{(2)}-sb_{(1,1)}+1)
\end{align*}%
The rational function $\dfrac{(q-q^{-1})^{2}}{(q-q^{-1})(q^{2}-q^{-2})}%
\mathrm{z}_{(2),(1)}g_{(2),(1)}(q,t)$ is also in the ring $\mathbb{Z}%
[t,t^{-1}][q-q^{-1}]$ without anti-symmetrization. We list the first three
in this family:
\begin{align*}
\frac{(q-q^{-1})^{2}\mathrm{z}_{(2),(1)}g_{(2),(1)}(T(2,2))}{%
(q-q^{-1})(q^{2}-q^{-2})}& =(t^{-3}-t^{-1}-t+t^{3})+(-t^{-2}+t^{2})z \\
\frac{(q-q^{-1})^{2}\mathrm{z}_{(2),(1)}g_{(2),(1)}(T(2,4))}{%
(q-q^{-1})(q^{2}-q^{-2})}&
=(2t^{-5}-2t^{-3}+2t^{-1}-6t+4t^{3})+(-2t^{-4}+2t^{-2}-6+6t^{2})z \\
& +(t^{-5}-t^{-3}+t^{-1}-5t+4t^{3})z^{2}+(-t^{-4}+t^{-2}-5+5t^{2})z^{3} \\
& +(-t+t^{3})z^{4}+(-1+t^{2})z^{5} \\
\frac{(q-q^{-1})^{2}\mathrm{z}_{(2),(1)}g_{(2),(1)}(T(2,6))}{%
(q-q^{-1})(q^{2}-q^{-2})}& =(3t^{-7}-3t^{-5}-t^{-3}+9t^{-1}-18t+10t^{3}) \\
& +(-3t^{-6}+9t^{-2}-24+18t^{2})z \\
& +(4t^{-7}-4t^{-5}-t^{-3}+15t^{-1}-39t+25t^{3})z^{2} \\
& +(-4t^{-6}+15t^{-2}-50+39t^{2})z^{3} \\
& +(t^{-7}-t^{-5}+7t^{-1}-29t+22t^{3})z^{4} \\
& +(-t^{-6}+7t^{-2}-35+29t^{2})z^{5} \\
& +(t^{-1}-9t+8t^{3})z^{6}+(\frac{1}{t^{2}}-10+9t^{2})z^{7} \\
& +(-t+t^{3})z^{8}+(-1+t^{2})z^{9}\text{.}
\end{align*}

Please see Section \ref{sec10} for the table of integers $N_{\overrightarrow{%
\mu },g,\beta }$ after anti-symmetrization.

The behavior of $g_{(2),(2)}(T(2,2k);q,t)$ is much different from the above
three examples. It is the first example that the multi-cover contribution
must be taken into account

Case D: Consider the partition $(2),(2)$ for link $T(2,2k)$%
\begin{align*}
\mathrm{z}_{(2),(2)}g_{(2),(2)}& =4F_{(2),(2)}(q,t)-2F_{(1),(1)}(q^{2},t^{2})
\\
&
=4Z_{(2),(2)}(q,t)-4Z_{(2)}^{2}(q,t)-2(Z_{(1),(1)}(q^{2},t^{2})-4Z_{(1)}^{2}(q^{2},t^{2}))
\\
&
=W_{(2),(2)}-2W_{(2),(1,1)}+W_{(1,1),(1)}-W_{(2)}^{2}-W_{(1,1)}^{2}+2W_{(2)}W_{(1,1)}
\\
& -2W_{(1),(1)}(q^{2},t^{2})+2W_{(1)}^{2}(q^{2},t^{2}) \\
&
=q^{8k}sb_{(4)}+(1-2q^{4k})sb_{(3,1)}+(q^{4k}+q^{-4k})sb_{(2,2)}+(1-2q^{-4k})sb_{(2,1,1)}+q^{-8k}sb_{(1,1,1,1)}
\\
&
+(q^{-2k}-2q^{2k}+q^{6k})t^{-2k}sb_{(2)}+(q^{-6k}-2q^{-2k}+q^{2k})t^{-2k}sb_{(1,1)}+(q^{4k}+q^{-4k})t^{-4k}
\\
&
-(sb_{(2)}-sb_{(1,1)})^{2}-2(q^{4k}sb_{(2)}(q^{2},t^{2})+q^{-4k}sb_{(1,1)}(q^{2},t^{2})+t^{-4k})+2sb_{(1)}^{2}(q^{2},t^{2})%
\text{.}
\end{align*}

The rational function $\dfrac{(q-q^{-1})^{2}}{(q^{2}-q^{-2})^{2}}\mathrm{z}%
_{(2),(2)}g_{(2),(2)}(q,t)$ is not in the ring $\mathbb{Z}%
[t,t^{-1}][q-q^{-1}]$ and anti-symmetrization is necessary here. Please see
Section \ref{sec10} for the table of integers $N_{\overrightarrow{\mu }%
,g,\beta }$.

Case E: Consider the partition $(3),(1)$ for link $T(2,2k)$%
\begin{align*}
\mathrm{z}_{(3),(1)}g_{(3),(1)}& =3(Z_{(3),(1)}-Z_{(3)}Z_{(1)}) \\
&
=(W_{(3),(1)}-W_{(2,1),(1)}+W_{(1,1,1),(1)})-W_{(1)}(W_{(3)}-W_{(2,1)}+W_{(1,1,1)})
\\
&
=(q^{6k}-1)sb_{(4)}+(q^{-2k}-q^{4k})sb_{(3,1)}+(q^{2k}-q^{-4k})sb_{(2,1,1)}+(q^{-6k}-1)sb_{(1,1,1,1)}
\\
& +(q^{-4k}-q^{2k})t^{-2k}sb_{(2)}+(q^{4k}-q^{-2k})t^{-2k}sb_{(1,1)}\text{.}
\end{align*}
The rational function $\dfrac{(q-q^{-1})^{2}}{(q^{3}-q^{-3})(q-q^{-1})}%
\mathrm{z}_{(3),(1)}g_{(3),(1)}(q,t)$ is not in the ring $\mathbb{Z}%
[t,t^{-1}][q-q^{-1}]$ and anti-symmetrization is necessary here. Please see
Section \ref{sec10} for the table of integers $N_{\overrightarrow{\mu }%
,g,\beta }$ after anti-symmetrization.

Example 2: The torus knots $T(2,k)$, where $k$ is an odd integer.

Case A: Consider the partition $(1,1)$ for link $T(2,k)$

The following calculation provides an example of the case proved in Section %
\ref{sec7}. We have%
\begin{align*}
\mathrm{z}_{(1,1)}g_{(1,1)}& =2g_{(1,1)} \\
& =W_{(2)}+W_{(1,1)}-W_{(1)}^{2}(q,t)+1 \\
&
=t^{-2k}(q^{2k}sb_{(4)}-q^{-2k}sb_{(3,1)}+q^{-4k}sb_{(2,2)}+q^{-3k}t^{-k}sb_{(2)}-q^{-5k}t^{-k}sb_{(1,1)}+q^{-4k}t^{-2k})
\\
&
+t^{-2k}(q^{4k}sb_{(2,2)}-q^{2k}sb_{(2,1,1)}+q^{-2k}sb_{(1,1,1,1)}+q^{5k}t^{-k}sb_{(2)}-q^{3k}t^{-k}sb_{(1,1)}+q^{4k}t^{-2k})
\\
& -t^{-2k}(q^{k}sb_{(2)}-q^{-k}sb_{(1,1)}+t^{-k})^{2}+1\text{.}
\end{align*}

It is interesting that The rational function $\dfrac{(q-q^{-1})^{2}}{%
(q-q^{-1})^{2}}\mathrm{z}_{(1,1)}g_{(1,1)}(q,t)$ is already in the ring $%
\mathbb{Z}[t,t^{-1}][q-q^{-1}]$ without anti-symmetrization.

We list the Trefoil knot in this family:
\begin{align*}
\frac{(q-q^{-1})^{2}\mathrm{z}_{(1,1)}g_{(1,1)}(T(2,3))}{(q-q^{-1})^{2}}&
=(t^{-12}+6t^{-10}-33t^{-8}+52t^{-6}-33t^{-4}+6t^{-2}+1) \\
& +(36t^{-11}-132t^{-9}+180t^{-7}-108t^{-5}+24t^{-3})z \\
& +(36t^{-12}-103t^{-10}+76t^{-8}+18t^{-6}-32t^{-4}+5t^{-2})z^{2} \\
& +(105t^{-11}-377t^{-9}+453t^{-7}-207t^{-5}+26t^{-3})z^{3} \\
& +(105t^{-12}-350t^{-10}+341t^{-8}-87t^{-6}-10t^{-4}+t^{-2})z^{4} \\
& +(112t^{-11}-450t^{-9}+494t^{-7}-165t^{-5}+9t^{-3})z^{5} \\
& +(112t^{-12}-441t^{-10}+440t^{-8}-110t^{-6}-t^{-4})z^{6} \\
& +(54t^{-11}-275t^{-9}+286t^{-7}-66t^{-5}+t^{-3})z^{7} \\
& +(54t^{-12}-274t^{-10}+274t^{-8}-54t^{-6})z^{8} \\
& +(12t^{-11}-90t^{-9}+91t^{-7}-13t^{-5})z^{9} \\
& +(12t^{-12}-90t^{-10}+90t^{-8}-12t^{-6})z^{10} \\
& +(t^{-11}-15t^{-9}+15t^{-7}-t^{-5})z^{11} \\
& +(t^{-12}-15t^{-10}+15t^{-8}-t^{-6})z^{12} \\
& +(-t^{-9}+t^{-7})z^{13}+(-t^{-10}+t^{-8})z^{14}\text{.}
\end{align*}

Please see Section \ref{sec10} for the table of integers $N_{\overrightarrow{%
\mu },g,\beta }$ after anti-symmetrization.

Case B: Consider the partition $(2)$ for link $T(2,k)$

The function%
\begin{align*}
\mathrm{z}_{(2)}g_{(2)}& =2g_{(2)} \\
& =W_{(2)}-W_{(1,1)}-W_{(1)}(q^{2},t^{2})+1 \\
&
=t^{-2k}[(q^{2k}-q^{-2k})sb_{(3,1)}+(q^{-4k}-q^{4k}-q^{2k}+q^{-2k})sb_{(2,2)}+(q^{2k}-q^{-2k})sb_{(2,1,1)}
\\
&
+((q^{-3k}-q^{5k})t^{-k}-q^{2k}+q^{-2k})sb_{(2)}+((-q^{-5k}+q^{3k})t^{-k}+q^{2k}-q^{-2k})sb_{(1,1)}
\\
& +(q^{-4k}-q^{4k}-1)t^{-2k}-q^{2k}+q^{-2k}]+1\text{.}
\end{align*}%
The rational function $\dfrac{(q-q^{-1})^{2}}{(q^{2}-q^{-2})}\mathrm{z}%
_{(2)}g_{(2)}(q,t)$ is not in the ring $\mathbb{Z}[t,t^{-1}][q-q^{-1}]$.
Please see Section \ref{sec10} for the table of integers $N_{\overrightarrow{%
\mu },g,\beta }$ after anti-symmetrization.

Example 3: Take $r=1$, the torus link $T(3,3k)$ has $3$ components.

Consider the partition $(2),(1),(1)$ for link $T(3,3k)$.

Denote $W_{(1),(1)}(T(2,2k))$ simply by $W_{(1),(1)}$ in the following
computations.

The function
\begin{align*}
\mathrm{z}_{(2),(1),(1)}g_{(2),(1),(1)}& =2g_{(2),(1),(1)} \\
& =W_{(2),(1),(1)}-W_{(1,1),(1),(1)}-W_{(1),(1)}W_{(2)}+W_{(1),(1)}W_{(1,1)}
\\
&
-2W_{(2),(1)}W_{(1)}+2W_{(1,1),(1)}W_{(1)}+2W_{(2)}W_{(1)}^{2}-2W_{(1,1)}W_{(1)}^{2}
\\
& =q^{10k}sb_{(4)}+(2q^{2k}-q^{6k})sb_{(3,1)}+(q^{-2k}-q^{2k})sb_{(2,2)} \\
& +(q^{-6k}-2q^{-2k})sb_{(2,1,1)}-q^{-10k}sb_{(1,1,1,1)} \\
&
+(3-2q^{4k})t^{-2k}sb_{(2)}+(2q^{-4k}-3)t^{-2k}sb_{(1,1)}+(q^{-2k}-q^{2k})t^{-4k}
\\
&
-(q^{2k}sb_{(2)}+q^{-2k}sb_{(1,1)}-2sb_{(1)}^{2}+t^{-2k})(sb_{(2)}-sb_{(1,1)})
\\
&
+2(-q^{4k}sb_{(3)}+(q^{2k}-q^{-2k})sb_{(2,1)}+q^{-4k}sb_{(1,1,1)}+(q^{2k}-q^{-2k})t^{-2k}sb_{(1)})sb_{(1)}
\end{align*}%
The rational function $\dfrac{(q-q^{-1})^{2}}{(q^{2}-q^{-2})(q-q^{-1})^{2}}%
\mathrm{z}_{(2),(1),(1)}g_{(2),(1),(1)}(q,t)$ is in the ring $\mathbb{Z}%
[t,t^{-1}][q-q^{-1}]$.

We list the first two in this family

\begin{align*}
\frac{(q-q^{-1})^{2}\mathrm{z}_{(2),(1),(1)}g_{(2),(1),(1)}(T(3,3))}{%
(q-q^{-1})^{2}(q^{2}-q^{-2})}&
=(-t^{-4}+4t^{-2}+2-12t^{2}+7t^{4})+(2t^{-3}-2t^{-1}-10t+10t^{3})z \\
& +(1-6t^{2}+5t^{4})z^{2}+(-6t+6t^{3})z^{3} \\
& +(-t^{2}+t^{4})z^{4}+(-t+t^{3})z^{5} \\
\frac{(q-q^{-1})^{2}\mathrm{z}_{(2),(1),(1)}g_{(2),(1),(1)}(T(3,6))}{%
(q-q^{-1})(q^{2}-q^{-2})}&
=(-2t^{-8}-4t^{-6}+22t^{-4}-48t^{-2}+146-204t^{2}+90t^{4}) \\
& +(16t^{-5}-48t^{-3}+176t^{-1}-336t+192t^{3})z \\
& +(-t^{-8}-2t^{-6}+15t^{-4}-68t^{-2}+361-650t^{2}+345t^{4})z^{2} \\
& +(12t^{-5}-68t^{-3}+452t^{-1}-1036t+640t^{3})z^{3} \\
& +(2t^{-4}-38t^{-2}+398-950t^{2}+588t^{4})z^{4} \\
& +(2t^{-5}-38t^{-3}+494t^{-1}-1406t+948t^{3})z^{5} \\
& +(-10t^{-2}+239-780t^{2}+551t^{4})z^{6} \\
& +(-10t^{-3}+286t^{-1}-1056t+780t^{3})z^{7} \\
& +(-t^{-2}+80-377t^{2}+298t^{4})z^{8} \\
& +(-t^{-3}+91t^{-1}-467t+377t^{3})z^{9} \\
& +(14-106t^{2}+92t^{4})z^{10}+(15t^{-1}-121t+106t^{3})z^{11} \\
& +(1-16t^{2}+15t^{4})z^{12}+(t^{-1}-17t+16t^{3})z^{13} \\
& +(-t^{2}+t^{4})z^{14}+(-t+t^{3})z^{15}
\end{align*}

After anti-symmetrization we obtain the table of the integers $N_{%
\overrightarrow{\mu },g,\beta }$ for $\overrightarrow{\mu }=(2),(1),(1)$.

Untill now, we have seen the orthogonal LMOV conjecture is valid for the
knots $T(2,k)$ and $T(3,3k)$ with small number $k$.

In fact, we can prove them for arbitrary $k\in \mathbb{Z}_{>0}$.

For instance, we investigate torus knot $T(2,k)$ for odd integer number $k$
with partition $(2)$. We can express%
\begin{equation*}
\mathrm{z}_{(2)}(g_{(2)}(q,t)-g_{(2)}(q,-t))/2
\end{equation*}%
in terms of $pb$'s instead of $sb$'s. After simplification, we have%
\begin{align*}
\frac{\mathrm{z}_{(2)}(g_{(2)}(q,t)-g_{(2)}(q,-t))}{2}& =t^{-2k}\frac{%
t-t^{-1}}{q-q^{-1}}(\frac{(q^{2k}-q^{-2k})}{(q-q^{-1})(q^{3}-q^{-3})}%
(-(t^{2}+t^{-2})(q^{2k}+q^{-2k}-q^{2}-q^{-2}) \\
& +(q^{4}+q^{-4})(q^{2k}+q^{-2k})-q^{4}-2-q^{-4}) \\
& +\frac{t^{-k}(q^{4k}-q^{-4k})}{q^{2}-q^{-2}}%
(-t(q^{k+1}-q^{-k-1})+t^{-1}(q^{k-1}-q^{-k+1})))\text{,}
\end{align*}%
It is easy to see that the rational function
\begin{equation*}
\dfrac{(q-q^{-1})^{2}}{2(q^{2}-q^{-2})}(g_{(2)}(q,t)-g_{(2)}(q,-t))
\end{equation*}%
is in the ring $\mathbb{Z}[t,t^{-1}][q-q^{-1}]$ by a tedious discussion on
the residue of $k$ modulo $6$. Actually, all these examples can be proved in
such a manner.

\section{Formulas of Lickorish-Millett Type}

\label{sec6}

The Skein relations of Kauffman polynomials are

\begin{enumerate}
\item $<\mathcal{L}_{+}>-<\mathcal{L}_{-}>=z(<\mathcal{L}_{||}>-<\mathcal{L}%
_{=}>)$,

where $\mathcal{L}_{+}$, $\mathcal{L}_{-}$, $\mathcal{L}_{||}$ and $\mathcal{%
L}_{=}$ stand for positive crossing, negative crossing, vertical resolution
and horizontal resolution respectively.

\item $<\mathcal{L}^{+kink}>=t<\mathcal{L}>,\hskip0.5cm<\mathcal{L}%
^{-kink}>=t^{-1}<\mathcal{L}>$.
\end{enumerate}

The variable $z$ is our $q-q^{-1}$ in previous sections, and the Kauffman
polynomials are given by%
\begin{equation*}
K_{\mathcal{L}}(z,t)=t^{-2lk(\mathcal{L})}<\mathcal{L}>
\end{equation*}
with the normalization $K_{\bigcirc }(z,t)=1$ for the unknot $\bigcirc $. In
terms of quantum group invariants, we have%
\begin{equation*}
W_{(1)^{L}}^{SO}(\mathcal{L})=(1+\frac{t-t^{-1}}{z})<\mathcal{L}>\text{.}
\end{equation*}
The Kauffman polynomials admit the expansions%
\begin{equation*}
K_{\mathcal{L}}(z,t)=\underset{g\geq 0}{\sum }\widetilde{p}_{g+1-L}^{%
\mathcal{L}}(t)z^{g+1-L}
\end{equation*}
and%
\begin{equation*}
<\mathcal{L}>=\underset{g\geq 0}{\sum }p_{g+1-L}^{\mathcal{L}}(t)z^{g+1-L}
\end{equation*}
with respect to variable $z$. The classical Lickorish-Millett Formula \cite%
{LiM} reads
\begin{equation*}
\widetilde{p}_{1-L}^{\mathcal{L}}(t)=t^{-2lk(\mathcal{L})}(t-t^{-1})^{L-1}%
\underset{\alpha =1}{\overset{L}{\prod }}p_{0}^{\mathcal{K}_{\alpha }}(t)
\end{equation*}%
and so
\begin{equation*}
p_{1-L}^{\mathcal{L}}(t)=(t-t^{-1})^{L-1}\underset{\alpha =1}{\overset{L}{%
\prod }}p_{0}^{\mathcal{K}_{\alpha }}(t),
\end{equation*}%
which give a concrete description of $\widetilde{p}_{1-L}^{\mathcal{L}}(t)$,
the coefficient of the lowest degree terms of $K_{\mathcal{L}}(z,t)$, in
terms of invariants of the sub-knots $\mathcal{K}_{1},\mathcal{K}_{2},\cdots
,\mathcal{K}_{L}$ of $\mathcal{L}$. In the following theorem, we provide
explicit formulas for $p_{2-L}^{\mathcal{L}}(t)$ and $p_{3-L}^{\mathcal{L}%
}(t)$, which are regarded as higher Lickorish-Millett relations. These
formulas can be proved purely by skein relations. Through resolving
intersections at different link components, it is not hard to prove the
following. Also, these formulas can be directly deduced from the Conjecture %
\ref{Degree Conj} (See Section \ref{sec7}). In \cite{Kan}, Kanenobu got some
relationships(non-explicit) between these terms.

\begin{theorem}
\label{Thm6.1} Let $\mathcal{L}_{1,2}$ be the sub-link of $\mathcal{L}$
which composed of components $K_{1}$ and $K_{2}$. The coefficients $p_{2-L}^{%
\mathcal{L}}(t)$ and $p_{3-L}^{\mathcal{L}}(t)$ are given by the formulas
\begin{equation*}
p_{2-L}^{\mathcal{L}}(t)=(L-1)(t-t^{-1})^{L-2}p_{0}^{\mathcal{K}%
_{1}}(t)\cdots p_{0}^{\mathcal{K}_{L}}(t)+(t-t^{-1})^{L-1}(p_{1}^{\mathcal{K}%
_{1}}(t)p_{0}^{\mathcal{K}_{2}}(t)\cdots p_{0}^{\mathcal{K}_{L}}(t)+perm);
\end{equation*}%
\begin{align*}
p_{3-L}^{\mathcal{L}}(t)& =\dbinom{L-1}{2}(t-t^{-1})^{L-3}p_{0}^{\mathcal{K}%
_{1}}(t)\cdots p_{0}^{\mathcal{K}_{L}}(t)+(t-t^{-1})^{L-2}(p_{1}^{\mathcal{L}%
_{1,2}}(t)p_{0}^{\mathcal{K}_{3}}(t)\cdots p_{0}^{\mathcal{K}_{L}}(t)+perm)
\\
& -(L-2)(t-t^{-1})^{L-1}(p_{2}^{\mathcal{K}_{1}}(t)p_{0}^{\mathcal{K}%
_{2}}(t)\cdots p_{0}^{\mathcal{K}_{L}}(t)+perm).
\end{align*}
\end{theorem}

\begin{proof}
The formulas in the theorem are obvious when $L=1$, and the formula for $%
p_{3-L}^{\mathcal{L}}(t)$ is also valid for $L=2$. We prove the theorem by
induction. Let $\mathcal{L}$ be a link with $L+1$ components. The main idea
is using skein relations at the intersection points of different components
of the $\mathcal{L}$ until the component $\mathcal{K}_{L+1}$ split from the
link.

First we perform the skein relation at the crossings between $\mathcal{K}%
_{1} $ and $\mathcal{K}_{L+1}$ until there is no intersection between them.
We need do skein relation $(n_{1,L+1}^{+}+n_{1,L+1}^{-})/2$ times, where $%
n_{1,L+1}^{+}$ and $n_{1,L+1}^{-}$ denote the number of positive and
negative crossings between $\mathcal{K}_{1}$ and $\mathcal{K}_{L+1}$
respectively. Thus the linking number between $\mathcal{K}_{1}$ and $%
\mathcal{K}_{L+1}$ is $lk(\mathcal{L}%
_{1,L+1})=(n_{1,L+1}^{+}-n_{1,L+1}^{-})/2$.

From the calculation, one can see that do the skein relation at a positive
crossing will lead similar result. Thus without loss of generality, we can
assume $n_{1,L+1}^{-}>0$ and do the skein relation at a negative crossing
first:
\begin{equation*}
<\mathcal{L}_{+}>-<\mathcal{L}_{-}>=z(<\mathcal{L}_{(1\Vert L+1),2,\cdots
,L}>-<\mathcal{L}_{(1=L+1),2,\cdots ,L}>),
\end{equation*}%
where $\mathcal{L}_{-}$ is the original link $\mathcal{L}$, $(1\Vert L+1)$%
(resp$.(1=L+1)$) is the new knot component derived from $\mathcal{K}_{1}$
and $\mathcal{K}_{L+1}$ by taking vertical (resp.horizontal) lines as its
resolution at the intersection of $\mathcal{K}_{1}$ and $\mathcal{K}_{L+1}$
in the new link $\mathcal{L}_{(1\Vert L+1),2,\cdots ,L}$ (resp.$\mathcal{L}%
_{(1=L+1),2,\cdots ,L}$). Both $\mathcal{L}_{(1\Vert L+1),2,\cdots ,L}$ and $%
\mathcal{L}_{(1=L+1),2,\cdots ,L}$ are of $L$-components, while $\mathcal{L}%
_{+}$ is the link obtained simply by changing the sign of the chosen
crossing, thus has the same $L+1$ components as the original link $\mathcal{L%
}=\mathcal{L}_{-}$.

Taking the few leading terms in the skein relation formula
\begin{align*}
& (p_{-L}^{\mathcal{L}_{+}}(t)z^{-L}+p_{1-L}^{\mathcal{L}%
_{+}}(t)z^{1-L}+p_{2-L}^{\mathcal{L}_{+}}(t)z^{2-L})-(p_{-L}^{\mathcal{L}%
}(t)z^{-L}+p_{1-L}^{\mathcal{L}}(t)z^{1-L}+p_{2-L}^{\mathcal{L}}(t)z^{2-L})
\\
& =z(p_{1-L}^{\mathcal{L}_{(1||L+1),2,\cdots ,L}}(t)z^{1-L}-p_{1-L}^{%
\mathcal{L}_{(1=L+1),2,\cdots ,L}}(t)z^{1-L})
\end{align*}%
and comparing the coefficients, we have

1) $p_{-L}^{\mathcal{L}_{+}}(t)=p_{-L}^{\mathcal{L}}(t)$ (this one gives the
formula for $p_{1-L}^{\mathcal{L}}$, which we don't use.)

2) $p_{1-L}^{\mathcal{L}_{+}}(t)=p_{1-L}^{\mathcal{L}}(t)$

3) $p_{2-L}^{\mathcal{L}_{+}}(t)-p_{2-L}^{\mathcal{L}}(t)=p_{1-L} ^{\mathcal{%
L}_{(1\|L+1),2,\cdots,L}}-p_{1-L}^{\mathcal{L}_{(1=L+1),2,\cdots,L}}$

By the Lickorish-Millett Formula,
\begin{equation*}
p_{1-L}^{\mathcal{L}_{(1\Vert L+1),2,\cdots ,L}}=(t-t^{-1})^{L-1}p_{0}^{%
\mathcal{K}_{2}}(t)\cdots p_{0}^{\mathcal{K}_{L}}(t)p_{0}^{\mathcal{K}%
_{(1\Vert L+1)}}(t)
\end{equation*}%
and
\begin{equation*}
p_{1-L}^{\mathcal{L}_{(1=L+1),2,\cdots ,L}}=(t-t^{-1})^{L-1}p_{0}^{\mathcal{K%
}_{2}}(t)\cdots p_{0}^{\mathcal{K}_{L}}(t)p_{0}^{\mathcal{K}_{(1=L+1)}}(t),
\end{equation*}%
where $\mathcal{K}_{(1||L+1)}$ (resp.$\mathcal{K}_{(1=L+1)}$) is the knot
derived from the sub-link $\mathcal{L}_{1,L+1}$ by taking vertical
(resp.horizontal) lines as its resolution at the chosen crossing. Thus
\begin{equation}
p_{2-L}^{\mathcal{L}_{+}}-p_{2-L}^{\mathcal{L}}=(t-t^{-1})^{L-1}p_{0}^{%
\mathcal{K}_{2}}\cdots p_{0}^{\mathcal{K}_{L}}(p_{0}^{\mathcal{K}%
_{(1||L+1)}}-p_{0}^{\mathcal{K}_{(1=L+1)}}).
\end{equation}

We play a trick here to find the expression of $p_{0}^{\mathcal{K}%
_{(1||L+1)}}$ and $p_{0}^{\mathcal{K}_{(1=L+1)}}$. Consider the sub-link $%
\mathcal{L}_{1,L+1}$ of $\mathcal{L}$, which only have two components $%
\mathcal{K}_{1}$ and $\mathcal{K}_{L+1}$. Then do the skein relation at
exactly the same crossing as we did in the original link $\mathcal{L}$. The
same argument lead to
\begin{equation*}
p_{1}^{(\mathcal{L}_{1,L+1})_{+}}-p_{1}^{\mathcal{L}_{1,L+1}}=p_{0} ^{%
\mathcal{K}_{(1\|L+1)}}-p_{0}^{\mathcal{K}_{(1=L+1)}}.
\end{equation*}
which substitute back gives
\begin{equation}
p_{2-L}^{\mathcal{L}_{+}}-p_{2-L}^{\mathcal{L}}=(t-t^{-1}) ^{L-1}p_{0}^{%
\mathcal{K}_{2}}\cdots p_{0}^{\mathcal{K}_{L}}(p_{1}^{(\mathcal{L}%
_{1,L+1})_{+}}-p_{1}^{\mathcal{L}_{1,L+1}})
\end{equation}

In the above equation, $p_{2-L}^{\mathcal{L}}$ is expressed in terms of
invariants of $\mathcal{L}_{+}$ and some simple terms. Then we do skein
relations $\mathcal{L}_{+}$ at other intersection points between $\mathcal{K}%
_{1}$ and $\mathcal{K}_{L+1}$ until these two components become unlinked. We
cancel all the middle state in this procedure, and finally we reached
\begin{equation}
p_{2-L}^{\mathcal{L}^{(1)}}-p_{2-L}^{\mathcal{L}}=(t-t^{-1})^{L-1}p_{0}^{%
\mathcal{K}_{2}}\cdots p_{0}^{\mathcal{K}_{L}}(p_{1}^{\mathcal{L}%
_{1,L+1}^{(1)}}-p_{1}^{\mathcal{L}_{1,L+1}}).
\end{equation}%
Here $\mathcal{L}^{(1)}$ is the final state of $\mathcal{L}$, in which $%
\mathcal{K}_{1}$ and $\mathcal{K}_{L+1}$ are unlinked. $\mathcal{L}%
_{1,L+1}^{(1)}$ is the corresponding final state of $\mathcal{L}_{1,L+1}$
under the same procedure of skein relations. In $\mathcal{L}_{1,L+1}^{(1)}$,
the two components $\mathcal{K}_{1}$ and $\mathcal{K}_{L+1}$ are unlinked
too, i.e, $\mathcal{L}_{1,L+1}^{(1)}$ is the disjoint union of two knots $%
\mathcal{K}_{1}$ and $\mathcal{K}_{L+1}$, $W_{(1),(1)}^{SO}(\mathcal{L}%
_{1,L+1}^{(1)})=W_{(1)}^{SO}(\mathcal{K}_{1})W_{(1)}^{SO}(\mathcal{K}_{L+1})$%
.

By definition of Kauffman polynomials, $W^{SO}(\mathcal{L})=(1+\frac{t-t^{-1}%
}{z})<\mathcal{L}>$ for all links $\mathcal{L}$, we have
\begin{equation*}
<\mathcal{L}_{1,L+1}^{(1)}>=(1+\frac{t-t^{-1}}{z})<\mathcal{K}_{1}><\mathcal{%
K}_{L+1}>
\end{equation*}
Up to the third leading terms,
\begin{equation*}
p_{-1}^{\mathcal{L}_{1,L+1}^{(1)}}z^{-1}+p_{0}^{\mathcal{L}%
_{1,L+1}^{(1)}}+p_{1}^{\mathcal{L}_{1,L+1}^{(1)}}z=(1+\frac{t-t^{-1}}{z}%
)(p_{0}^{\mathcal{K}_{1}}+p_{1}^{\mathcal{K}_{1}}z+p_{2}^{\mathcal{K}%
_{1}}z^{2})(p_{0}^{\mathcal{K}_{L+1}}+p_{1}^{\mathcal{K}_{L+1}}z+p_{2}^{%
\mathcal{K}_{L+1}}z^{2})
\end{equation*}%
and comparing the coefficients:

1) $p_{-1}^{\mathcal{L}_{1,L+1}^{(1)}}=(t-t^{-1})p_{0}^{\mathcal{K}%
_{1}}p_{0}^{\mathcal{K}_{L+1}}$,

2) $p_{0}^{\mathcal{L}_{1,L+1}^{(1)}}=p_{0}^{\mathcal{K}_{1}}p_{0}^{\mathcal{%
K}_{L+1}} +(t-t^{-1})(p_{0}^{\mathcal{K}_{1}}p_{1}^{\mathcal{K}_{L+1}}
+p_{1}^{\mathcal{K}_{1}}p_{0}^{\mathcal{K}_{L+1}})$,

3) $p_{1}^{\mathcal{L}_{1,L+1}^{(1)}}=p_{0}^{\mathcal{K}_{1}}p_{1}^{\mathcal{%
K}_{L+1}} +p_{1}^{\mathcal{K}_{1}}p_{0}^{\mathcal{K}_{L+1}}
+(t-t^{-1})(p_{0}^{\mathcal{K}_{1}}p_{2}^{\mathcal{K}_{L+1}} +p_{1}^{%
\mathcal{K}_{1}}p_{1}^{\mathcal{K}_{L+1}}+p_{2}^{\mathcal{K}_{1}}p_{0}^{%
\mathcal{K}_{L+1}})$.

In summary, we have the following formula now
\begin{align*}
& p_{2-L}^{\mathcal{L}^{(1)}}-p_{2-L}^{\mathcal{L}} \\
& =(t-t^{-1})^{L-1}p_{0}^{\mathcal{K}_{1}}\cdots p_{0}^{\mathcal{K}%
_{L}}p_{1}^{\mathcal{K}_{L+1}}+(t-t^{-1})^{L-1}p_{1}^{\mathcal{K}_{1}}p_{0}^{%
\mathcal{K}_{2}}\cdots p_{0}^{\mathcal{K}_{L}}p_{0}^{\mathcal{K}_{L+1}} \\
& +(t-t^{-1})^{L}p_{0}^{\mathcal{K}_{1}}\cdots p_{0}^{\mathcal{K}_{L}}p_{2}^{%
\mathcal{K}_{L+1}}+(t-t^{-1})^{L}p_{1}^{\mathcal{K}_{1}}p_{0}^{\mathcal{K}%
_{2}}\cdots p_{0}^{\mathcal{K}_{L}}p_{1}^{\mathcal{K}_{L+1}} \\
& +(t-t^{-1})^{L}p_{2}^{\mathcal{K}_{1}}p_{0}^{\mathcal{K}_{2}}\cdots p_{0}^{%
\mathcal{K}_{L}}p_{0}^{\mathcal{K}_{L+1}}-(t-t^{-1})^{L-1}p_{1}^{\mathcal{L}%
_{1,L+1}}p_{0}^{\mathcal{K}_{2}}\cdots p_{0}^{\mathcal{K}_{L}}.
\end{align*}

Next perform all the above procedures for the link $\mathcal{L}^{(1)}$ to
the final state $\mathcal{L}^{(1)(2)}$ in which the components $\mathcal{K}%
_{2}$ and $\mathcal{K}_{L+1}$ are unlinked.

Repeat this process totally $L$ times until $\mathcal{K}_{L+1}$ is unlinked
to the sub-link $\mathcal{L}_{1,\cdots ,L}$ of $\mathcal{L}$, the result is
given by
\begin{align*}
& p_{2-L}^{\mathcal{L}^{(1)\cdots (L)}}-p_{2-L}^{\mathcal{L}} \\
& =L(t-t^{-1})^{L-1}p_{0}^{\mathcal{K}_{1}}\cdots p_{0}^{\mathcal{K}%
_{L}}p_{1}^{\mathcal{K}_{L+1}}+(t-t^{-1})^{L-1}(p_{1}^{\mathcal{K}%
_{1}}p_{0}^{\mathcal{K}_{2}}\cdots p_{0}^{\mathcal{K}_{L}}+perm)p_{0}^{%
\mathcal{K}_{L+1}} \\
& +L(t-t^{-1})^{L}p_{0}^{\mathcal{K}_{1}}\cdots p_{0}^{\mathcal{K}%
_{L}}p_{2}^{\mathcal{K}_{L+1}}+(t-t^{-1})^{L}(p_{1}^{\mathcal{K}_{1}}p_{0}^{%
\mathcal{K}_{2}}\cdots p_{0}^{\mathcal{K}_{L}}+perm)p_{1}^{\mathcal{K}_{L+1}}
\\
& +(t-t^{-1})^{L}(p_{2}^{\mathcal{K}_{1}}p_{0}^{\mathcal{K}_{2}}\cdots
p_{0}^{\mathcal{K}_{L}}+perm)p_{0}^{\mathcal{K}%
_{L+1}}-(t-t^{-1})^{L-1}(p_{1}^{\mathcal{L}_{1,L+1}}p_{0}^{K_{2}}\cdots
p_{0}^{\mathcal{K}_{L}}+perm).
\end{align*}

As the link $\mathcal{L}^{(1)\cdots (L)}$ is the disjoint union of the
sub-link $\mathcal{L}_{1,\cdots ,L}$ of $\mathcal{L}$ and the knot $\mathcal{%
K}_{L+1}$, $W_{(1)^{L+1}}^{SO}(\mathcal{L}^{(1)\cdots (L)})=W_{(1)^{L}}^{SO}(%
\mathcal{L}_{1,\cdots ,L})W_{(1)}^{SO}(\mathcal{K}_{L+1})$. Again, this can
be rewrite into the form
\begin{equation*}
<\mathcal{L}^{(1)\cdots (L)}>=(1+\frac{t-t^{-1}}{z})<\mathcal{L}_{1,\cdots
,L}><\mathcal{K}_{L+1}>
\end{equation*}%
as $lk(\mathcal{L}^{(1)\cdots (L)})=lk(\mathcal{L}_{1,\cdots ,L})$. Up to
third leading terms,
\begin{align*}
& p_{-L}^{\mathcal{L}^{(1)\cdots (L)}}z^{-L}+p_{-L}^{\mathcal{L}^{(1)\cdots
(L)}}z^{1-L}+p_{2-L}^{\mathcal{L}^{(1)\cdots (L)}}z^{2-L} \\
& =(1+\frac{t-t^{-1}}{z})(p_{1-L}^{\mathcal{L}_{1,\cdots
,L}}z^{1-L}+p_{2-L}^{\mathcal{L}_{1,\cdots ,L}}z^{2-L}+p_{3-L}^{\mathcal{L}%
_{1,\cdots ,L}}z^{3-L})(p_{0}^{\mathcal{K}_{L+1}}+p_{1}^{\mathcal{K}%
_{L+1}}z+p_{2}^{\mathcal{K}_{L+1}}z^{2})
\end{align*}%
and comparing the coefficients

1) $p_{-L}^{\mathcal{L}^{(1)\cdots(L)}}=(t-t^{-1})p_{1-L} ^{\mathcal{L}%
_{1,\cdots,L}}p_{0}^{\mathcal{K}_{L+1}}$,

2) $p_{1-L}^{\mathcal{L}^{(1)\cdots(L)}}=p_{1-L}^{\mathcal{L}
_{1,\cdots,L}}p_{0}^{\mathcal{K}_{L+1}}+(t-t^{-1})(p_{1-L}^{\mathcal{L}%
_{1,\cdots,L}} p_{1}^{\mathcal{K}_{L+1}}+p_{2-L}^{\mathcal{L}%
_{1,\cdots,L}}p_{0}^{\mathcal{K}_{L+1}})$,

3)$p_{2-L}^{\mathcal{L}^{(1)\cdots(L)}}=p_{1-L}^{\mathcal{L}_{1,\cdots,L}}
p_1^{\mathcal{K}_{L+1}}+p_{2-L}^{\mathcal{L}_{1,\cdots,L}}p_{0}^{\mathcal{K}%
_{L+1}} +(t-t^{-1})(p_{1-L}^{\mathcal{L}_{1,\cdots,L}}p_{2}^{\mathcal{K}%
_{L+1}}+p_{2-L} ^{\mathcal{L}_{1,\cdots,L}}p_{1}^{\mathcal{K}%
_{L+1}}+p_{3-L}^{\mathcal{L}_{1,\cdots,L}} p_{0}^{\mathcal{K}_{L+1}})$.

We now can finish the proof by induction. Be careful that our link $\mathcal{%
L}$ has $L+1$ components. The sub-link $\mathcal{L}_{1,\cdots ,L}$ has $L$
components and by induction
\begin{equation*}
p_{2-L}^{\mathcal{L}_{1,\cdots ,L}}=(L-1)(t-t^{-1})^{L-1}p_{0}^{\mathcal{K}%
_{1}}\cdots p_{0}^{\mathcal{K}_{L}}+(t-t^{-1})^{L}(p_{1}^{\mathcal{K}%
_{1}}p_{0}^{\mathcal{K}_{2}}\cdots p_{0}^{\mathcal{K}_{L}}+perm),
\end{equation*}%
so
\begin{align*}
& p_{2-(L+1)}^{\mathcal{L}}(t)=p_{1-L}^{\mathcal{L}^{(1)\cdots (L)}} \\
& =p_{1-L}^{\mathcal{L}_{1,\cdots ,L}}p_{0}^{\mathcal{K}%
_{L+1}}+(t-t^{-1})(p_{1-L}^{\mathcal{L}_{1,\cdots ,L}}p_{1}^{\mathcal{K}%
_{L+1}}+p_{2-L}^{\mathcal{L}_{1,\cdots ,L}}p_{0}^{\mathcal{K}_{L+1}}) \\
& =L(t-t^{-1})^{L-1}p_{0}^{\mathcal{K}_{1}}\cdots p_{0}^{\mathcal{K}%
_{L}}p_{0}^{\mathcal{K}_{L+1}}+(t-t^{-1})^{L}(p_{1}^{\mathcal{K}_{1}}\cdots
p_{0}^{\mathcal{K}_{L}}p_{0}^{\mathcal{K}_{L+1}}+perm).
\end{align*}%
This finishes the proof of the first part of the theorem.

Now we have sufficient results to prove the second part. We have seen that
\begin{align*}
p_{2-L}^{\mathcal{L}^{(1)\cdots (L)}}& =p_{1-L}^{\mathcal{L}_{1,\cdots
,L}}p_{1}^{\mathcal{K}_{L+1}}+p_{2-L}^{\mathcal{L}_{1,\cdots ,L}}p_{0}^{%
\mathcal{K}_{L+1}}+(t-t^{-1})(p_{1-L}^{\mathcal{L}_{1,\cdots ,L}}p_{2}^{%
\mathcal{K}_{L+1}}+p_{2-L}^{\mathcal{L}_{1,\cdots ,L}}p_{1}^{\mathcal{K}%
_{L+1}}+p_{3-L}^{\mathcal{L}_{1,\cdots ,L}}p_{0}^{\mathcal{K}_{L+1}}) \\
& =L(t-t^{-1})^{L-1}p_{0}^{\mathcal{K}_{1}}\cdots p_{0}^{\mathcal{K}%
_{L}}p_{1}^{\mathcal{K}_{L+1}}+(L-1)(t-t^{-1})^{L-2}p_{0}^{\mathcal{K}%
_{1}}\cdots p_{0}^{\mathcal{K}_{L}}p_{0}^{\mathcal{K}_{L+1}} \\
& +(t-t^{-1})^{L-1}(p_{1}^{\mathcal{K}_{1}}\cdots p_{0}^{\mathcal{K}%
_{L}}+perm)p_{0}^{\mathcal{K}_{L+1}}+(t-t^{-1})^{L}p_{0}^{\mathcal{K}%
_{1}}\cdots p_{0}^{K_{L}}p_{2}^{\mathcal{K}_{L+1}} \\
& +(t-t^{-1})^{L}(p_{1}^{\mathcal{K}_{1}}\cdots p_{0}^{\mathcal{K}%
_{L}}+perm)p_{1}^{\mathcal{K}_{L+1}}+(t-t^{-1})p_{3-L}^{\mathcal{L}%
_{1,\cdots ,L}}p_{0}^{\mathcal{K}_{L+1}}.
\end{align*}%
Combined with the expression of $p_{2-L}^{\mathcal{L}^{(1)\cdots
(L)}}-p_{2-L}^{\mathcal{L}},$we got an expression of $p_{2-L}^{\mathcal{L}}$
in terms of sub-links
\begin{align*}
p_{2-L}^{\mathcal{L}}& =(L-1)(t-t^{-1})^{L-2}p_{0}^{\mathcal{K}_{1}}\cdots
p_{0}^{\mathcal{K}_{L}}p_{0}^{\mathcal{K}_{L+1}}-(L-1)(t-t^{-1})^{L}p_{0}^{%
\mathcal{K}_{1}}\cdots p_{0}^{\mathcal{K}_{L}}p_{2}^{\mathcal{K}_{L+1}} \\
& -(t-t^{-1})^{L}(p_{2}^{\mathcal{K}_{1}}p_{0}^{\mathcal{K}_{2}}\cdots
p_{0}^{\mathcal{K}_{L}}+perm)p_{0}^{\mathcal{K}_{L+1}} \\
& +(t-t^{-1})^{L-1}(p_{1}^{\mathcal{L}_{1,L+1}}p_{0}^{\mathcal{K}_{2}}\cdots
p_{0}^{\mathcal{K}_{L}}+perm)+(t-t^{-1})p_{3-L}^{\mathcal{L}_{1,\cdots
,L}}p_{0}^{\mathcal{K}_{L+1}}.
\end{align*}%
Since the sub-link $\mathcal{L}_{1,\cdots ,L}$ of $\mathcal{L}$ contains $L$
components, by induction we have
\begin{align*}
p_{3-L}^{\mathcal{L}_{1,\cdots ,L}}(t)& =\dbinom{L-1}{2}%
(t-t^{-1})^{L-3}p_{0}^{\mathcal{K}_{1}}(t)\cdots p_{0}^{\mathcal{K}%
_{L}}(t)+(t-t^{-1})^{L-2}(p_{1}^{\mathcal{L}_{1,2}}(t)p_{0}^{\mathcal{K}%
_{3}}(t)\cdots p_{0}^{\mathcal{K}_{L}}(t)+perm) \\
& -(L-2)(t-t^{-1})^{L-1}(p_{2}^{\mathcal{K}_{1}}(t)p_{0}^{\mathcal{K}%
_{2}}(t)\cdots p_{0}^{\mathcal{K}_{L}}(a)+perm).
\end{align*}%
Here the permutation only involves the first $L$ components. Later, when
computing the invariants of $\mathcal{L}$, the permutations will also
include the $L+1$'th component. As the content is self-evident, we will not
mention this issue in the future. Substitute the above induction formula to
the expression of $\mathcal{L}$, the proof of the second part of the theorem
is finished:
\begin{align*}
p_{2-L}^{\mathcal{L}}& =(L-1)(t-t^{-1})^{L-2}p_{0}^{\mathcal{K}_{1}}\cdots
p_{0}^{\mathcal{K}_{L}}p_{0}^{\mathcal{K}_{L+1}}-(L-1)(t-t^{-1})^{L}p_{0}^{%
\mathcal{K}_{1}}\cdots p_{0}^{\mathcal{K}_{L}}p_{2}^{\mathcal{K}_{L+1}} \\
& -(t-t^{-1})^{L}(p_{2}^{\mathcal{K}_{1}}p_{0}^{\mathcal{K}_{2}}\cdots
p_{0}^{\mathcal{K}_{L}}+perm)p_{0}^{\mathcal{K}%
_{L+1}}+(t-t^{-1})^{L-1}(p_{1}^{\mathcal{L}_{1,L+1}}p_{0}^{\mathcal{K}%
_{2}}\cdots p_{0}^{\mathcal{K}_{L}}+perm) \\
& +[\dbinom{L-1}{2}(t-t^{-1})^{L-2}p_{0}^{\mathcal{K}_{1}}\cdots p_{0}^{%
\mathcal{K}_{L}}+(t-t^{-1})^{L-1}(p_{1}^{\mathcal{L}_{1,2}}p_{0}^{\mathcal{K}%
_{3}}\cdots p_{0}^{\mathcal{K}_{L}}+perm) \\
& -(L-2)(t-t^{-1})^{L}(p_{2}^{\mathcal{K}_{1}}p_{0}^{\mathcal{K}_{2}}\cdots
p_{0}^{\mathcal{K}_{L}}+perm)]p_{0}^{\mathcal{K}_{L+1}} \\
& =\dbinom{L}{2}(t-t^{-1})^{L-2}p_{0}^{\mathcal{K}_{1}}\cdots p_{0}^{%
\mathcal{K}_{L+1}}+(t-t^{-1})^{L-1}(p_{1}^{\mathcal{L}_{1,2}}p_{0}^{\mathcal{%
K}_{3}}\cdots p_{0}^{\mathcal{K}_{L+1}}+perm) \\
& -(L-1)(t-t^{-1})^{L}(p_{2}^{\mathcal{K}_{1}}p_{0}^{\mathcal{K}_{2}}\cdots
p_{0}^{\mathcal{K}_{L+1}}+perm).
\end{align*}
\end{proof}

\section{The Proof of the Conjecture for Column Diagram}

\label{sec7}

In the last section, we provide two formulas of Lickorish-Millett type. In
general, similar computation leads to expressions of $p_{n}^{\mathcal{L}}(t)$
in terms of invariants of sub-links of $\mathcal{L}$. Each additional
component of $\mathcal{L}$ gives rise to two such relations; thus we expect
there should be $2L-2$ such relations, i.e., all the $p_{n}^{\mathcal{L}}(t)$%
's for $1-L\leq n\leq L$ should be able to be described by sub-links of $%
\mathcal{L}$.

When the index $n$ increases, the expression become messy. To give a unified
treatment, we formulate the problem in terms of the partition function $%
Z_{CS}^{SO}(\mathcal{L};q,t)$ and free energy $F^{SO}(\mathcal{L};q,t)$.
Recall that we write
\begin{equation*}
Z_{CS}^{SO}(\mathcal{L};q,t)=1+\sum_{\overrightarrow{\mu }\neq
\overrightarrow{0}}Z_{\overrightarrow{\mu }}^{SO}pb_{\overrightarrow{\mu }}
\end{equation*}%
and
\begin{equation*}
F^{SO}(\mathcal{L};q,t)=\sum_{\overrightarrow{\mu }\neq \overrightarrow{0}%
}F_{\overrightarrow{\mu }}^{SO}pb_{\overrightarrow{\mu }}.
\end{equation*}%
Where $\overrightarrow{\mu }=(\mu ^{1},\cdots ,\mu ^{L})$ for partitions $%
\mu ^{1},\cdots ,\mu ^{L}$. In this section, we mainly focus on the
situation when all $\mu ^{i}$ are column like partitions. We look at the
first situation that all $\mu ^{i}$ are partitions $1$ now. We simply denote
such $\overrightarrow{\mu }$ by $(1)^{L}=(1),\cdots ,(1)$, since the
partition of $1$ is unique and there is no ambiguity. The coefficients $%
Z_{(1)^{L}}^{SO}=W_{(1)^{L}}^{SO}$.

Let $\Delta $ be a subset of the set $[L]:=\{1,\cdots ,L\}$. Write $\mathcal{%
L}_{\Delta }$ for the sub-link of $\mathcal{L}$ composed only by the
components with labels in $\Delta $. For example, when $\Delta =\{1,2\}$, $%
\mathcal{L}_{\Delta }$ is the link $\mathcal{L}_{1,2}$ discussed in the
previous section. We also denote by $\Delta $ the partition $\overrightarrow{%
\mu }=(\mu ^{1},\cdots ,\mu ^{L})$ such that $\mu ^{i}=(1)$ if $i\in \Delta $%
, and $0$ otherwise. The convention in the definition of quantum group
invariants is $W_{\Delta }^{SO}(\mathcal{L}):=W_{(1)^{|\Delta |}}^{SO}(%
\mathcal{L}_{\Delta })$. The formula (\ref{E:4.11}) then can be written as
\begin{equation}
F_{(1)^{L}}^{SO}(\mathcal{L})=\sum_{r=1}^{L}\frac{(-1)^{r-1}}{r}\sum_{\Delta
_{1},\cdots ,\Delta _{r}}\prod_{i=1}^{r}W_{\Delta _{i}}^{SO}(\mathcal{L})
\label{E:7.1}
\end{equation}%
where the second summation is over all nonempty subsets $\Delta _{1},\cdots
,\Delta _{r}$ which form a partition of the set $[L]$. We have seen that $%
F_{(1)^{L}}^{SO}(\mathcal{L})\in \mathbb{Q}(t)((z))$ for $z=q-q^{-1}$ has an
expansion
\begin{equation*}
F_{(1)^{L}}^{SO}(\mathcal{L})=\sum_{i\geq -L}a_{i}(t)z^{i}.
\end{equation*}

Conjecture \ref{Degree Conj} predict that $val_{z}F_{(1)^{L}}^{SO}(\mathcal{L%
})\geq L-2$, i.e, $a_{-L}=a_{1-L}=\cdots =a_{L-3}=0$. We now prove $%
a_{-L}=a_{1-L}=a_{2-L}=0$ by the classical Lichorish-Millett theorem and the
two formulas derived in last section.

\begin{theorem}
\label{Thm7.1} Expand $F_{(1)^{L}}^{SO}(\mathcal{L})$ as above, then we have
the vanishing result $a_{-L}=a_{1-L}=a_{2-L}=0$ if $L\geq 3$. In other
words, $val_{u}(F_{(1)^{L}}^{SO}(\mathcal{L}))=val_{z}(F_{(1)^{L}}^{SO}(%
\mathcal{L}))\geq 3-L$. In the case $L=2$, the second formula in Theorem \ref%
{Thm6.1} is empty, thus we only have $a_{-2}=a_{-1}=0$ and $%
val_{z}(F_{(1)^{2}}^{SO}(\mathcal{L}))\geq 0$.
\end{theorem}

\begin{proof}
We prove the theorem for $a_{1-L}$ when $L\geqslant 2$ by calculating (\ref%
{E:7.1}). The proof for $a_{-L}$ and $a_{2-L}$ are similar and we leave them
to the reader.
\begin{equation*}
W_{\Delta }^{SO}(\mathcal{L})\cong (1+\frac{t-t^{-1}}{z})(p_{1-|\Delta |}^{%
\mathcal{L}_{\Delta }}z^{1-|\Delta |}+p_{2-|\Delta |}^{\mathcal{L}_{\Delta
}}z^{2-|\Delta |}+p_{3-|\Delta |}^{\mathcal{L}_{\Delta }}z^{3-|\Delta |})(%
\bmod z^{3-|\Delta |})
\end{equation*}%
Denote by $[z^{n}]f$ the coefficient of $z^{n}$ in $f\in \mathbb{Q}(t)((z))$%
.
\begin{equation*}
a_{1-L}=\sum_{r=1}^{L}\frac{(-1)^{r-1}}{r}\sum_{\Delta _{1},\cdots ,\Delta
_{r}}[z^{1-L}](1+\frac{t-t^{-1}}{z})^{r}\prod_{i=1}^{r}<\mathcal{L}_{\Delta
_{i}}>
\end{equation*}

For each possible collection $\Delta _{1},\cdots ,\Delta _{r}$,
\begin{align*}
& [z^{1-L}](1+\frac{t-t^{-1}}{z})^{r}\prod_{i=1}^{r}<\mathcal{L}_{\Delta
_{i}}> \\
& =r(t-t^{-1})^{r-1}\prod_{i=1}^{r}p_{1-|\Delta _{i}|}^{\mathcal{L}_{\Delta
_{i}}}+(t-t^{-1})^{r}\sum_{i=1}^{r}p_{1-|\Delta _{1}|}^{\mathcal{L}_{\Delta
_{1}}}\cdots \widehat{p_{1-|\Delta _{i}|}^{\mathcal{L}_{\Delta _{i}}}}\cdots
p_{1-|\Delta _{r}|}^{\mathcal{L}_{\Delta _{r}}}\cdot p_{2-|\Delta _{i}|}^{%
\mathcal{L}_{\Delta _{i}}} \\
& =r(t-t^{-1})^{L-1}\prod_{i=1}^{r}\prod_{j\in \Delta
_{i}}p_{0}^{K_{j}}(t)+(t-t^{-1})^{L-1}\sum_{k=1}^{r}\underset{i\neq k}{%
\prod_{i=1}^{r}}\prod_{j\in \Delta _{i}}p_{0}^{K_{j}}(t)\cdot \lbrack
(|\Delta _{k}|-1)\prod_{j\in \Delta _{k}}p_{0}^{K_{j}} \\
& +(t-t^{-1})\sum_{l\in \Delta _{k}}p_{1}^{K_{l}}\underset{j\neq l}{%
\prod_{j\in \Delta _{k}}}p_{0}^{K_{j}}] \\
& =L(t-t^{-1})^{L-1}\prod_{\alpha =1}^{L}p_{0}^{K_{\alpha
}}(t)+(t-t^{-1})^{L}\sum_{j=1}^{L}p_{1}^{K_{j}}\underset{i\neq j}{%
\prod_{i=1}^{L}}p_{0}^{K_{i}}
\end{align*}%
has the same contribution. We need to count the number of these collections.
Let $\Lambda $ be a partition of $L$ of length $r$, the number of
collections $\{\Delta _{1},\cdots ,\Delta _{r}\}$ with $\{|\Delta
_{1}|,\cdots ,|\Delta _{r}|\}$ equal to the partition $\Lambda $ is given by
$\frac{r!}{|Aut\Lambda |}\cdot \frac{L!}{\Lambda _{1}!\cdots \Lambda _{r}!}$%
, hence
\begin{equation*}
a_{1-L}=(L(t-t^{-1})^{L-1}\prod_{\alpha =1}^{L}p_{0}^{K_{\alpha
}}(t)+(t-t^{-1})^{L}\sum_{j=1}^{L}p_{1}^{K_{j}}\underset{i\neq j}{%
\prod_{i=1}^{L}}p_{0}^{K_{i}})\cdot \sum_{\Lambda \vdash L}\frac{(-1)^{\ell
(\Lambda )-1}\ell (\Lambda )!}{\ell (\Lambda )|Aut\Lambda |}\cdot \frac{L!}{%
\Lambda !}
\end{equation*}%
which is zero by the following Lemma \ref{Lemma7.2} as $L\geq 2$.
\end{proof}

\begin{lemma}
\label{Lemma7.2} Assume $d_{\alpha }\geq 1$ for $\alpha =1,2,\cdots ,L$ and
the sum $d=d_{1}+\cdots +d_{L}$ is strictly greater than $1$ (i.e., if all $%
d_{i}=1$, then we assume $L>1$), then
\begin{equation}
\sum_{\overrightarrow{\lambda }\vdash \overrightarrow{d}}\frac{(-1)^{\ell (%
\overrightarrow{\lambda })-1}\ell (\overrightarrow{\lambda })!}{\ell (%
\overrightarrow{\lambda })|Aut\overrightarrow{\lambda }|\underset{\alpha =1}{%
\overset{L}{\prod }}\underset{j=1}{\overset{\ell (\lambda ^{\alpha })}{\prod
}}\lambda _{j}^{\alpha }!}=0.
\end{equation}
\end{lemma}

\begin{proof}
Let $\overrightarrow{t}=(t_{1},\cdots ,t_{L})$ and $|\overrightarrow{t}|=$ $%
t_{1}+\cdots +t_{L}$ in the trivial equality
\begin{equation*}
|\overrightarrow{t}|=\log (\exp (|\overrightarrow{t}|))=\log
(1+\sum_{n=1}^{+\infty }\frac{|\overrightarrow{t}|^{n}}{n!})
\end{equation*}%
so we have
\begin{equation*}
t_{1}+\cdots +t_{L}=\log (1+\underset{\overrightarrow{\beta }\neq 0}{\sum_{%
\overrightarrow{\beta }\in \mathbb{Z}_{\geq 0}^{L}}}\frac{\overrightarrow{t}%
^{\overrightarrow{\beta }}}{\overrightarrow{\beta }!})
\end{equation*}%
where we have adopt the notation $\overrightarrow{t}^{\overrightarrow{\beta }%
}=\underset{\alpha =1}{\overset{L}{\prod }}t_{\alpha }^{\beta _{\alpha }}$
and $\overrightarrow{\beta }!=\underset{\alpha =1}{\overset{L}{\prod }}%
(\beta _{\alpha }!)$. Expand the logarithm
\begin{equation*}
t_{1}+\cdots +t_{L}=\sum_{\overrightarrow{\beta }\in \mathbb{Z}_{\geq 0}^{L}}%
\overrightarrow{t}^{\overrightarrow{\beta }}\sum_{\overrightarrow{\lambda }%
\vdash \overrightarrow{\beta }}\frac{(-1)^{\ell (\overrightarrow{\lambda }%
)-1}}{\ell (\overrightarrow{\lambda })|Aut\overrightarrow{\lambda }|\underset%
{\alpha =1}{\overset{L}{\prod }}{\lambda }^{\alpha }{!}}.
\end{equation*}%
and comparing the coefficients of the term $t_{1}^{d_{1}}\cdots
t_{L}^{d_{L}} $ gives the vanishing formula.
\end{proof}

We remark that the vanishing of $a_{1-L}$ and $a_{2-L}$ also imply the
formulas for $p_{2-L}^{\mathcal{L}}$ and $p_{3-L}^{\mathcal{L}}$ proved in
last section. The approach in the previous section has the merit that it
produces explicit expressions, while the statement in terms of free energy
can give a uniform treatment to contain all the relations of
Lickorish-Millett type, as in the following theorem.

\begin{theorem}
\label{Thm7.3} Under the same notations as above, then we have the vanishing
result $a_{-L}=a_{1-L}=\cdots =a_{L-3}=0$. In other words, $%
val_{z}(F_{(1)^{L}}^{SO}(\mathcal{L}))\geq L-2$. Indeed, we have
\begin{equation*}
(q-q^{-1})^{2-L}F_{(1)^{L}}^{SO}(\mathcal{L})\in \mathbb{Z}%
[t,t^{-1}][q-q^{-1}].
\end{equation*}%
As a corollary, Conjecture \ref{Main Conj} is true for partitions $%
\overrightarrow{\mu }=(1,1,\cdots ,1)$.
\end{theorem}

\begin{proof}
We prove the theorem by induction. When $L=1$, $\mathcal{L}$ is a knot, $%
F_{(1)^{1}}^{SO}(\mathcal{L})=W_{(1)}^{SO}(\mathcal{L})=(1+\frac{t-t^{-1}}{z}%
)<\mathcal{L}>=(1+\frac{t-t^{-1}}{z})K_{\mathcal{L}}$ for the Kauffman
polynomial of $\mathcal{L}$ obviously has $z$-valuation equal to $-1=L-2$.
The theorem thus holds for knots.

Now Assume $\mathcal{L}$ is a link with $L>1$ components $\mathcal{K}%
_{1},\cdots ,\mathcal{K}_{L}$. We first deal with the simple case when $%
\mathcal{L}$ is the disjoint union of $\mathcal{K}_{\alpha }$'s. Then for
any partition $\Delta _{1},\cdots ,\Delta _{r}$ of the set $[L]$, the
product $\overset{r}{\underset{i=1}{\prod }}W_{\Delta _{i}}^{SO}(\mathcal{L}%
)=\underset{\alpha =1}{\overset{L}{\prod }}W_{(1)}^{SO}(\mathcal{K}_{\alpha
})$ is independent of the partition. Again let $\Lambda $ be a partition of $%
L$ of length $r$, the number of collections $\{\Delta _{1},\cdots ,\Delta
_{r}\}$ with $\{|\Delta _{1}|,\cdots ,|\Delta _{r}|\}$ equal to the
partition $\Lambda $ is given by $\frac{r!}{|Aut\Lambda |}\cdot \frac{L!}{%
\Lambda _{1}!\cdots \Lambda _{r}!}$, hence
\begin{equation*}
F_{(1)^{L}}^{SO}(\mathcal{L})=\prod_{\alpha =1}^{L}W_{(1)}^{SO}(\mathcal{K}%
_{\alpha })\cdot \sum_{\Lambda \vdash L}\frac{(-1)^{\ell (\Lambda )-1}\ell
(\Lambda )!L!}{\ell (\Lambda )|Aut\Lambda |\Lambda !}=0.
\end{equation*}%
There is another way to see this directly. For if the link $\mathcal{L}$ is
the disjoint union of $\mathcal{K}_{\alpha }$'s, then the free energy $%
F^{SO}(\mathcal{L},pb(z_{1}),\cdots ,pb(z_{L}))$ is the sum of the free
energy $F^{SO}(\mathcal{K}_{\alpha };pb(z_{\alpha }))$. The expansion of
such a sum $F^{SO}(\mathcal{L})$ with respect to $pb_{\overrightarrow{\mu }}$
does not contain terms of the form $\overset{L}{\underset{\alpha =1}{\prod }}%
p_{1}(z_{\alpha })$. Thus the theorem is true for links of the type of
disjoint union.

Finally, consider the Skein relation
\begin{equation*}
<\mathcal{L}_{+}>-<\mathcal{L}_{-}>=z(<\mathcal{L}_{||}>-<\mathcal{L}_{=}>),
\end{equation*}%
where $<\mathcal{L}_{+}>$ and $<\mathcal{L}_{-}>$ are two links coincide
everywhere except at one crossing $P$ between two different components $%
\mathcal{K}_{a}$ and $\mathcal{K}_{b}$ of the link $\mathcal{L}$ for $1\leq
a<b\leq L$. The link $<\mathcal{L}_{||}>$ (resp. $<\mathcal{L}_{=}>$) is the
link by replacing the crossing $P$ by two parallel vertical (resp.
horizontal) lines. Both $<\mathcal{L}_{||}>$ and $<\mathcal{L}_{=}>$ have $%
L-1$ components. Let's compute the difference
\begin{equation*}
F_{(1)^{L}}^{SO}(\mathcal{L}_{+})-F_{(1)^{L}}^{SO}(\mathcal{L}%
_{-})=\sum_{r=1}^{L}\frac{(-1)^{r-1}}{r}\sum_{\Delta _{1},\cdots ,\Delta
_{r}}(\prod_{i=1}^{r}W_{\Delta _{i}}^{SO}(\mathcal{L}_{+})-%
\prod_{i=1}^{r}W_{\Delta _{i}}^{SO}(\mathcal{L}_{-})).
\end{equation*}%
The summation is again over all partitions $\Delta _{1},\cdots ,\Delta _{r}$
of the set $[L]$. An important observation is that $\underset{i=1}{\overset{r%
}{\prod }}W_{\Delta _{i}}^{SO}(\mathcal{L}_{+})-\overset{r}{\underset{i=1}{%
\prod }}W_{\Delta _{i}}^{SO}(\mathcal{L}_{-})=0$ if $a$ and $b$ are not in
the same set $\Delta _{i}$ for some $i$, because in this situation the
sub-links $\mathcal{L}_{+,\Delta _{i}}$ coincide to the sub-links $\mathcal{L%
}_{-,\Delta _{i}}$. In particular, this is the case if $r=L$. The above
difference can be simplified
\begin{align*}
& F_{(1)^{L}}^{SO}(\mathcal{L}_{+})-F_{(1)^{L}}^{SO}(\mathcal{L}_{-}) \\
& =\sum_{r=1}^{L-1}\frac{(-1)^{r-1}}{r}\sum_{i=1}^{r}\sum_{\Delta
_{1},\cdots ,\Delta _{r};a,b\in \Delta _{i}}(W_{\Delta _{i}}^{SO}(\mathcal{L}%
_{+})-W_{\Delta _{i}}^{SO}(\mathcal{L}_{-}))\prod_{j=1,j\neq i}^{r}W_{\Delta
_{j}}^{SO}(\mathcal{L}_{+}) \\
& =\sum_{r=1}^{L-1}\frac{(-1)^{r-1}}{r}\sum_{i=1}^{r}\sum_{\Delta
_{1},\cdots ,\Delta _{r};a,b\in \Delta _{i}}z\cdot (W_{\Delta _{i}}^{SO}(%
\mathcal{L}_{||})-W_{\Delta _{i}}^{SO}(\mathcal{L}_{=}))\prod_{j=1,j\neq
i}^{r}W_{\Delta _{j}}^{SO}(\mathcal{L}_{+}) \\
& =z\cdot (F_{(1)^{L-1}}^{SO}(\mathcal{L}_{||})-F_{(1)^{L-1}}^{SO}(\mathcal{L%
}_{=})).
\end{align*}%
By induction, both $z^{2-(L-1)}\cdot F_{(1)^{L-1}}^{SO}(\mathcal{L}_{||})$
and $z^{3-L}\cdot F_{(1)^{L-1}}^{SO}(\mathcal{L}_{=})$ are in the ring $%
\mathbb{Z}[t,t^{-1}][z]$, thus if the theorem is true for link $\mathcal{L}%
_{+}$ if and only if it is true for link $\mathcal{L}_{-}$.

For a general link $\mathcal{L}$ which is not necessarily a disjoint union,
changing crossings between different components of $\mathcal{L}$
respectively until it becomes a disjoint union of $L$ knots. As the theorem
is true for disjoint union, it is true for $\mathcal{L}$.
\end{proof}

The whole results of Section \ref{sec6} can be view as application of
Theorem \ref{Thm7.3} combined with some combinatorial identities like Lemma %
\ref{Lemma7.2}.

To study the cases of partitions with more boxes, we first develop the
cabling technique. Let $\beta $ be a braid of which the closure is the link $%
\mathcal{L}$. For each $\overrightarrow{d}\in \mathbb{Z}_{+}^{L}$, denote $%
\beta _{\overrightarrow{d}}$ the braid obtained by cabling the $k$-th strand
of $\beta $ to $d_{\alpha }$ parallel ones if it in the $\alpha $-th
component of $\mathcal{L}$. The partition function of $\mathcal{L}$ and the
Kauffman polynomials are related by the following Lemma.

\begin{lemma}
\label{Lemma7.4} Assume $\beta $ is of writhe zero on every components, then
the partition function of $\mathcal{L}$ is related to the Kauffman
polynomial of the cabling link by
\begin{equation*}
W_{(1)^{d}}^{SO}(\beta _{\overrightarrow{d}})=\sum_{\overrightarrow{A}\in
\widehat{Br}_{\overrightarrow{d}}}\chi _{\overrightarrow{A}}(\mathrm{id})W_{%
\overrightarrow{A}}^{SO}(\mathcal{L};q,t)={\overrightarrow{d}!}\cdot
Z_{((1^{d_{1}}),\cdots ,(1^{d_{L}}))}^{SO}(\mathcal{L}),
\end{equation*}

where $d=\overset{L}{\underset{\alpha =1}{\sum }}d_{\alpha }$ and $%
\overrightarrow{d}!=\mathrm{z}_{((1^{d_{1}}),\cdots
,(1^{d_{L}}))}=d_{1}!\cdot \cdot \cdot d_{L}!$.
\end{lemma}

\begin{proof}
Take a $\beta $ of zero writhe on every component, the cabling link $\beta _{%
\overrightarrow{d}}$ is also of zero writhe on every component, and the
quantum group invariants $W_{\overrightarrow{A}}$ are equal to the trace of
\begin{equation}
\beta _{\overrightarrow{d}}\cdot (p_{A^{1}}\otimes \cdots \otimes p_{A^{L}})
\end{equation}%
in the Birman-Murakami-Wenzl algebra $C_{M}$ for $M=d_{1}r_{1}+\cdots
+d_{L}r_{L}$ and $p_{A^{\alpha }}$ is the minimal idempotents in $%
C_{d_{\alpha }}$ corresponding to the irreducible representation numbered by
the partition $A^{\alpha }$. Apparently, each $p_{A^{\alpha }}$ should
appear $r_{i}$ times in the above tensor. However, the naturality of the
universal $\mathcal{R}$-matrices plus the trace property will move all $%
p_{A^{\alpha }}$ to the same strand and thus one $p_{A^{\alpha }}$ for each $%
\alpha =1,2,\cdots ,L$ is enough.

The expansion coefficients $Z_{(1^{d_{1}},\cdots ,1^{d_{L}})}^{SO}(\mathcal{L%
})$ of the partition function can be calculated directly
\begin{align*}
{\overrightarrow{d}!}\cdot Z_{(1^{d_{1}},\cdots ,1^{d_{L}})}^{SO}(\mathcal{L}%
)& =\sum_{\overrightarrow{A}\in \widehat{Br}_{\overrightarrow{d}}}\chi _{%
\overrightarrow{A}}(\mathrm{id})W_{\overrightarrow{A}}^{SO}(\mathcal{L};q,t)
\\
& =\sum_{\overrightarrow{A}\in \widehat{Br}_{\overrightarrow{d}}}\chi _{%
\overrightarrow{A}}(\mathrm{id})\mathrm{tr}_{V^{M}}(\beta _{\overrightarrow{d%
}}\cdot (p_{A^{1}}\otimes \cdots \otimes p_{A^{L}})) \\
& =\mathrm{tr}_{V^{M}}(\beta _{\overrightarrow{d}}) \\
& =W_{(1)^{d}}^{SO}(\beta _{\overrightarrow{d}};q,t).
\end{align*}%
We have used the fact that for a semi-simple algebra, the dimension of an
irreducible representation $\chi _{A^{i}}(\mathrm{id})$ is the same as the
multiplicity of $A^{i}$ in the semi-simple decomposition of the algebra, so
\begin{equation*}
\sum_{\overrightarrow{A}\in \widehat{Br}_{\overrightarrow{d}}}\chi _{%
\overrightarrow{A}}(\mathrm{id})(p_{A^{1}}\otimes \cdots \otimes p_{A^{L}})=%
\mathrm{id}
\end{equation*}%
in the third equality.
\end{proof}

\begin{remark}
A similar formula holds for HOMFLY polynomials and can be proved in the same
way.
\end{remark}

\begin{theorem}
\label{Thm7.5} For partitions $\overrightarrow{\mu }=(\mu ^{1},\cdots ,\mu
^{L})\in \mathcal{P}^{L}$ such that $\mu ^{\alpha }=(1,1,\cdots ,1)\vdash
d_{\alpha }$ for each $\alpha =1,\cdots ,L$, we have
\begin{equation*}
\overrightarrow{d}!(q-q^{-1})^{2-d}\cdot F_{\overrightarrow{\mu }}(\mathcal{L%
},q,t)\in \mathbb{Z}[t,t^{-1}][q-q^{-1}].
\end{equation*}%
In particular, the Conjecture \ref{Main Conj} (Orthgonal LMOV Conj.) is
valid for such column like partitions.
\end{theorem}

\begin{proof}
We will use the symbol $(1)^{\overrightarrow{d}}$ to denote the partition $%
\overrightarrow{\mu }$ in the theorem. Let $\beta $ be a braid whose closure
is the link $\mathcal{L}$ with zero writhe. Let $\beta _{\overrightarrow{d}}$
be the cabling braid as in section \ref{sec3}. The calculation in Lemma \ref%
{Lemma7.4} in fact provide that
\begin{equation*}
Z_{(1)^{\overrightarrow{d}}}^{SO}(\mathcal{L})=\frac{1}{\overrightarrow{d}!}%
W_{(1)^{d}}^{SO}(\beta _{\overrightarrow{d}}),
\end{equation*}%
which reduce the situation back to the Kauffman case. A more careful
observation is the following cabling equality
\begin{equation*}
F_{(1)^{\overrightarrow{d}}}^{SO}(\mathcal{L})=\frac{1}{\overrightarrow{d}!}%
F_{(1)^{d}}^{SO}(\beta _{\overrightarrow{d}}),
\end{equation*}%
which, together with Theorem \ref{Thm7.3}, finishes the proof.

We now prove the cabling equality by comparing both sides. The LHS equal to
\begin{align*}
& \sum_{r=1}^{d}\frac{(-1)^{r-1}}{r}\sum_{\overrightarrow{A_{1}},\cdots ,%
\overrightarrow{A_{r}}}Z_{\overrightarrow{A_{1}}}^{SO}(\mathcal{L})\cdots Z_{%
\overrightarrow{A_{r}}}^{SO}(\mathcal{L}) \\
& =\sum_{r=1}^{d}\frac{(-1)^{r-1}}{r}\sum_{\overrightarrow{A_{1}},\cdots ,%
\overrightarrow{A_{r}}}\frac{W_{(1)^{||\overrightarrow{A_{1}}||}}^{SO}(\beta
_{|\overrightarrow{A_{1}}|})\cdots W_{(1)^{||\overrightarrow{A_{r}}%
||}}^{SO}(\beta _{|\overrightarrow{A_{r}}|})}{\overrightarrow{A_{1}}!\cdots
\overrightarrow{A_{r}}!},
\end{align*}%
where the summation is over all partitions $(\overrightarrow{A_{1}},\cdots ,%
\overrightarrow{A_{r}})$ of length of the partition $(1)^{\overrightarrow{d}%
} $. As each $\overrightarrow{A_{i}}$ must be of the form $%
((1^{a_{i}^{1}}),(1^{a_{i}^{2}}),\cdots ,(1^{a_{i}^{L}}))$ such that $%
\underset{i=1}{\overset{r}{\sum }}a_{i}^{\alpha }=d_{\alpha }$ for every $%
\alpha =1,2,\cdots ,L$. Then the symbols $|\overrightarrow{A_{i}}%
|=(a_{i}^{1},a_{i}^{2},\cdots ,a_{i}^{L})$, $||\overrightarrow{A_{i}}%
||=a_{i}^{1}+a_{i}^{2}+\cdots +a_{i}^{L}$ and $\overrightarrow{A_{i}}%
!=a_{i}^{1}!a_{i}^{2}!\cdots a_{i}^{L}!$ as in the introduction. Again $%
\beta $ is the braid with zero writhe on every components representing the
link $\mathcal{L}$, and $\beta _{|\overrightarrow{A_{i}}|}$ is the cabling
link.

The RHS equal to
\begin{equation*}
\frac{1}{\overrightarrow{d}!}\sum_{r=1}^{d}\frac{(-1)^{r-1}}{r}\sum_{\Delta
_{1},\cdots ,\Delta _{r}}\prod_{i=1}^{r}W_{\Delta _{i}}^{SO}(\beta _{%
\overrightarrow{d}})
\end{equation*}%
for $\Delta _{1},\cdots ,\Delta _{r}$ are non-empty sets which form a
partition of the set $[d]$. Each $\Delta _{i}$ can be further decompose into
a partition $\Xi _{i}^{1},\Xi _{i}^{2},\cdots ,\Xi _{i}^{L}$, such that
elements in $\Xi _{i}^{\alpha }$ labelling the components in $\beta _{%
\overrightarrow{d}}$ arising from the cabling of the $\alpha $th component
of $\mathcal{L}$. Write $a_{i}^{\alpha }=|\Xi _{i}^{\alpha }|$ for the
number of elements in $\Xi _{i}^{\alpha }$, which can be zero. Then the
vectors $\overrightarrow{A_{i}}$'s defined by $\overrightarrow{A_{i}}%
=((1^{a_{i}^{1}}),(1^{a_{i}^{2}}),\cdots ,(1^{a_{i}^{L}}))$ become one term
in the summation appear in the LHS. Furthermore, for each fixed such $%
\overrightarrow{A_{i}}$'s, there are $\overset{L}{\underset{\alpha =1}{\prod
}}\frac{d_{\alpha }!}{a_{1}^{\alpha }!\cdots a_{r}^{\alpha }!}$ possible
partition sets $\Xi _{i}^{\alpha }$'s. The equality holds.
\end{proof}

\section{The Case of Rows Implies the Conjecture}

\label{sec8} In this section, we discuss the case for a general partition $%
\overrightarrow{\mu }$, and reduce it to the case of rectangular ones.

We first define an equivalence relation on BMW algebra $C_{n}$: two elements
$x,y\in C_{n}$ are equivalent, denoted by $x\sim y$, if $\mathrm{tr}(xz)=%
\mathrm{tr}(yz)$ for all central elements $z\in C_{n}$. Obviously, if two
elements $x$ and $y$ are conjugate, say if there exist an invertible element
$g\in C_{n}$, such that $gxg^{-1}=y$, then $x\sim y$. As the algebra $C_{n}$
is semi-simple, two idempotents $p_{1}$ and $p_{2}$ are equivalent, if and
only if they give isomorphic representations of $C_{n}$.

Let $p_{\lambda }$ be a minimal path idempotent in $C_{n}$. Denote by $%
m_{\mu }=\underset{\lambda }{\sum }\chi _{\lambda }(\gamma _{\mu
})p_{\lambda }$, and also regard this as an element in the Grothendieck
group of representations of the Birman-Murakami-Wenzl algebra. The branching
rule \cite{BB} for Birman-Murakami-Wenzl algebra is
\begin{equation*}
p_{\lambda }\otimes 1=\sum_{\lambda ^{\prime }}p_{\lambda ^{\prime }},
\end{equation*}%
where the summation is over all partitions $\lambda ^{\prime }$ which is
either add one box to $\lambda $ or remove one box from $\lambda $. As the
characters $\chi _{A}(\gamma _{\mu })$ of Brauer algebra are all integers,
repeated using the branching rule lead the decomposition of the tensor
product of minimal idempotents:
\begin{equation}
m_{(\mu _{1})}\otimes m_{(\mu _{2})}\otimes \cdots \otimes m_{(\mu _{\ell
})}\sim \sum_{A}b_{A}p_{A},  \label{E:tensor}
\end{equation}%
where the summation is over all possible partitions $A$ and the multiplicity
$b_{A}$ are all integers. Furthermore, the integers $b_{A}$ are uniquely
determined by this equivalence relation, by multiplying both sides the
minimal central idempotents $\pi _{A}$ of $C_{n}$.

\begin{lemma}
The integers $b_{A}=\chi _{A}(\gamma _{\mu })$ for the characters of Brauer
algebras.
\end{lemma}

\begin{proof}
As the Birman-Murakami-Wenzl algebras are deformations of the Brauer
algebras, they share the same branching rules. Specialize (\ref{E:tensor})
to the Brauer algebras by fixing $x=1+\frac{t-t^{-1}}{q-q^{-1}}$ and let $t$
and $q$ goes to $1$, and using the isomorphism $Br_{n}\cong \mathrm{End}_{%
\mathfrak{so}(2N+1)}(V^{\otimes n})$ for $x=2N+1$, we get
\begin{equation}
\widetilde{m}_{(\mu _{1})}\otimes \widetilde{m}_{(\mu _{2})}\otimes \cdots
\otimes \widetilde{m}_{(\mu _{\ell })}\sim \sum_{A}b_{A}\widetilde{p}_{A},
\label{E:SObranching}
\end{equation}%
where $\widetilde{m}_{(\mu _{i})}=\underset{A\in \widehat{Br}_{\mu _{i}}}{%
\sum }\chi _{A}(\gamma _{(\mu _{i})})\widetilde{p}_{A}$ and $\widetilde{p}%
_{A}$ is a minimal idempotent in $\mathrm{\ End}_{\mathfrak{so}%
(2N+1)}(V^{\otimes n})$. Regard (\ref{E:SObranching}) as an equality in the
Grothendieck group of (finite dimensional representations) of the Lie group $%
SO(2N+1)$, and the character is given by
\begin{align*}
& \prod_{i=1}^{\ell }[\underset{h=0}{\overset{[\frac{\mu _{i}}{2}]}{\sum }}%
\sum_{\lambda \vdash \mu _{i}-2h}\chi _{\lambda }(\gamma _{(\mu
_{i})})sb_{\lambda }(z_{-N},z_{1-N},\cdots ,z_{-1},z_{0},z_{1},\cdots
,z_{N-1},z_{N})] \\
& =\prod_{i=1}^{\ell }p_{\mu _{i}}(z_{-N},z_{1-N},\cdots
,z_{-1},z_{0},z_{1},\cdots ,z_{N-1},z_{N}) \\
& =p_{\mu }(z_{-N},z_{1-N},\cdots ,z_{-1},z_{0},z_{1},\cdots ,z_{N-1},z_{N})
\\
& =\underset{h=0}{\overset{[\frac{|\mu |}{2}]}{\sum }}\sum_{\lambda \vdash
|\mu |-2h}\chi _{\lambda }(\gamma _{\mu })sb_{\lambda
}(z_{-N},z_{1-N},\cdots ,z_{-1},z_{0},z_{1},\cdots ,z_{N-1},z_{N}).
\end{align*}

Thus the two elements $\widetilde{m}_{(\mu _{1})}\otimes \widetilde{m}_{(\mu
_{2})}\otimes \cdots \otimes \widetilde{m}_{\mu _{\ell }}$ and $\widetilde{m}%
_{\mu }$ equal in the Grothendieck group of $SO(2N+1)$, which determines the
integers $b_{A}=\chi _{A}(\gamma _{\mu })$, i.e, we have
\begin{equation*}
m_{(\mu _{1})}\otimes m_{(\mu _{2})}\otimes \cdots \otimes m_{(\mu _{\ell
})}\sim m_{\mu }.
\end{equation*}
\end{proof}

Let $\overrightarrow{\ell }=(\ell _{1},\cdots ,\ell _{L})$, and let $%
\mathcal{L}_{\overrightarrow{\ell }}$ be closure of the cabling braid $\beta
_{\overrightarrow{\ell }}$, which is obtained by cabling the $\alpha $-th
component of $\beta $ into $\ell _{\alpha }$ parallel ones. Then we have
\begin{align*}
\mathrm{z}_{\overrightarrow{\mu }}\cdot Z_{\overrightarrow{\mu }}(\mathcal{L}%
)& =\sum_{\overrightarrow{A}}\chi _{\overrightarrow{A}}(\gamma _{%
\overrightarrow{\mu }})\mathrm{tr}(\beta _{\overrightarrow{\ell }}\cdot p_{%
\overrightarrow{A}}) \\
& =\mathrm{tr}(\beta _{\overrightarrow{\ell }}\cdot m_{\overrightarrow{\mu }%
}) \\
& =\mathrm{tr}[\beta _{\overrightarrow{\ell }}\cdot \underset{\alpha =1}{%
\overset{L}{\bigotimes }}(m_{(\mu _{1}^{\alpha })}\otimes \cdots \otimes
m_{(\mu _{\ell _{\alpha }}^{\alpha })})] \\
& =(\prod_{\alpha =1}^{L}\prod_{i=1}^{\ell _{\alpha }}\mu _{i}^{\alpha
})\cdot Z_{(\mu _{1}^{1}),(\mu _{2}^{1}),\cdots ,(\mu _{\ell _{1}}^{1})(\mu
_{1}^{2}),(\mu _{2}^{2}),\cdots ,(\mu _{\ell _{2}}^{2}),...,(\mu
_{1}^{L}),(\mu _{2}^{L}),\cdots ,(\mu _{\ell _{L}}^{L})}(\mathcal{L}_{%
\overrightarrow{\ell }})
\end{align*}%
and
\begin{equation}
\mathrm{z}_{\overrightarrow{\mu }}\cdot F_{\overrightarrow{\mu }}(\mathcal{L}%
)=(\prod_{\alpha =1}^{L}\prod_{i=1}^{\ell _{\alpha }}\mu _{i}^{\alpha
})\cdot F_{(\mu _{1}^{1}),(\mu _{2}^{1}),\cdots ,(\mu _{\ell _{1}}^{1}),(\mu
_{1}^{2}),(\mu _{2}^{2}),\cdots ,(\mu _{\ell _{2}}^{2}),...,(\mu
_{1}^{L}),(\mu _{2}^{L}),\cdots ,(\mu _{\ell _{L}}^{L})}(\mathcal{L}_{%
\overrightarrow{\ell }}).  \label{E:lmu1}
\end{equation}

Now the partition $(\mu _{1}^{1}),(\mu _{2}^{1}),\cdots ,(\mu _{\ell
_{1}}^{1}),(\mu _{1}^{2}),(\mu _{2}^{2}),\cdots ,(\mu _{\ell
_{2}}^{2}),...,(\mu _{1}^{L}),(\mu _{2}^{L}),\cdots ,(\mu _{\ell
_{L}}^{L})\in \mathcal{P}^{|\ell (\overrightarrow{\mu })|}$ has the property
that each component is of length one. In particular, it is rectangular, and
we have the following theorem.

\begin{theorem}
The Conjecture \ref{Main Conj} is true for all partition $\overrightarrow{\mu%
}$, if and only if it is true for rectangular $\overrightarrow{\mu}$, if and
only if it is true for $\overrightarrow{\mu}=(\mu^1,\cdots,\mu^L)$ such that
each $\mu^\alpha=(d_\alpha)$ is of length one.
\end{theorem}

(\ref{E:lmu1}) together with Proposition \ref{degree} implies Conjecture \ref%
{Degree Conj} (Degree Conj.), i.e, the degree estimate at $q=1$ is valid for
all partitions $\overrightarrow{\mu }$.

\begin{theorem}
\label{Thm8.3} Conjecture \ref{Degree Conj} is true for all links and all
partitions.
\end{theorem}

Theorem \ref{Thm8.3} implies that Conjecture \ref{Main Conj} is "true at $%
q=1 $" (Theorem \ref{Thm1.6}), i.e., the left hand side of Conjecture \ref%
{Main Conj} is regular at $q-1$. The situation at other roots of unity seems
to be more difficult. There are some torus knots and links examples verified
in Section \ref{sec5}, which can be treated as the conjecture at roots of
unity besides $1$.

\section{Estimation of Degree}

\label{sec9} Call a partition $\lambda \vdash n$ rectangular, if the Young
diagram of $\lambda =(\lambda _{1},\lambda _{2},\cdots ,\lambda _{\ell })$
is rectangular, i.e. $\lambda _{1}=\lambda _{2}=\cdots =\lambda _{\ell }$. A
rectangular partition is determined by its length $\ell $ and its size $n$.

Let $\delta _{n}=\sigma _{1}\sigma _{2}\cdots \sigma _{n-1}$ be a braid in $%
B_{n}$. Let $\ell $ be an integer dividing $n$ and write $a=n/\ell $, then
the braid $(\delta _{n})^{\ell }$ is associated to the rectangular partition
$\lambda =(a,a,\cdots ,a)\vdash n$. It is easy to see that $h((\delta
_{n})^{n})$ is in the center of $C_{n}$. Let $A$ be a partition of $n-2f$
for some integer $f$, and let $\pi _{A}$ be a minimal central idempotent in $%
C_{n}$, and let $p_{A}$ be a minimal idempotent such that $p_{A}\pi
_{A}=p_{A}$. Under the isomorphism
\begin{equation*}
C_{n}\cong \underset{A\in \widehat{Br}_{n}}{\bigoplus }M_{d_{A}\times d_{A}}(%
\mathbb{C}),
\end{equation*}%
the product $h((\delta _{n})^{n})\cdot \pi _{A}$ is a scaler matrix at the
block corresponding to $A$, and zero at other places. In Section \ref{sec3},
we know that $h((\delta _{n})^{n})\cdot \pi _{A}=q^{\kappa _{A}}t^{-2f}\pi
_{A}$, which implies that the eigenvalues of $h((\delta _{n})^{\ell })\cdot
\pi _{A}$ are either $0$ or $q^{\kappa _{A}/a}t^{-2f/a}$ times $n$-th roots
of unity. We conclude that $\mathrm{tr}(h((\delta _{n})^{\ell })\cdot
p_{A})=b_{A}\cdot q^{\kappa _{A}/a}t^{-2f/a}$ for some rational number $%
b_{A} $. Taking the specialization $q,t\rightarrow 1$, we obtain the value $%
b_{A}=\chi _{A}(\gamma _{\lambda })$ for the character $\chi _{A}$ of Brauer
algebra.

Now we compute $Z_{\overrightarrow{\lambda }}(\mathcal{L},q,t)$ for
rectangular partition $\overrightarrow{\lambda }$. Write $\overrightarrow{%
\lambda }=(\lambda ^{1},\cdots ,\lambda ^{L})$ such that $\lambda ^{\alpha
}=(a_{\alpha },\cdots ,a_{\alpha })=(a_{\alpha }{}^{\ell _{\alpha }})$ for
each $\alpha =1,2,\cdots ,L$. Our goal in this section is to estimate the $u$
degree of $F_{\overrightarrow{\lambda }}(\mathcal{L};,q,t)$ for $u=\log q$
and rectangular partition $\overrightarrow{\lambda }$.

\begin{definition}
Let $\overrightarrow{\tau }=(\tau _{1},\cdots ,\tau _{L})\in \mathbb{C}^{L}$
be a vector, define the framing dependent link invariants $W_{%
\overrightarrow{A}}(\mathcal{L},q,t,\overrightarrow{\tau }):=W_{%
\overrightarrow{A}}(\mathcal{L},q,t)q^{\underset{\alpha =1}{\overset{L}{\sum
}}\kappa _{A^{\alpha }}\tau _{\alpha }}t^{\underset{\alpha =1}{\overset{L}{%
\sum }}|A^{\alpha }|\tau _{\alpha }}$, and the framing dependent partition
function by
\begin{equation}
Z_{CS}^{SO}(\mathcal{L};q,t,\overrightarrow{\tau })=\sum_{\overrightarrow{%
\mu }\in \mathcal{P}^{L}}\frac{pb_{\overrightarrow{\mu }}}{\mathrm{z}_{%
\overrightarrow{\mu }}}\cdot \sum_{\overrightarrow{A}\in \widehat{Br}_{|%
\overrightarrow{\mu }|}}\chi _{\overrightarrow{A}}(\gamma _{\overrightarrow{%
\mu }})W_{\overrightarrow{A}}^{SO}(\mathcal{L};q,t,\overrightarrow{\tau }).
\end{equation}
\end{definition}

Similarly we define the free energy $F^{SO}(\mathcal{L};q,t,\overrightarrow{%
\tau})=\log Z_{CS}^{SO}(\mathcal{L};q,t,\overrightarrow{\tau})$ and the
coefficients $F_{\overrightarrow{\mu}}^{SO}(\mathcal{L};q,t,\overrightarrow{%
\tau})$ and $W_{\overrightarrow{\mu}}^{SO}(\mathcal{L};q,t,\overrightarrow{%
\tau})$ as before, replacing the link invariants by the framing dependent
invariants. The specialization $\overrightarrow{\tau}=0$ gives the framing
independent invariants.

We compute the partition functions at the special values $\tau _{\alpha
}=w_{\alpha }+\frac{1}{a_{\alpha }}$ for $w_{\alpha }\in \mathbb{Z}$, by
taking a braid $\beta (\overrightarrow{w})$ with writhe number $w_{\alpha }$
on each component $\mathcal{K}_{\alpha }$ of$\ \mathcal{L}$. Let $\mathcal{L}%
_{\overrightarrow{n},\overrightarrow{w}}^{twist}$ be the closure of the
product of the cabling braid $\beta _{\overrightarrow{n}}(\overrightarrow{w}%
) $ of $\beta (\overrightarrow{w})$ and the braid $(\omega _{\lambda
^{1}}\otimes \cdots \otimes \omega _{\lambda ^{L}})$, where $n_{\alpha
}=a_{\alpha }\ell _{\alpha }$ and $\omega _{\lambda ^{\alpha }}=(\delta
_{n_{\alpha }})^{\ell _{\alpha }}$. The diagrams below provide an example to
illustrate the twisted cabling process in the case $a=2$, $\ell =1$ and $%
n=a\ell =2$. Suppose $\mathcal{L}$ is a braid in Picture a, which represents
a knot with writhe number $w=4$. Picture b is obtained by cabling each
component into two strands. The twist $\omega _{\lambda }$ is then as in the
bottom of Picture c, which add a crossing to Picture b. The final twisted
cabling link $\mathcal{L}_{\overrightarrow{n},\overrightarrow{w}}^{twist}$
is the closure of the braid in Picture d.
\begin{equation*}
{\includegraphics[height=2.0721in, width=1.5013in]
{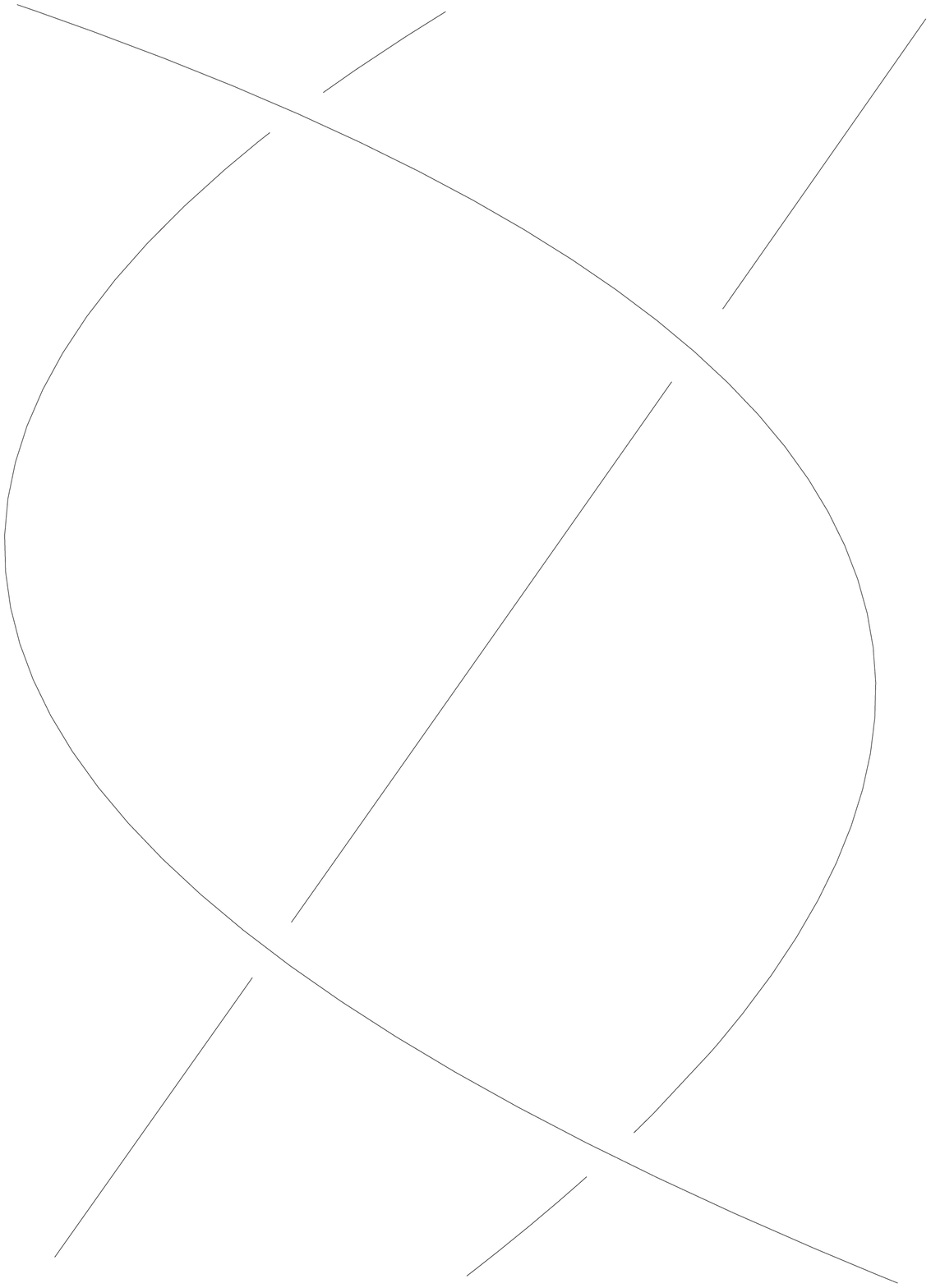}}
\end{equation*}%
\ \ \ \ \ \ \ \ \ \ \ \ \ \ \ \ \ \ \ \ \ \ \ \ \ \ \ \ \ \ \ \ \ \ \ \ \ \
\ \ \ \ \ \ \ \ \ \ \ \ \ \ \ \ \ \ \ \ Picture a

\vskip 0.4cm

\begin{equation*}
{\includegraphics[height=2.0903in, width=1.7504in]
{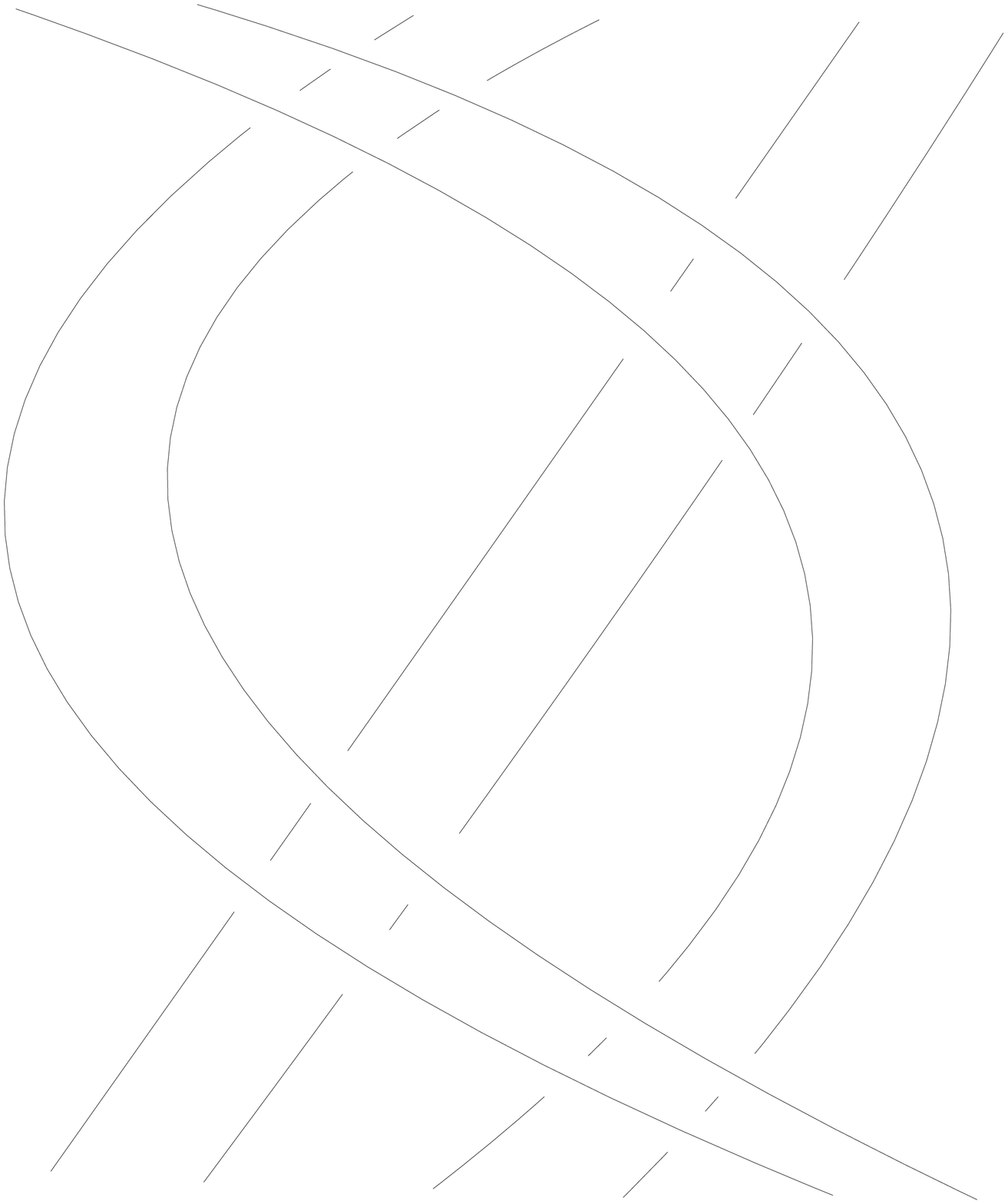}} \hskip 0.5cm {%
\includegraphics[ height=2.0998in, width=1.4356in]
{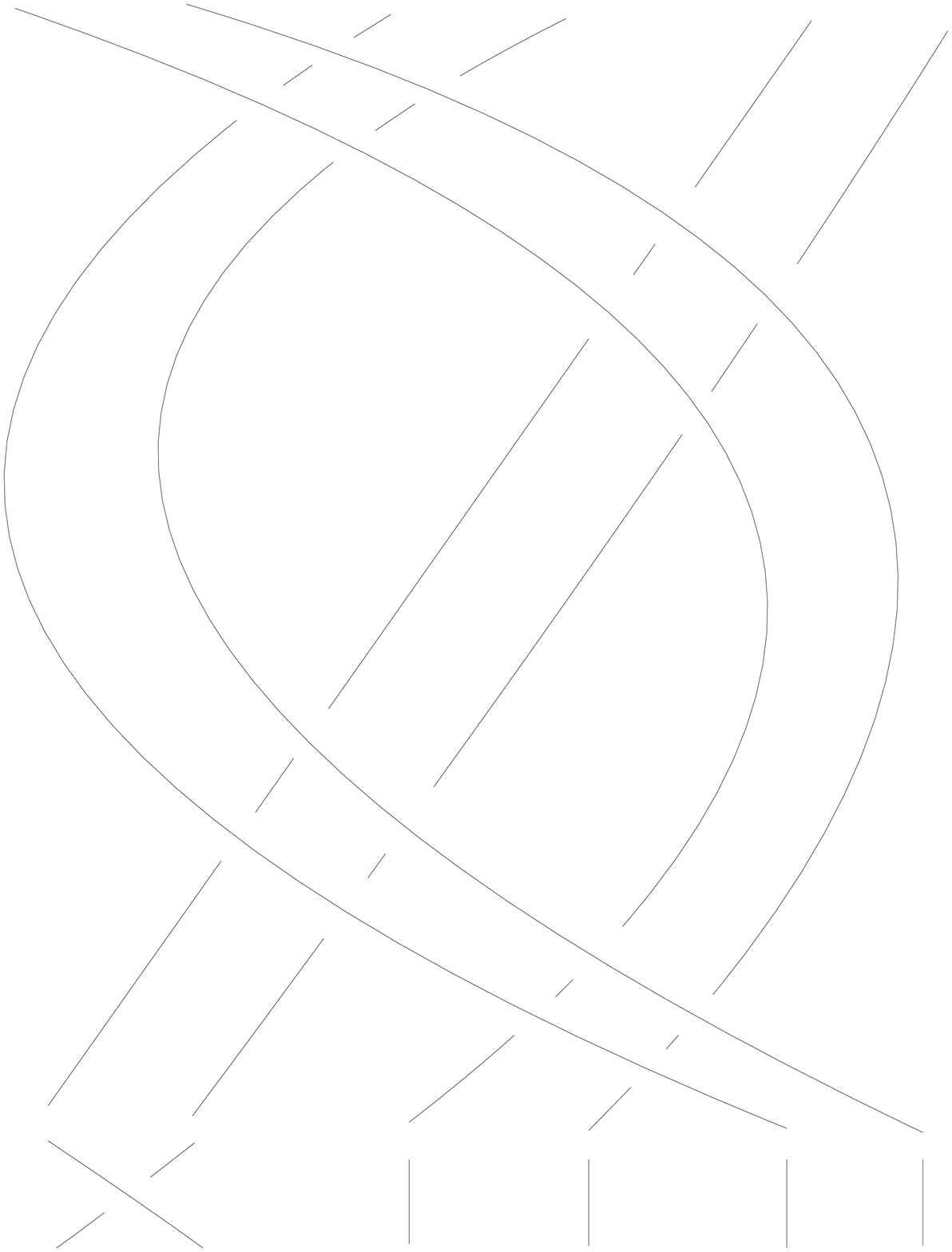}} \hskip 0.8cm {%
\includegraphics[ height=2.0678in, width=1.3578in]
{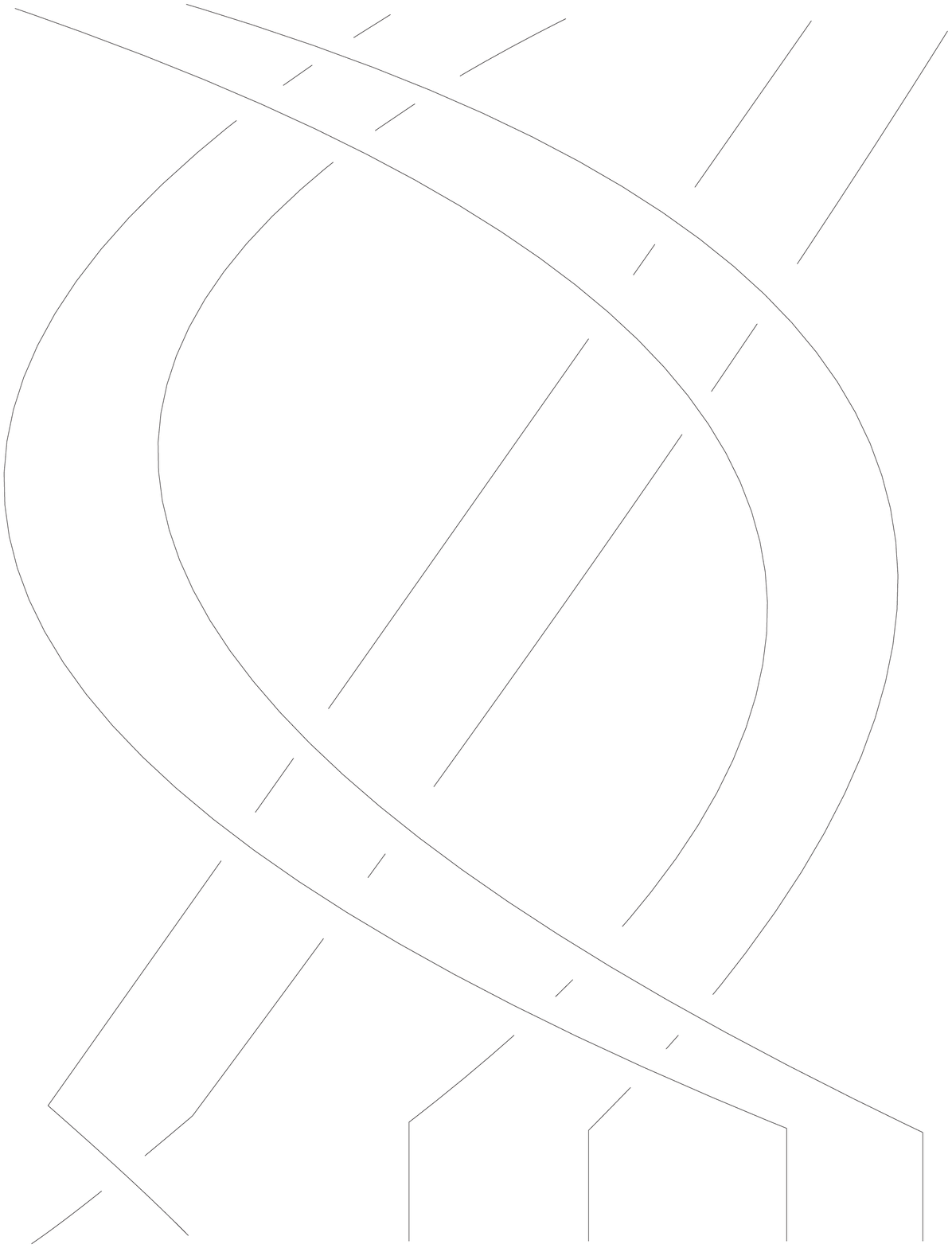}}
\end{equation*}

\ \ \ \ \ \ \ \ \ \ \ \ \ \ \ \ \ \ \ \ Picture b \ \ \ \ \ \ \ \ \ \ \ \ \
\ \ \ \ \ \ \ \ \ \ \ Picture c \ \ \ \ \ \ \ \ \ \ \ \ \ \ \ \ \ \ \ \ \ \
\ Picture d

\vskip 0.4cm

The link $\mathcal{L}_{\overrightarrow{n},\overrightarrow{w}}^{twist}$ is of
$\ell (\overrightarrow{\lambda })=\ell _{1}+\ell _{2}+\cdots +\ell _{L}$
components, and there are $\ell _{i}$ components of writhe $w_{\alpha
}a_{\alpha }^{2}+a_{\alpha }-1$.

\begin{align*}
Z_{\overrightarrow{\lambda }}^{SO}(\mathcal{L};q,t,\overrightarrow{\tau })& =%
\frac{1}{\mathrm{z}_{\overrightarrow{\lambda }}}\cdot \sum_{\overrightarrow{A%
}\in \widehat{Br}_{|\overrightarrow{\lambda }|}}\chi _{\overrightarrow{A}%
}(\gamma _{\overrightarrow{\lambda }})W_{\overrightarrow{A}}^{SO}(\mathcal{L}%
;q,t,\overrightarrow{\tau }) \\
& =\frac{1}{\mathrm{z}_{\overrightarrow{\lambda }}}\cdot \mathrm{tr}[\beta _{%
\overrightarrow{n}}(\overrightarrow{w})\cdot \sum_{\overrightarrow{A}\in
\widehat{Br}_{|\overrightarrow{\lambda }|}}\chi _{\overrightarrow{A}}(\gamma
_{\overrightarrow{\lambda }})q^{\overset{L}{\underset{\alpha =1}{\sum }}%
\frac{\kappa _{A^{\alpha }}}{a_{\alpha }}}t^{\overset{L}{\underset{\alpha =1}%
{\sum }}\frac{|A^{\alpha }|}{a_{\alpha }}}\cdot p_{\overrightarrow{A}}] \\
& =\frac{t^{\overset{L}{\underset{\alpha =1}{\sum }}\ell _{\alpha }}}{%
\mathrm{z}_{\overrightarrow{\lambda }}}\cdot \mathrm{tr}(\beta _{%
\overrightarrow{n}}(\overrightarrow{w})\cdot (\omega _{\lambda ^{1}}\otimes
\cdots \otimes \omega _{\lambda ^{L}})) \\
& =\frac{t^{\overset{L}{\underset{\alpha =1}{\sum }}a_{\alpha }\ell _{\alpha
}(w_{\alpha }a_{\alpha }+1)}}{\mathrm{z}_{\overrightarrow{\lambda }}}\cdot
W_{(1)^{\ell (\overrightarrow{\lambda })}}(\mathcal{L}_{\overrightarrow{n},%
\overrightarrow{w}}^{twist},q,t).
\end{align*}%
As in the proof of Theorem \ref{Thm7.5}, we get
\begin{equation}
F_{\overrightarrow{\lambda }}^{SO}(\mathcal{L};q,t,\overrightarrow{\tau })=%
\frac{t^{\overset{L}{\underset{\alpha =1}{\sum }}a_{\alpha }\ell _{\alpha
}(w_{\alpha }a_{\alpha }+1)}}{\underset{\alpha =1}{\overset{L}{\prod }}\ell
_{\alpha }!a_{\alpha }^{\ell _{\alpha }}}\cdot F_{(1)^{\ell (\overrightarrow{%
\lambda })}}^{SO}(\mathcal{L}_{\overrightarrow{n},\overrightarrow{w}%
}^{twist};q,t).  \label{E:Rectangle}
\end{equation}%
In particular, we get the following proposition.

\begin{proposition}
For a rectangular partition $\overrightarrow{\lambda }$ such that $\mu
^{\alpha }=(a_{\alpha }{}^{\ell _{\alpha }})$, and any tube of integers $%
\overrightarrow{w}=(w_{1},\cdots ,w_{L})$, we have
\begin{equation*}
(q-q^{-1})^{2-\ell (\overrightarrow{\lambda })}\cdot F_{\overrightarrow{%
\lambda }}^{SO}(\mathcal{L};q,t,\overrightarrow{\tau })\in \mathbb{Q}%
[q-q^{-1}][t,t^{-1}]
\end{equation*}%
for $\overrightarrow{\tau }=(w_{1}+\frac{1}{a_{1}},w_{2}+\frac{1}{a_{2}}%
,\cdots ,w_{L}+\frac{1}{a_{L}})$.
\end{proposition}

Consider the embedding $\mathbb{Q}(q)[t,t^{-1}]\hookrightarrow \mathbb{Q}%
[[T]]((u))$ via the changing of variables $q=e^{u}$ and $t=e^{T}$, we can
expend the rational function $F_{\overrightarrow{\lambda }}^{SO}(\mathcal{L}%
;q,t,\overrightarrow{\tau })$ into a formal power series in variables $u$
and $T$
\begin{equation*}
F_{\overrightarrow{\lambda }}^{SO}(\mathcal{L};e^{u},e^{T},\overrightarrow{%
\tau })=\sum_{k=0}^{\infty }\sum_{i\geq -||\overrightarrow{\lambda }%
||}P_{k,i}(\overrightarrow{\tau })T^{k}u^{i}
\end{equation*}%
with coefficients $P_{k,i}\in \mathbb{Q}[\tau ^{1},\cdots ,\tau ^{L}]$.

The above proposition implies that the coefficients $P_{k,i}(\overrightarrow{%
\tau })$ for $i<\ell (\overrightarrow{\lambda })-2$ vanish when every $\tau
_{k}-\frac{1}{a_{k}}$ takes arbitrary integer values, which is possible only
when the polynomials $P_{k,i}(\overrightarrow{\tau })$ for $i<\ell (%
\overrightarrow{\lambda })-2$ are zero polynomials (a lattice is Zariski
dense). Now specialize at the framing $\tau _{1}=\tau _{2}=\cdots =\tau
_{L}=0$ leads to the following theorem.

\begin{proposition}
\label{degree} Suppose $\overrightarrow{\lambda }$ is a rectangular
partition, the formal power series $F_{\overrightarrow{\lambda }}^{SO}(%
\mathcal{L};e^{u},t)$ and $g_{\overrightarrow{\lambda }}(\mathcal{L}%
;e^{u},t) $ in the valuation field $\mathbb{Q}(t)((u))$ has $u$-valuation
greater or equal to $\ell (\overrightarrow{\lambda })-2$.
\end{proposition}

\section{Appendix}

\label{sec10}

\subsection{The Case of Unknot}

In this Appendix, we calculate the $F^{SO}(\bigcirc ^{L};q,t)$. We only deal
with the case of unknot, i.e., $L=1$, since the general case $L\geq 1$ can
be done exactly the same way, except that the notation will be more
complicated.

\begin{proposition}
\begin{equation*}
\sum_{A\in \widehat{Br}_{|\mu |}}\chi _{A}(\gamma _{\mu
})W_{A}^{SO}(\bigcirc ;q,t))=\prod_{i=1}^{\ell (\mu )}[1+\frac{t^{\mu
_{i}}-t^{-\mu _{i}}}{q^{\mu _{i}}-q^{-\mu _{i}}}].
\end{equation*}
\end{proposition}

\begin{proof}
Let $t=q^{2N}$ and compare with the quantum group definition of the colored
Kauffman polynomials,
\begin{align*}
& \sum_{A\in \widehat{Br}_{|\mu |}}\chi _{A}(\gamma _{\mu
})W_{A}^{SO}(\bigcirc ;q,q^{2N})) \\
& =\sum_{A\in \widehat{Br}_{|\mu |}}\chi _{A}(\gamma _{\mu
})tr_{V_{A}}(K_{2\rho }) \\
& =\sum_{A\in \widehat{Br}_{|\mu |}}\chi _{A}(\gamma _{\mu })sb_{A} \\
& =pb_{\mu }(q^{1-2N},q^{3-2N},\cdots ,q^{-1},1,q,\cdots ,q^{2N-3},q^{2N-1})
\\
& =\prod_{i=1}^{\ell (\mu )}[1+\frac{t^{\mu _{i}}-t^{-\mu _{i}}}{q^{\mu
_{i}}-q^{-\mu _{i}}}].
\end{align*}%
As both the left and right hand side of the equation in the proposition are
rational functions in $t$, and they agree for arbitrary sufficiently large $%
N $, they must coincide.
\end{proof}

\begin{proposition}
\label{Prop10.2}
\begin{equation}
Z_{CS}^{SO}(\bigcirc ;q,t)=\exp (\sum_{k=1}^{+\infty }\frac{1}{k}(1+\frac{%
t^{k}-t^{-k}}{q^{k}-q^{-k}})pb_{k}).
\end{equation}
\end{proposition}

\begin{align*}
Z_{CS}^{SO}(\bigcirc ;q,t)=& \sum_{\lambda \in \mathcal{P}}\frac{1}{\mathrm{z%
}_{\lambda }}\cdot \prod_{i=1}^{\ell (\lambda )}[1+\frac{t^{\lambda
_{i}}-t^{-\lambda _{i}}}{q^{\lambda _{i}}-q^{-\lambda _{i}}}]\cdot
pb_{\lambda } \\
& =\sum_{n=1}^{+\infty }\sum_{l(\lambda )=n}\frac{1}{\mathrm{z}_{\lambda }}%
\cdot \prod_{i=1}^{\ell (\lambda )}[1+\frac{t^{\lambda _{i}}-t^{-\lambda
_{i}}}{q^{\lambda _{i}}-q^{-\lambda _{i}}}]\cdot pb_{\lambda } \\
& =\sum_{n=1}^{+\infty }\sum_{\lambda _{1},\cdots ,\lambda _{n}=1}^{+\infty }%
\frac{1}{n!}\prod_{i=1}^{n}\frac{1}{\lambda _{i}}\prod_{i=1}^{n}[1+\frac{%
t^{\lambda _{i}}-t^{-\lambda _{i}}}{q^{\lambda _{i}}-q^{-\lambda _{i}}}%
]\cdot pb_{\lambda } \\
& =\sum_{n=1}^{+\infty }\frac{1}{n!}[\sum_{k=1}^{+\infty }\frac{1}{k}(1+%
\frac{t^{k}-t^{-k}}{q^{k}-q^{-k}})pb_{k}]^{n} \\
& =\exp (\sum_{k=1}^{+\infty }\frac{1}{k}(1+\frac{t^{k}-t^{-k}}{q^{k}-q^{-k}}%
)pb_{k}).
\end{align*}%
So we get the expression of free energy
\begin{equation*}
F^{SO}(\bigcirc ;q,t)=\sum_{k=1}^{+\infty }\frac{1}{k}(1+\frac{t^{k}-t^{-k}}{%
q^{k}-q^{-k}})pb_{k}.
\end{equation*}

\begin{remark}
This expression appeared in \cite{BR}, but is computed from the path
integral definition of the Chern-Simons partition function. Our derivation
is based on our mathematical definition in terms of quantum group invariants
and representations of Brauer algebra.
\end{remark}

\subsection{An Alternative Definition of Colored Kauffman Polynomial via
Markov Trace and Hopf Link}

The quantum group approach to the knot/link theory has produce a lot of
invariants via the representation theory. However, the difficulty is that
the calculation involved are usually very complicated to compute.
Fortunately, only quantum trace are essentially used. This enable us to find
a combinatorial way instead of the quantum group method. Birman-Wenzl \cite%
{BW} and Wenzl \cite{Wen1} introduce a Markov trace definition. We will
briefly introduce their construction here.

There is a well-defined \emph{Markov trace} $\mathrm{tr}$ on the union of
BMW algebra $C_{n}$ with the following properties.

\begin{enumerate}
\item $\mathrm{tr}(h(\alpha _{1})h(\alpha _{2}))=\mathrm{tr}(h(\alpha
_{2})h(\alpha _{1}))$ for any $\alpha _{i}\in B_{n}$

\item $\mathrm{tr}(h(\beta )g_{n}^{\pm 1})=\frac{t^{\pm 1}}{x}\mathrm{tr}%
(h(\beta ))$ for any $\beta \in B_{n}$

\item $\mathrm{tr}(1)=1$

\item $\mathrm{tr}(h(\beta ))=x^{1-n}t^{2\underset{i<j}{\sum }lk_{ij}}K(%
\widehat{\beta },q,t)$,

where $\widehat{\beta }$ is the link by closing the braid $\beta \in B_{n}$
and $K(\mathcal{L},q,t)$ is the classic Kauffman polynomial of the link $%
\mathcal{L}$.
\end{enumerate}

First normalize the trace by setting

\begin{equation*}
\mathrm{Tr}(\xi )=x^{n}\cdot \mathrm{tr}(\xi )\text{ for }\xi \in C_{n}\text{%
.}
\end{equation*}

Let $\mathcal{L}$ be a link with $L$ components $\mathcal{K}_{\alpha }$, $%
\alpha =1,\ldots ,L$, represented by the closure of $\beta \in B_{m}$. We
associate each $\mathcal{K}_{\alpha }$ an irreducible representation $%
V_{A^{\alpha }}$ of quantized universal enveloping algebra $U_{q}(\mathfrak{%
so}(2N+1))$. Let $p_{\alpha }\in C_{d_{\alpha }},\alpha =1,\cdots ,L$ be $L$
minimal idempotents corresponding to the irreducible representations $%
V_{A^{1}},\cdots ,V_{A^{L}}$, where $A^{\alpha }$ denote the partition of $%
|A^{\alpha }|=d_{\alpha }$ labelling $V_{A^{\alpha }}$. Denote $%
\overrightarrow{d}=(d_{1},\cdots ,d_{L})$ and let $i_{1},\cdots ,i_{m}$ be
integers such that $i_{k}=\alpha $ if the $k$-th strand of $\beta $ belongs
to the $\alpha $-th component of $\mathcal{L}$. Let $\beta _{\overrightarrow{%
d}}$ be the cabling braid of $\beta $, replacing the $k$-th strand of $\beta
$ by $d_{i_{k}}$ parallel ones. Then
\begin{equation}
W_{\overrightarrow{A}}^{SO}(\mathcal{L};q,t)=q^{-\overset{L}{\underset{%
\alpha =1}{\sum }}\kappa _{A^{\alpha }}w(\mathcal{K}_{\alpha })}t^{-\overset{%
L}{\underset{\alpha =1}{\sum }}|A^{\alpha }|w(\mathcal{K}_{\alpha })}\cdot
\mathrm{Tr}\left( h(\beta \mathcal{_{\overrightarrow{d}}})\cdot
(p_{i_{1}}\otimes \cdots \otimes p_{i_{m}})\right) \text{.}
\end{equation}

Now we look at a concrete example to illustrate this method.

Let $\mathcal{L}$ be the Hopf link, represented by the braid $\beta
=g_{1}^{2}$. Set $z=q-q^{-1}$. Easy to get $W_{(1),(1)}^{SO}(\mathcal{L})=x(%
\frac{t-t^{-1}}{z}+1+z(t-t^{-1}))$. Let $\mu $ be a partition of $2$, and
let $A$ be a partition of $2$ or $0$ labelling the irreducible
representations of Brauer algebra $Br_{2}$. The character table reads $\chi
_{(1,1)}(\gamma _{(2)})=-1$ and $\chi _{(2)}(\gamma _{(2)})=\chi
_{(1,1)}(\gamma _{(1,1)})=\chi _{(1,1)}(\gamma _{(1,1)})=1$. The
representation labelled by $A=(2)$ is the trivial representation.

We want to compute $W_{(1),(2)}^{SO}(\mathcal{L})$. The minimal idempotents
(studied by \cite{BB}) in $C_{2}$ are%
\begin{equation*}
p_{(2)}=(\frac{q^{-1}+g_{2}}{q+q^{-1}})(1-x^{-1}e_{2})\text{, }p_{(1)^{2}}=(%
\frac{q-g_{2}}{q+q^{-1}})(1-x^{-1}e_{2})\text{ and }p_{\phi }=x^{-1}e_{2}%
\text{,}
\end{equation*}%
where $\phi $ is the empty partition.

Denote the cabling of $\beta $ by $\beta _{1,2}$, which is given by%
\begin{equation*}
h(\beta _{1,2})=g_{1}g_{2}^{2}g_{1}\text{.}
\end{equation*}

By using the definition of BMW algebras $C_{n}$\ and the properties of
Markov trace. We have the following formulas for the twisted cabling braids%
\begin{align*}
\mathrm{tr}(h(\beta _{1,2})\cdot g_{2})& =\mathrm{tr}%
(g_{1}g_{2}^{2}g_{1}g_{2}) \\
& =\mathrm{tr}(g_{1}^{4}g_{2}) \\
& =\frac{t}{x}\mathrm{tr}(g_{1}^{4}) \\
& =\frac{t}{x}\mathrm{tr}%
((z+t^{-1})g_{1}^{3}+(1-t^{-1}z)g_{1}^{2}-t^{-1}g_{1}) \\
& =\frac{t}{x^{2}}((z+t^{-1})(2t-t^{-1}+(1-t^{-2})z+(t-t^{-1})z^{2}) \\
& +(1-t^{-1}z)(\frac{t-t^{-1}}{z}+1+z(t-t^{-1}))-1) \\
& =\frac{t}{x^{2}}(\frac{t-t^{-1}}{z}%
+1+(-2t^{-1}+3t-t^{-3})z+(1-t^{-2})z^{2}+(t-t^{-1})z^{3})\text{,}
\end{align*}

where we used property $(P2)$ of BMW algebra $C_{n}$ as well as the classic
Kauffman polynomial of Trefoil knot and Hopf link.

Similarly, we have
\begin{equation*}
\mathrm{tr}(h(\beta _{1,2}))=\frac{(x+zt-zt^{-1})^{2}}{x^{2}},
\end{equation*}%
and
\begin{equation*}
\mathrm{tr}(h(\beta _{1,2})\cdot e_{2})=\frac{1}{x}\text{,}
\end{equation*}%
where $h(\beta _{1,2})\cdot e_{2}$ is actually the image of a link of the
disjoint union of two unknots.

As
\begin{eqnarray*}
p_{(2)}-p_{(1)^{2}}+p_{\phi } &=&\frac{-z}{q+q^{-1}}+\frac{2g}{q+q^{-1}}+%
\frac{x^{-1}(z+q+q^{-1}-2g_{2})e_{2}}{q+q^{-1}} \\
&=&\frac{-z}{q+q^{-1}}+\frac{2g}{q+q^{-1}}+\frac{(z+q+q^{-1}-2t^{-1})e_{2}}{%
x(q+q^{-1})}
\end{eqnarray*}%
Then we have%
\begin{eqnarray*}
2Z_{(1),(2)}(\mathcal{L};q,t) &=&W_{(1),(2)}(\mathcal{L};q,t)-W_{(1),(1,1)}(%
\mathcal{L};q,t)+W_{(1),(0)}(\mathcal{L};q,t) \\
&=&x^{3}\mathrm{tr}[h(\beta _{1,2})\cdot (p_{(1)}\otimes
(p_{(2)}-p_{(1)^{2}}+p_{\phi }))] \\
&=&\frac{x^{3}}{q+q^{-1}}\cdot \lbrack 2\mathrm{tr}(h(\beta _{1,2})\cdot
g_{2})-z\mathrm{tr}(h(\beta _{1,2}))+\frac{z+q+q^{-1}-2t^{-1}}{x}\mathrm{tr}%
(h(\beta _{1,2})\cdot e_{2})] \\
&=&\frac{x}{q+q^{-1}}[\frac{t^{2}-t^{-2}}{z}%
+(q+q^{-1})+(t^{2}-t^{-2})(z^{3}+4z)]
\end{eqnarray*}%
and
\begin{align*}
2Z_{(1)}(\bigcirc )Z_{(2)}(\bigcirc )& =x(1+\frac{t^{2}-t^{-2}}{q^{2}-q^{-2}}%
) \\
& =\frac{x}{q+q^{-1}}(q+q^{-1}+\frac{t^{2}-t^{-2}}{z}),
\end{align*}

Thus we have
\begin{align*}
2F_{(1),(2)}(\mathcal{L},q,t)& =2Z_{(1),(2)}(\mathcal{L})-2Z_{(1)}(\bigcirc
)Z_{(2)}(\bigcirc ) \\
& =\frac{x}{(q+q^{-1})}(t^{2}-t^{-2})(z^{3}+4z) \\
& =(q+q^{-1})(t^{2}-t^{-2})[z+(t-t^{-1})]
\end{align*}%
and
\begin{equation*}
\frac{2(q-q^{-1})^{2}F_{(1),(2)}}{(q-q^{-1})(q^{2}-q^{-2})}%
=(t^{2}-t^{-2})[(t-t^{-1})+z]\in \mathbb{Z}[z][t,t^{-1}]
\end{equation*}%
as predicted in the Conjecture \ref{Main Conj}. Actually this example has
already been discussed in Case C of Example 1 in Section \ref{sec5}.

\subsection{Explicit Computation of Quantum Trace for Orthogonal Quantum
Group}

The universal matrix $\check{\mathcal{R}}$ acting on $V\otimes V$ for the
natural representation of $U_{q}(\mathfrak{so}(2N+1))$ on $V$ is given by
Turaev \cite{Tur}:

\begin{align*}
\check{\mathcal{R}}=& q\sum_{i\neq N+1}E_{i,i}\otimes
E_{i,i}+E_{N+1,N+1}\otimes E_{N+1,N+1}+\sum_{j}\underset{i\neq 2N+2-j}{%
\sum_{i\neq j}}E_{j,i}\otimes E_{i,j} \\
& +q^{-1}\sum_{i\neq N+1}E_{2N+2-i,i}\otimes
E_{i,2N+2-i}+(q-q^{-1})\sum_{i<j}E_{i,i}\otimes E_{j,j} \\
& -(q-q^{-1})\sum_{i<j}q^{\overline{i}-\overline{j}}E_{2N+2-j,i}\otimes
E_{j,2N+2-i}\text{,}
\end{align*}%
where $E_{i,j}$ is the $(2N+1)\times (2N+1)$ matrix with%
\begin{equation*}
(E_{i,j})_{kl}=\left\{
\begin{array}{c}
1 \\
0%
\end{array}%
\right.
\begin{array}{c}
(k,l)=(i,j) \\
elsewhere%
\end{array}%
\end{equation*}

and

\begin{equation*}
\overline{i}=\left\{
\begin{array}{c}
i+\frac{1}{2} \\
i \\
i-\frac{1}{2}%
\end{array}%
\right.
\begin{array}{c}
1\leq i\leq N \\
i=N+1 \\
N+2\leq i\leq 2N+1%
\end{array}%
\text{.}
\end{equation*}

The enhancement of $\check{\mathcal{R}}$, $K_{2\rho }$ is given by
\begin{equation*}
K_{2\rho }(v_{i})=\left\{
\begin{array}{c}
q^{2i-1-2N}v_{i} \\
v_{i} \\
q^{2i-3-2N}v_{i}%
\end{array}%
\right.
\begin{array}{c}
1\leq i\leq N \\
i=N+1 \\
N+2\leq i\leq 2N+1%
\end{array}%
\text{.}
\end{equation*}

Then we compute the $\theta _{V}=\mathrm{tr}_{V}\check{\mathcal{R}}_{V,V}$
as follows

\begin{align*}
\mathrm{tr}_{V}\check{\mathcal{R}}_{V,V}& =q\sum_{i\neq N+1}\mathrm{tr}%
(E_{i,i}K_{2\rho })\cdot E_{i,i}+\mathrm{tr}(E_{N+1,N+1}K_{2\rho })\cdot
E_{N+1,N+1} \\
& +\sum_{j}\underset{i\neq 2N+2-j}{\sum_{i\neq j}}\mathrm{tr}%
(E_{i,j}K_{2\rho })\cdot E_{j,i}+q^{-1}\sum_{i\neq N+1}\mathrm{tr}%
(E_{i,2N+2-i}K_{2\rho })\cdot E_{2N+2-i,i} \\
& +(q-q^{-1})\sum_{i<j}\mathrm{tr}(E_{j,j}K_{2\rho })\cdot
E_{i,i}-(q-q^{-1})\sum_{i<j}q^{\overline{i}-\overline{j}}\mathrm{tr}%
(E_{j,2N+2-i}K_{2\rho })\cdot E_{2N+2-j,i} \\
& =q\left( \overset{N}{\sum_{i=1}}q^{2i-1-2N}E_{i,i}+\overset{2N+1}{%
\sum_{i=N+2}}q^{2i-3-2N}E_{i,i}\right) +E_{N+1,N+1} \\
& +(q-q^{-1})\left( \underset{j=1}{\overset{N}{\sum }}\overset{j-1}{%
\sum_{i=1}}q^{2j-1-2N}E_{i,i}+\overset{N}{\sum_{i=1}}E_{i,i}+\underset{j=N+2}%
{\overset{2N+1}{\sum }}\overset{j-1}{\sum_{i=1}}q^{2j-3-2N}E_{i,i}\right) \\
& -(q-q^{-1})\overset{2N+2}{\sum_{j=N+2}}q^{\overline{2N+2-j}-\overline{j}%
}q^{2j-3-2N}E_{2N+2-j,2N+2-j}\text{.}
\end{align*}

Then we need to check the action of $\mathrm{tr}_{V}\check{\mathcal{R}}%
_{V,V} $ on those basis $v_{i}$'s.

\begin{enumerate}
\item For $1\leq k\leq N$, we have
\begin{eqnarray*}
\mathrm{tr}_{V}\check{\mathcal{R}}_{V,V}(v_{k}) &=&q\cdot
q^{2k-1-2N}v_{k}+(q-q^{-1})\left( \underset{j=k+1}{\overset{N}{\sum }}%
q^{2j-1-2N}v_{k}+v_{k}+\underset{j=N+2}{\overset{2N+1}{\sum }}%
q^{2j-3-2N}v_{k}\right) \\
&&-(q-q^{-1})q^{\overline{k}-\overline{2N+2-k}}q^{2(2N+2-k)-3-2N}v_{k} \\
&=&(q^{2k-2N}+(q-q^{-1})(q^{2k+1-2N}\frac{q^{2N-2k}-1}{q^{2}-1}+1+q\frac{%
q^{2N}-1}{q^{2}-1} \\
&&-q^{k+\frac{1}{2}-(2N+2-k-\frac{1}{2})}q^{2N+1-2k}))v_{k} \\
&=&q^{2N}v_{k}\text{.}
\end{eqnarray*}

\item For $k=N+1$, we have
\begin{eqnarray*}
\mathrm{tr}_{V}\check{\mathcal{R}}_{V,V}(v_{N+1})
&=&v_{N+1}+(q-q^{-1})\left( \underset{j=N+2}{\overset{2N+1}{\sum }}%
q^{2j-3-2N}v_{N+1}\right) \\
&=&(1+(q-q^{-1})(q\frac{q^{2N}-1}{q^{2}-1}))v_{N+1} \\
&=&q^{2N}v_{N+1}\text{.}
\end{eqnarray*}

\item For $N+2\leq k\leq 2N+1$, we have
\begin{eqnarray*}
\mathrm{tr}_{V}\check{\mathcal{R}}_{V,V}(v_{k}) &=&q\cdot
q^{2k-3-2N}v_{k}+(q-q^{-1})\left( \underset{j=k+1}{\overset{2N+1}{\sum }}%
q^{2j-3-2N}v_{k}\right) \\
&=&(q^{2k-2-2N}+(q-q^{-1})(q^{2k-1-2N}\frac{q^{4N+2-2k}-1}{q^{2}-1}))v_{k} \\
&=&q^{2N}v_{k}\text{.}
\end{eqnarray*}%
Thus we get the following result%
\begin{equation*}
\theta _{V}=q^{2N}\mathrm{id}_{V}\text{.}
\end{equation*}
\end{enumerate}

\subsection{Character tables of Brauer Algebras and Type-B Schur Functions}

Here are some character tables for Brauer algebras. Write $pb_{\lambda }=%
\underset{A\in \widehat{Br}_{|\lambda |}}{\sum }\chi _{A}(\gamma _{\lambda
})sb_{A}$, and we compute the character table by the following formula
proved by A. Ram as Theorem 5.1 in \cite{Ram2}:%
\begin{equation*}
\chi _{\lambda }(\gamma _{\mu })=\underset{\nu \supset \lambda }{\underset{%
\nu \vdash |\mu |}{\sum }}(\underset{\beta \text{{}}even}{\sum }c_{\lambda
\beta }^{\nu })\chi _{\nu }^{S_{|\mu |}}(\gamma _{\mu })\text{.}
\end{equation*}

where $c_{\lambda \beta }^{\nu }$'s are called the Littlewood-Richardson
coefficients and defined via type-A Schur functions as follows

\begin{equation*}
s_{\alpha }s_{\beta }=\underset{|\gamma |=|\alpha |+|\beta |}{\sum }%
c_{\alpha \beta }^{\gamma }s_{\gamma }
\end{equation*}

Recall the character table for permutation group $S_{n}$

\begin{equation*}
\begin{array}{ccc}
\chi & (2) & (1,1) \\
(2) & 1 & 1 \\
(1,1) & -1 & 1%
\end{array}%
\end{equation*}

\begin{equation*}
\begin{array}{cccc}
\chi & (3) & (2,1) & (1,1,1) \\
(3) & 1 & 1 & 1 \\
(2,1) & -1 & 0 & 2 \\
(1,1,1) & 1 & -1 & 1%
\end{array}%
\end{equation*}

\begin{equation*}
\begin{array}{cccccc}
\chi & (4) & (3,1) & (2,2) & (2,1,1) & (1,1,1,1) \\
(4) & 1 & 1 & 1 & 1 & 1 \\
(3,1) & -1 & 0 & -1 & 1 & 3 \\
(2,2) & 0 & -1 & 2 & 0 & 2 \\
(2,1,1) & 1 & 0 & -1 & -1 & 3 \\
(1,1,1,1) & -1 & 1 & 1 & -1 & 1%
\end{array}%
\end{equation*}%
Combine the above formulas, we have the following tables of Brauer algebra
for small partitions
\begin{equation*}
\begin{array}{ccc}
\chi & (2) & (1,1) \\
(0) & 1 & 1%
\end{array}%
\end{equation*}

\begin{equation*}
\begin{array}{cccc}
\chi & (3) & (2,1) & (1,1,1) \\
(1) & 0 & 1 & 3%
\end{array}%
\end{equation*}

\begin{equation*}
\begin{array}{cccccc}
\chi & (4) & (3,1) & (2,2) & (2,1,1) & (1,1,1,1) \\
(2) & 0 & 0 & 2 & 2 & 6 \\
(1,1) & 0 & 0 & -2 & 0 & 6 \\
(0) & 1 & 0 & 3 & 1 & 3%
\end{array}%
\end{equation*}%
We thus obtain the following expressions
\begin{align*}
& pb_{(1)}=sb_{(1)} \\
& pb_{(2)}=sb_{(2)}-sb_{(1,1)}+1 \\
& pb_{(1,1)}=sb_{(2)}+sb_{(1,1)}+1 \\
& pb_{(3)}=sb_{(3)}-sb_{(2,1)}+sb_{(1,1,1)} \\
& pb_{(2,1)}=sb_{(3)}-sb_{(1,1,1)}+sb_{(1)} \\
& pb_{(1,1,1)}=sb_{(3)}+2sb_{(2,1)}+sb_{(1,1,1)}+3sb_{(1)} \\
& pb_{(4)}=sb_{(4)}-sb_{(3,1)}+sb_{(2,1,1)}-sb_{(1,1,1,1)}+1 \\
& pb_{(3,1)}=sb_{(4)}-sb_{(2,2)}+sb_{(1,1,1,1)} \\
&
pb_{(2,2)}=sb_{(4)}-sb_{(3,1)}+2sb_{(2,2)}-sb_{(2,,1,1)}+sb_{(1,1,1,1)}+2sb_{(2)}-2sb_{(1,1)}+3
\\
& pb_{(2,1,1)}=sb_{(4)}+sb_{(3,1)}-sb_{(2,1,1)}-sb_{(1,1,1,1)}+2sb_{(2)}+1 \\
&
pb_{(1,1,1,1)}=sb_{(4)}+3sb_{(3,1)}+2sb_{(2,2)}+3sb_{(2,1,1)}+sb_{(1,1,1,1)}+6sb_{(2)}+6sb_{(1,1)}+3,
\end{align*}%
and conversely
\begin{align*}
& sb_{(1)}=pb_{(1)} \\
& sb_{(2)}=\frac{1}{2}[pb_{(2)}+pb_{(1,1)}]-1 \\
& sb_{(1,1)}=\frac{1}{2}[-pb_{(2)}+pb_{(1,1)}] \\
& sb_{(3)}=\frac{1}{6}[2pb_{(3)}+3pb_{(2,1)}+pb_{(1,1,1)}]-pb_{(1)} \\
& sb_{(2,1)}=\frac{1}{3}[-pb_{(3)}+pb_{(1,1,1)}]-pb_{(1)} \\
& sb_{(1,1,1)}=\frac{1}{6}[2pb_{(3)}-3pb_{(2,1)}+pb_{(1,1,1)}] \\
& sb_{(4)}=\frac{1}{24}%
[6pb_{(4)}+8pb_{(3,1)}+3pb_{(2,2)}+6pb_{(2,1,1)}+pb_{(1,1,1,1)}]-\frac{1}{2}%
[pb_{(2)}+pb_{(1,1)}] \\
& sb_{(3,1)}=\frac{1}{8}%
[-2pb_{(4)}-pb_{(2,2)}+2pb_{(2,1,1)}+pb_{(1,1,1,1)}]-pb_{(1,1)}+1 \\
& sb_{(2,2)}=\frac{1}{12}[-4pb_{(3,1)}+3pb_{(2,2)}+pb_{(1,1,1,1)}]-\frac{1}{2%
}[pb_{(2)}+pb_{(1,1)}] \\
& sb_{(2,1,1)}=\frac{1}{8}%
[2pb_{(4)}-pb_{(2,2)}-2pb_{(2,1,1)}+pb_{(1,1,1,1)}]+\frac{1}{2}%
[pb_{(2)}-pb_{(1,1)}] \\
& sb_{(1,1,1,1)}=\frac{1}{24}%
[-6pb_{(4)}+8pb_{(3,1)}+3pb_{(2,2)}-6pb_{(2,1,1)}+pb_{(1,1,1,1)}].
\end{align*}

Some of the type-B Schur functions are listed as follows.
\begin{align*}
& sb_{(1)}=1+\frac{t-t^{-1}}{q-q^{-1}} \\
& sb_{(2)}=\left( 1+\frac{tq-t^{-1}q^{-1}}{q^{2}-q^{-2}}\right) \frac{%
t-t^{-1}}{q-q^{-1}} \\
& sb_{(1,1)}=\left( 1+\frac{tq^{-1}-t^{-1}q}{q^{2}-q^{-2}}\right) \frac{%
t-t^{-1}}{q-q^{-1}} \\
& sb_{(3)}=\left( 1+\frac{tq^{2}-t^{-1}q^{-2}}{q^{3}-q^{-3}}\right) \frac{%
t-t^{-1}}{q^{2}-q^{-2}}\frac{tq-t^{-1}q^{-1}}{q-q^{-1}} \\
& sb_{(2,1)}=\left( 1+\frac{t-t^{-1}}{q^{3}-q^{-3}}\right) \frac{%
tq^{-1}-t^{-1}q}{q-q^{-1}}\frac{tq-t^{-1}q^{-1}}{q-q^{-1}} \\
& sb_{(1,1,1)}=\left( 1+\frac{tq^{-2}-t^{-1}q^{2}}{q^{3}-q^{-3}}\right)
\frac{t-t^{-1}}{q^{2}-q^{-2}}\frac{tq^{-1}-t^{-1}q}{q-q^{-1}} \\
& sb_{(4)}=\left( 1+\frac{tq^{3}-t^{-1}q^{-3}}{q^{4}-q^{-4}}\right) \frac{%
t-t^{-1}}{q^{3}-q^{-3}}\frac{tq-t^{-1}q^{-1}}{q^{2}-q^{-2}}\frac{%
tq^{2}-t^{-1}q^{-2}}{q-q^{-1}} \\
& sb_{(3,1)}=\left( 1+\frac{tq-t^{-1}q^{-1}}{q^{4}-q^{-4}}\right) \frac{%
tq^{-1}-t^{-1}q}{q^{2}-q^{-2}}\frac{t-t^{-1}}{q-q^{-1}}\frac{%
tq^{2}-t^{-1}q^{-2}}{q-q^{-1}} \\
& sb_{(2,2)}=\left( 1+\frac{t-t^{-1}}{q^{3}-q^{-3}}\right) \left( 1+\frac{%
t-t^{-1}}{q-q^{-1}}\right) \frac{tq^{-2}-t^{-1}q^{2}}{q^{2}-q^{-2}}\frac{%
tq^{2}-t^{-1}q^{-2}}{q^{2}-q^{-2}} \\
& sb_{(2,1,1)}=\left( 1+\frac{tq^{-1}-t^{-1}q}{q^{4}-q^{-4}}\right) \frac{%
tq^{-2}-t^{-1}q^{2}}{q-q^{-1}}\frac{tq-t^{-1}q^{-1}}{q^{2}-q^{-2}}\frac{%
t-t^{-1}}{q-q^{-1}} \\
& sb_{(1,1,1,1)}=\left( 1+\frac{tq^{-3}-t^{-1}q^{3}}{q^{4}-q^{-4}}\right)
\frac{t-t^{-1}}{q^{3}-q^{-3}}\frac{tq^{-1}-t^{-1}q}{q^{2}-q^{-2}}\frac{%
tq^{-2}-t^{-1}q^{2}}{q-q^{-1}}
\end{align*}

\subsection{Colored Kauffman Polynomials of Torus Links/Knots and Tables of
Integer Coefficients $N_{\protect\overrightarrow{\protect\mu },g,\protect%
\beta }$}

The torus link $\mathcal{L}=T(rL,kL)$ has $L$ components if $(r,k)=1$. We
compute the orthogonal quantum group invariants by the following formula
proved in Theorem \ref{Thm3.6}.
\begin{equation}
W_{A}^{SO}(\mathcal{L};q,t)=q^{-kr\underset{\alpha =1}{\overset{L}{\sum }}%
\kappa _{A^{\alpha }}}t^{-k(r-1)n}\underset{f=0}{\overset{[\frac{rn}{2}]}{%
\sum }}\underset{\lambda \vdash rn-2f}{\sum }\widetilde{c}_{\overrightarrow{A%
}}^{\lambda }q^{\frac{k\kappa _{\lambda }}{r}}t^{-\frac{2fk}{r}}sb_{\lambda
}(q,t)  \label{E:10.3}
\end{equation}

The explicit formula for these type-B Schur functions $sb_{\lambda }(q,t)$
are computed in the above subsection.

Recall the definition of the constants $\widetilde{c}_{\overrightarrow{A}%
}^{\lambda }$ by the formula
\begin{equation*}
\underset{\alpha =1}{\overset{L}{{\prod }}}sb_{A^{\alpha }}(z^{r})=\underset{%
f=0}{\overset{[rn/2]}{\sum }}\underset{\lambda \vdash rn-2f}{\sum }%
\widetilde{c}_{\overrightarrow{A}}^{\lambda }sb_{\lambda }(z).
\end{equation*}%
In the case $r=2$ and $L=1$, we have
\begin{equation*}
\begin{array}{ccccccccc}
\widetilde{c} & (0) & (2) & (1,1) & (4) & (3,1) & (2,2) & (2,1,1) & (1,1,1,1)
\\
(1) & 1 & 1 & -1 & \times & \times & \times & \times & \times \\
(2) & 1 & 1 & -1 & 1 & -1 & 1 & 0 & 0 \\
(1,1) & 1 & 1 & -1 & 0 & 0 & 1 & -1 & 1%
\end{array}%
;
\end{equation*}%
in the case when $r=1$ and $L=2$, we have
\begin{equation*}
\begin{array}{ccccccccccccc}
\widetilde{c} & (0) & (1) & (2) & (1,1) & (3) & (2,1) & (1,1,1) & (4) & (3,1)
& (2,2) & (2,1,1) & (1,1,1,1) \\
(1),(1) & 1 & \times & 1 & 1 & \times & \times & \times & \times & \times &
\times & \times & \times \\
(2),(1) & \times & 1 & \times & \times & 1 & 1 & 0 & \times & \times & \times
& \times & \times \\
(1,1),(1) & \times & 1 & \times & \times & 0 & 1 & 1 & \times & \times &
\times & \times & \times \\
(2),(2) & 1 & \times & 1 & 1 & \times & \times & \times & 1 & 1 & 1 & 0 & 0
\\
(2),(1,1) & 0 & \times & 1 & 1 & \times & \times & \times & 0 & 1 & 0 & 1 & 0
\\
(1,1),(1,1) & 1 & \times & 1 & 1 & \times & \times & \times & 0 & 0 & 1 & 1
& 1 \\
(3),(1) & 0 & \times & 1 & 0 & \times & \times & \times & 1 & 1 & 0 & 0 & 0
\\
(2,1),(1) & 0 & \times & 1 & 1 & \times & \times & \times & 0 & 1 & 1 & 1 & 0
\\
(1,1,1),(1) & 0 & \times & 0 & 1 & \times & \times & \times & 0 & 0 & 0 & 1
& 1%
\end{array}%
;
\end{equation*}

in the case when $r=1$ and $L=3$, we have
\begin{equation*}
\begin{array}{ccccccccccccc}
\widetilde{c} & (0) & (1) & (2) & (1,1) & (3) & (2,1) & (1,1,1) & (4) & (3,1)
& (2,2) & (2,1,1) & (1,1,1,1) \\
(1),(1),(1) & \times & 3 & \times & \times & 1 & 2 & 1 & \times & \times &
\times & \times & \times \\
(2),(1),(1) & 1 & \times & 3 & 2 & \times & \times & \times & 1 & 2 & 1 & 1
& 0 \\
(1,1),(1),(1) & 1 & \times & 2 & 3 & \times & \times & \times & 0 & 1 & 1 & 2
& 1%
\end{array}%
\text{.}
\end{equation*}

In this subsection, we provide the tables for the values of the integers $N_{%
\overrightarrow{\mu },g,\beta }$ in the following formula

\begin{equation*}
\frac{\mathrm{z}_{\overrightarrow{\mu }}(q-q^{-1})^{2}\cdot \lbrack g_{%
\overrightarrow{\mu }}(q,t)-g_{\overrightarrow{\mu }}(q,-t)]}{2\overset{L}{%
\underset{\alpha =1}{\prod }}\overset{\ell (\mu ^{\alpha })}{\underset{i=1}{%
\prod }}(q^{\mu _{i}^{\alpha }}-q^{-\mu _{i}^{\alpha }})}=\sum_{g\in \mathbb{%
Z}_{+}/2}\sum_{\beta \in \mathbb{Z}}N_{\overrightarrow{\mu },g,\beta
}z^{2g}t^{\beta }.
\end{equation*}

Example 1: Take $r=1$, the torus link $T(2,2k)$ has $2$ components. By (\ref%
{E:10.3}), we have
\begin{align*}
& W_{(1)(1)}=q^{2k}sb_{(2)}+q^{-2k}sb_{(1,1)}+t^{-2k} \\
& W_{(2)(1)}=q^{4k}sb_{(3)}+q^{-2k}sb_{(2,1)}+q^{-2k}t^{-2k}sb_{(1)} \\
& W_{(1,1)(1)}=q^{2k}sb_{(2,1)}+q^{-4k}sb_{(1,1,1)}+q^{2k}t^{-2k}sb_{(1)} \\
&
W_{(2)(2)}=q^{8k}sb_{(4)}+sb_{(3,1)}+q^{-4k}sb_{(2,2)}+q^{-2k}t^{-2k}sb_{(2)}+q^{-6k}t^{-2k}sb_{(1,1)}+q^{-4k}t^{-4k}
\\
&
W_{(2)(1,1)}=q^{4k}sb_{(3,1)}+q^{-4k}sb_{(2,1,1)}+q^{2k}t^{-2k}sb_{(2)}+q^{-2k}t^{-2k}sb_{(1,1)}
\\
&
W_{(1,1)(1,1)}=q^{4k}sb_{(2,2)}+sb_{(2,1,1)}+q^{-8k}sb_{(1,1,1,1)}+q^{6k}t^{-2k}sb_{(2)}+q^{2k}t^{-2k}sb_{(1,1)}+q^{4k}t^{-4k}
\end{align*}

\begin{align*}
& W_{(3)(1)}=q^{6k}sb_{(4)}+q^{-2k}sb_{(3,1)}+q^{-4k}t^{-2k}sb_{(2)} \\
&
W_{(2,1)(1)}=q^{4k}sb_{(3,1)}+sb_{(2,2)}+q^{-4k}sb_{(2,1,1)}+q^{2k}t^{-2k}sb_{(2)}+q^{-2k}t^{-2k}sb_{(1,1)}
\\
&
W_{(1,1,1)(1)}=q^{2k}sb_{(2,1,1)}+q^{-6k}sb_{(1,1,1,1)}+q^{4k}t^{-2k}sb_{(1,1)}
\end{align*}

Case A: Torus link $T(2,2k)$ with partitions $(1),(1)$

$N_{\overrightarrow{\mu },g,\beta }=0$.

Case B: Torus link $T(2,2k)$ with partitions $(1,1),(1)$

\begin{equation*}
\begin{array}{ccccc}
k=1 & \beta=-3 & -1 & 1 & 3 \\
g=0 & -1 & 3 & -3 & 1%
\end{array}%
\end{equation*}

\begin{equation*}
\begin{array}{cccccc}
k=2 & \beta=-5 & -3 & -1 & 1 & 3 \\
g=0 & -4 & 4 & 12 & -20 & 8 \\
1 & -1 & 1 & 3 & -9 & 6 \\
2 & 0 & 0 & 0 & -1 & 1%
\end{array}%
\end{equation*}

\begin{equation*}
\begin{array}{ccccccc}
k=3 & \beta=-7 & -5 & -3 & -1 & 1 & 3 \\
g=0 & -9 & 9 & -3 & 45 & -72 & 30 \\
1 & -6 & 6 & -1 & 39 & -93 & 55 \\
2 & -1 & 1 & 0 & 11 & -47 & 36 \\
3 & 0 & 0 & 0 & 1 & -11 & 10 \\
4 & 0 & 0 & 0 & 0 & -1 & 1%
\end{array}%
\end{equation*}

Case C: Torus link $T(2,2k)$ partitions $(2),(1)$
\begin{equation*}
\begin{array}{ccccc}
k=1 & \beta =-3 & -1 & 1 & 3 \\
g=0 & 1 & -1 & -1 & 1%
\end{array}%
\end{equation*}

\begin{equation*}
\begin{array}{cccccc}
k=2 & \beta=-5 & -3 & -1 & 1 & 3 \\
g=0 & 2 & -2 & 2 & -6 & 4 \\
1 & 1 & -1 & 1 & -5 & 4 \\
2 & 0 & 0 & 0 & -1 & 1%
\end{array}%
\end{equation*}

\begin{equation*}
\begin{array}{ccccccc}
k=3 & \beta=-7 & -5 & -3 & -1 & 1 & 3 \\
g=0 & 3 & -3 & -1 & 9 & -18 & 10 \\
1 & 4 & -4 & -1 & 15 & -39 & 25 \\
2 & 1 & -1 & 0 & 7 & -29 & 22 \\
3 & 0 & 0 & 0 & 1 & -9 & 8 \\
4 & 0 & 0 & 0 & 0 & -1 & 1%
\end{array}%
\end{equation*}

Case D: Torus link $T(2,2k)$ with partitions $(2),(2)$
\begin{equation*}
\begin{array}{ccccc}
k=1 & \beta =-3 & -1 & 1 & 3 \\
g=1/2 & -2 & 2 & -2 & 2 \\
3/2 & -1 & 1 & -1 & 1%
\end{array}%
\end{equation*}

\begin{equation*}
\begin{array}{cccccc}
k=2 & \beta=-5 & -3 & -1 & 1 & 3 \\
g=1/2 & -8 & 4 & 20 & -36 & 20 \\
3/2 & -24 & 20 & 40 & -96 & 60 \\
5/2 & -22 & 21 & 29 & -97 & 69 \\
7/2 & -8 & 8 & 9 & -47 & 38 \\
9/2 & -1 & 1 & 1 & -11 & 10 \\
11/2 & 0 & 0 & 0 & -1 & 1%
\end{array}%
\end{equation*}

Case E: Torus link $T(2,2k)$ with partitions $(3),(1)$
\begin{equation*}
\begin{array}{ccccc}
k=1 & \beta =-3 & -1 & 1 & 3 \\
g=1/2 & -1 & 0 & 0 & 1%
\end{array}%
\end{equation*}

\begin{equation*}
\begin{array}{cccccc}
k=2 & \beta =-5 & -3 & -1 & 1 & 3 \\
g=1/2 & -4 & 4 & 0 & -8 & 8 \\
3/2 & -5 & 5 & 0 & -14 & 14 \\
5/2 & -1 & 1 & 0 & -7 & 7 \\
7/2 & 0 & 0 & 0 & -1 & 1%
\end{array}%
\end{equation*}

Example 2: The torus knots $T(2,k),$ where $k$ is an odd integer. Again we
compute the following quantum group invariants by (\ref{E:10.3}).
\begin{align*}
& W_{(1)}=t^{-k}(q^{k}sb_{(2)}(q,t)-q^{-k}sb_{(1,1)}(q,t)+t^{-k}) \\
&
W_{(2)}=t^{-2k}(q^{2k}sb_{(4)}-q^{-2k}sb_{(3,1)}+q^{-4k}sb_{(2,2)}+q^{-3k}t^{-k}sb_{(2)}-q^{-5k}t^{-k}sb_{(1,1)}+q^{-4k}t^{-2k})
\\
&
W_{(1,1)}=t^{-2k}(q^{4k}sb_{(2,2)}-q^{2k}sb_{(2,1,1)}+q^{-2k}sb_{(1,1,1,1)}+q^{5k}t^{-k}sb_{(2)}-q^{3k}t^{-k}sb_{(1,1)}+q^{4k}t^{-2k})
\end{align*}

Case A: Torus knot $T(2,k)$ with partition $(1,1)$

\begin{equation*}
\begin{array}{cccccc}
k=3 & \beta =-11 & -9 & -7 & -5 & -3 \\
g=1/2 & 36 & -132 & 180 & -108 & 24 \\
3/2 & 105 & -377 & 453 & -207 & 26 \\
5/2 & 112 & -450 & 494 & -165 & 9 \\
7/2 & 54 & -275 & 286 & -66 & 1 \\
9/2 & 12 & -90 & 91 & -13 & 0 \\
11/2 & 1 & -15 & 15 & -1 & 0 \\
13/2 & 0 & -1 & 1 & 0 & 0%
\end{array}%
\end{equation*}

Case B: Torus knot $T(2,k)$ with partition $(2)$

\begin{equation*}
\begin{array}{cccccc}
k=3 & \beta =-11 & -9 & -7 & -5 & -3 \\
g=1/2 & -6 & 26 & -42 & 30 & -8 \\
3/2 & -35 & 125 & -161 & 85 & -14 \\
5/2 & -56 & 210 & -238 & 91 & -7 \\
7/2 & -36 & 165 & -174 & 46 & -1 \\
9/2 & -10 & 66 & -67 & 11 & 0 \\
11/2 & -1 & 13 & -13 & 1 & 0 \\
13/2 & 0 & 1 & -1 & 0 & 0%
\end{array}%
\end{equation*}

Example 3: Take $r=1$, the torus link $T(3,3k)$ has $3$ components. By (\ref%
{E:10.3}), We have

\begin{eqnarray*}
W_{(1),(1),(1)}
&=&q^{6k}sb_{(3)}+2sb_{(2,1)}+q^{-6k}sb_{(1,1,1)}+3t^{-2k}sb_{(1)} \\
W_{(2),(1),(1)}
&=&q^{10k}sb_{(4)}+2q^{2k}sb_{(3,1)}+q^{-2k}sb_{(2,2)}+q^{-6k}sb_{(2,1,1)} \\
&&+3t^{-2k}sb_{(2)}+2q^{-4k}t^{-2k}sb_{(1,1)}+q^{-2k}t^{-4k} \\
W_{(1,1),(1),(1)}
&=&q^{6k}sb_{(3,1)}+q^{2k}sb_{(2,2)}+2q^{-2k}sb_{(2,1,1)}+q^{-10k}sb_{(1,1,1,1)}
\\
&&+2q^{4k}t^{-2k}sb_{(2)}+3t^{-2k}sb_{(1,1)}+q^{2k}t^{-4k}
\end{eqnarray*}

Torus link $T(3,3k)$ with partitions $(2),(1),(1)$

\begin{equation*}
\begin{array}{ccccc}
k=1 & \beta =-3 & -1 & 1 & 3 \\
g=1/2 & 2 & -2 & -10 & 10 \\
3/2 & 0 & 0 & -6 & 6 \\
5/2 & 0 & 0 & -1 & 1%
\end{array}%
\end{equation*}

\begin{equation*}
\begin{array}{cccccc}
k=2 & \beta =-5 & -3 & -1 & 1 & 3 \\
g=1/2 & 16 & -48 & 176 & -336 & 192 \\
3/2 & 12 & -68 & 452 & -1036 & 640 \\
5/2 & 2 & -38 & 494 & -1406 & 948 \\
7/2 & 0 & -10 & 286 & -1056 & 780 \\
9/2 & 0 & -1 & 91 & -467 & 377 \\
11/2 & 0 & 0 & 15 & -121 & 106 \\
13/2 & 0 & 0 & 1 & -17 & 16 \\
15/2 & 0 & 0 & 0 & -1 & 1%
\end{array}%
\end{equation*}

\end{document}